\documentclass[12pt]{amsart}
\usepackage{amsmath,amstext,amsfonts,amsthm,amssymb,epigraph}
\usepackage{eucal}
\usepackage{framed}
\usepackage{xcolor}
\usepackage{bbm}
\usepackage[all]{xy}                                        %
\swapnumbers
\newtheorem{thm}[subsubsection]{Theorem}
\newtheorem{lem}[subsubsection]{Lemma}
\newtheorem{prp}[subsubsection]{Proposition}
\newtheorem{crl}[subsubsection]{Corollary}

\newtheorem{PRP}[subsection]{Proposition}

\newtheorem*{Thm}{Theorem}
\newtheorem*{Lem}{Lemma}

{  \theoremstyle{definition}
           \newtheorem{dfn}[subsubsection]{Definition}
           
           \newtheorem{rem}[subsubsection]{Remark}

           \newtheorem*{Dfn}{Definition}
           
           \newtheorem*{Rem}{Remark}
           
}

\newcommand{\ad}{\mathrm{ad}}
\newcommand{\an}{\mathit{an}}
\newcommand{\nr}{\mathit{nr}}

\newcommand{\Aut}{\operatorname{Aut}}

\newcommand{\codim}{\mathrm{codim}}

\newcommand{\fl}{\mathrm{fl}}

\newcommand{\GL}{\mathrm{GL}}
\newcommand{\Gp}{\mathtt{Gp}}

\newcommand{\Hom}{\mathrm{Hom}} 

\newcommand{\htt}{\operatorname{ht}}
\newcommand{\id}{\mathrm{id}}
\newcommand{\im}{\mathit{im}}
\renewcommand{\int}{\mathit{int}}

\newcommand{\iso}{\mathit{iso}}
\newcommand{\Ker}{\mathrm{Ker}}
\newcommand{\CG}{\mathtt{C}}

\mathchardef\mhyphen="2D

\newcommand{\Ob}{\operatorname{Ob}}
\newcommand{\ol}{\overline}

\newcommand{\re}{\mathit{re}}
\newcommand{\rk}{\operatorname{rk}}
\newcommand{\Sk}{\mathtt{Sk}} 
\newcommand{\Sp}{\mathtt{Sp}} 

\newcommand{\Span}{\operatorname{Span}}

\newcommand{\U}{\mathtt{U}}
\newcommand{\Vect}{\mathtt{Vect}}

\newcommand{\wt}{\widetilde}

\newcommand{\bC}{\mathbb{C}}

\newcommand{\bN}{\mathbb{N}}

\newcommand{\bR}{\mathbb{R}}

\newcommand{\bZ}{\mathbb{Z}}

\newcommand{\cF}{\mathcal{F}}
\newcommand{\cG}{\mathcal{G}}
\newcommand{\cH}{\mathcal{H}}

\newcommand{\cL}{\mathcal{L}}

\newcommand{\cP}{\mathcal{P}}
\newcommand{\cQ}{\mathcal{Q}}
\newcommand{\cR}{\mathcal{R}}

\newcommand{\cT}{\mathcal{T}}
\newcommand{\cU}{\mathcal{U}}

\newcommand{\fb}{\mathfrak{b}}

\newcommand{\fg}{\mathfrak{g}}
\newcommand{\fh}{\mathfrak{h}}
\newcommand{\fk}{\mathfrak{k}}
\renewcommand{\fl}{\mathfrak{l}}

\newcommand{\fn}{\mathfrak{n}}
\newcommand{\fp}{\mathfrak{p}}
\newcommand{\fr}{\mathfrak{r}}
\newcommand{\fs}{\mathfrak{s}}
\newcommand{\ft}{\mathfrak{t}}
\newcommand{\fa}{\mathfrak{a}}

\newcommand{\Z}{\mathbb{Z}}

\newcommand{\fsl}{\mathfrak{sl}}
\newcommand{\fgl}{\mathfrak{gl}}

\newcommand{\fosp}{\mathfrak{osp}}

\newcommand{\fq}{\mathfrak q}

\begin{document}

\title[]{Root groupoid and related Lie superalgebras}
\author{M.~Gorelik}
\address{Weizmann Institute of Science}
\email{maria.gorelik@gmail.com}
\author{V.~Hinich}
\address{University of Haifa}
\email{vhinich@gmail.com}
\author{V. Serganova}
\address{UC Berkeley}
\email{serganov@math.berkeley.edu}
 
\begin{abstract}

 We introduce a notion of a root groupoid as a replacement of the notion of Weyl group for (Kac-Moody) Lie superalgebras. The objects of the root 
groupoid classify certain root data, the arrows are defined by generators and relations. As an abstract groupoid the root groupoid has many connected components and we show that to some of them one can associate 
an interesting family of Lie superalgebras which we call root
superalgebras. We classify root superalgebras satisfying some additional assumptions.
  To each root groupoid component we associate a graph (called skeleton) generalizing the Cayley graph of the Weyl group. We establish the
  Coxeter property of the skeleton generalizing in this way the fact that the Weyl group of a Kac-Moody Lie algebra is Coxeter.  
\end{abstract}
\maketitle
\section{Introduction}

\subsection{Generalities}
\subsubsection{}

In this paper we present an attempt to generalize the notion of Weyl group 
to Lie superalgebras. For a semisimple Lie algebra Weyl group parametrizes 
Borel subalgebras containing a fixed torus. This cannot be directly extended 
to  Lie superalgebras since there are essentially different choices of Borel 
subalgebras.
In order to describe all Borel subalgebras, the notion of odd (or isotropic) reflection was introduced many years ago, \cite{S2},~\cite{P},~\cite{DP}.
An odd reflection can not be naturally extended to 
an automorphism of the Lie superalgebra. For many years a strong feeling 
persisted among the experts  that one should extend Weyl group 
to ``Weyl groupoid''.
One  attempt was made in \cite{S3}. A somewhat reminiscent construction
of groupoid was suggested by I.~Heckenberger and collaborators ~\cite{H},
\cite{HY}, see also~\cite{AA}. In~\ref{ss:comment} we comment on these definitions. 
More recently another notion named Weyl 
groupoid was introduced by A.~Sergeev and A.~Veselov~\cite{SV} for
finite-dimensional superalgebras in order to describe the character ring
(but we do not see a connection with our notion).

The notion of root groupoid presented in this paper is close to the one
defined in~\cite{S3}.

\subsubsection{}The connection between semisimple Lie algebras and root systems can be described from two opposite perspectives.
One can start with a Lie algebra, choose a Cartan subalgebra and study the geometry of the set of roots. On the other hand, one can start with a Cartan matrix
and construct a Lie algebra by generators and relations. The second approach was vastly extended to construct a  very important family of infinite-dimensional Lie algebras
by Kac, Moody, Borcherds and others. Our approach follows the same pattern for construction of Lie superalgebras from combinatorial data.
\subsubsection{}
The standard construction of a Kac-Moody Lie algebra comprises two steps.
In the first step one defines a huge graded Lie algebra factoring a free
Lie algebra by so-called Chevalley relations (we call this algebra
{\sl half-backed algebra}). In the second step one factors the half-backed
algebra by the unique maximal ideal having trivial intersection with the 
Cartan subalgebra (this is Victor Kac's approach). Alternatively
(Robert Moody) one imposes an explicit set of Serre relations. Both approaches yield the same result in the symmetrizable case. 

One of the motivations of our work came from the observation that the 
classical Serre relation can be interpreted as reflected Chevalley 
relations. This led 
us to the notion of a root algebra which, roughly speaking, respects the 
symmetries determined by a root groupoid. In many cases there is a unique 
root algebra which can be defined by Chevalley relations reflected in all 
possible ways. Sometimes there is a number of root algebras defined by a 
given root datum. The description of all root algebras is
an open question --- we don't know the answer even for Lie algebras.
For finite dimensional and affine Lie superalgebras all Serre's relations were described in \cite{Y}, see also~\cite{Zh}. One can see from this description that
Serre's relations may involve more than two generators.

\subsection{Root groupoid}

In \cite{Kbook}, 1.1, Kac defines a realization of a Cartan matrix $A=(a_{xy})$, $x,y\in X$, as a triple
$(\fh, a,b)$ such that $a=\{a(x)\in\fh\}$, $b=\{b(x)\in\fh^*\}$
and $\langle a(x),b(x)\rangle=a_{xy}$. Adapting this definition to Lie superalgebras,
we add the parity function $p:X\to\bZ_2$ on the index set $X$ and make a quadruple 
$v=(\fh,a,b,p)$ an
object of {\sl the root groupoid $\cR$} --- the main object of our study. Every quadruple $v$ defines a Cartan matrix by the formula above. 
The pair $(A,p)$ is called in this paper {\sl Cartan datum}. There are three kinds of generators
in the set of arrows in $\cR$. Two of them are quite dull: one (a homothety) rescales 
$a(x)$, another (tautological arrow) is defined by an isomorphism $\theta:\fh\to\fh'$;
the third kind are {\sl reflexions} that retain the same vector space $\fh$ but change the collections $a(x)$ and $b(x)$ by usual (even or odd) reflection formulas.
Each generator $f:v\to v'$ defines a linear transformation $\fh_v\to\fh_{v'}$
(it is identity for homotheties and reflexions, and $\theta$ for the tautological arrow 
defined by $\theta$); two compositions of generators leading from $v$ to $v'$ are equivalent 
if they define the same linear transformation.  The root groupoid $\cR$ has a lot of 
components, some of them, most probably, useless. However, some connected components 
(we call them admissible) lead to interesting Lie superalgebras. It is worth mentioning that
Cartan matrices $A_v$ are different even inside one component: one type of reflexions, 
{\sl isotropic reflexions}, modify Cartan matrices in a certain way (see the formulas in
\ref{sss:cartanmatrix-change}).

\subsection{Root algebras}

For each $v\in\cR$ one defines a (huge) Lie superalgebra $\wt\fg(v)$ (we call it
half-baked Lie superalgebra) in the same way as did  V.~Kac and R.~Moody, see~\ref{sss:half}. For an arrow $\gamma:v\to v'$ in $\cR$ the 
isomorphism
$\fh(\gamma):\fh(v)\to \fh(v')$ does not extend to a homomorphism of the half-baked algebras.
We define a root algebra supported on a component $\cR_0$ of $\cR$ as a collection of 
quotients $\fg(v)$ of $\wt\fg(v)$ such that for any $\gamma:v\to v'$ the isomorphism 
$\fh(\gamma)$ extends to an isomorphism $\fg(v)\to\fg(v')$.

A component $\cR_0$ of $\cR$ is called admissible if it admits a root algebra. Admissibility
can be expressed in terms of weak symmetricity of the Cartan matrices at $\cR_0$, see
Theorem~\ref{thm:admissible=wsym}.

For an admissible component $\cR_0$ there always exists an initial and a final object
in the category of root algebras. The initial root algebra $\fg^\U$ is called {\sl universal}. The final root algebra $\fg^\CG$ is called {\sl contragredient}. Note that $\fg^\CG$ in the admissible case is defined as the quotient of $\wt\fg(v)$ by the maximal
ideal having zero intersection with $\fh$. The universal root algebra $\fg^\U$ is obtained by imposing on $\wt\fg(v)$ reflected Chevalley relations --- so generalizing the classical Serre relations.

{Note that these were two different approaches of the founding fathers of
Kac-Moody Lie algebras: Victor Kac~\cite{Kbook} factored the half-baked 
algebra by the maximal ideal having zero intersection with the Cartan 
subalgebra, whereas Robert Moody~\cite{M} imposed on it the Serre relations.}

\subsection{Graphs associated to the root groupoid}

Define $\Sk\subset\cR$ (skeleton) as the subgroupoid whose arrows are
the compositions  of reflexions.
Denote by $\Sk(v)$ the connected component of $v\in\Sk$.
This is a contractible groupoid; it makes sense to study it as a marked graph, whose edges 
are reflexions marked by the elements of the index set $X$.

We define the spine $\Sp$ as the subgroupoid of $\Sk$ whose arrows are the compositions of isotropic reflexions only. For instance, if there are no isotropic reflexions (for example,  if $p(x)=0$ for all $x$) then $\Sp$ has no arrows.
The connected component of $v$ in $\Sp$ is denoted by $\Sp(v)$.

We show that the vertices of $\Sk(v)$ parametrize Borel subalgebras
of a root algebra $\fg(v)$ that can be obtained by a chain of reflexions
from the original Borel subalgebra $\fb(v)$, see~\ref{rem:attainable}
and~\ref{crl:unique-in-sk}. Similarly, the vertices of $\Sp(v)$
parametrize Borel subalgebras of a root algebra $\fg(v)$ whose even part
coincides with that of $\fb(v)$, see~\ref{crl:Spiff}.

In the classical case of Kac-Moody Lie algebras $\Sk(v)$ is the Cayley
graph of the Weyl group and $\Sp(v)=\{v\}$.

The Weyl group $W$ acts on $\Sk(v)$ and each $W$-orbit
has unique element in $\Sp(v)$, see~\ref{crl:Spiff} for details.

\subsection{Groups associated to the root groupoid}

An only algebraic invariant of an abstract connected groupoid is the automorphism group of its
object. The group $\Aut_\cR(v)$ is one of a plethora of groups we assign to an admissible 
component $\cR_0$. It acts (up to a torus) on any root Lie algebra and on the set of its
roots. For the component corresponding to a  semisimple Lie algebra, $\Aut_\cR(v)$ coincides with the Weyl group.
In the case of conventional Kac-Moody Lie algebras $\Aut_\cR(v)$ is the product
of the Weyl group and a certain group of ``irrelevant'' automorphisms. 
The group of irrelevant automorphism $K(v)$ is very easy to describe. This is a subgroup
of automorphisms $\theta\in\GL(\fh(v))$ preserving all $b(x)\in\fh^*$ as well as 
all $a(x)$ up to constant. It is a unipotent abelian group in the case of Kac-Moody algebras. The equality $\Aut_\cR(v)=W\times K$ does not hold already for $\fgl(1|1)$,
see~\ref{sss:aut-gl11}.

\subsubsection{Skeleton subgroup} We will now present a combinatorial description of the quotient group $\Aut_\cR(v)/K(v)$.   We denote by $\Sk^D(v)$ the 
subset of vertices in $\Sk(v)$ having a Cartan datum $D$-equivalent 
to $A_v$, see~\ref{dfn:Deq}. The set $\Sk^D(v)$ has a group structure and Proposition~\ref{prp:structure-Aut}(3) claims that there is an isomorphism
$\Aut_\cR(v)/K(v)=\Sk^D(v)$.

\subsubsection{Weyl group}

For a vertex $v$ in an admissible $\cR_0$ we define Weyl group $W(v)$ (up to isomorphism, it depends on the component only) as a certain subgroup of $\GL(\fh)$ generated by reflections (more precisely, by the reflections with respect to anisotropic roots, see~\ref{ss:weyl}). The Weyl group $W(v)$ is a normal subgroup of $\Aut_\cR(v)$.

\subsubsection{Spine subgroup}

The intersection
$\Sp^D(v)=\Sp(v)\cap\Sk^D(v)$ is a subgroup in $\Sk^D(v)$. Proposition~\ref{prp:structure-Aut}
claims that $\Aut_\cR(v)/K(v)=\Sk^D(v)$ is a semidirect product $W(v)\rtimes\Sp^D(v)$.
In particular, if $\Sp^D$ is trivial, this gives $\Aut(v)=W\times K$.

\subsection{Coxeter properties}

A fundamental property of Kac-Moody Lie algebras is that its Weyl group
is a Coxeter group. We generalize this result to the Weyl groups appearing in any admissible component. Similarly to the classical result, the length of an element $w\in W$ can be expressed as the number of positive anisotropic roots that become negative under $w$,
see Corollary~\ref{crl:W-len}.

An analog of Coxeter property holds also for the skeleton $\Sk(v)$. The length of the
shortest path from $v$ to $v'$ in $\Sk(v)$ can also be expressed as the number of real positive roots that become negative, see~\ref{prp:Sk-len}.

Coxeter property for groups provides its presentation in terms of generators and relations,
with relations defined by ``pairwise interaction'' of the generators.
It turns out that a similar presentation exists for the skeleton.
In Section~\ref{sec:coxeter2} we define the notion of Coxeter graph
that generalizes that of Coxeter group, and prove that the skeleton
$\Sk(v)$ is a Coxeter graph.

\subsection{Fully reflectable components}
\label{ss:fully}
Admissible Cartan matrices are not in general required to allow reflexions $r_x$ for all $x\in X$. We call a component $\cR_0$ fully reflectable if all reflexions are allowed
at all vertices of $\cR_0$. This means that $\Sk(v)$ is a regular graph of degree $|X|$.
In Section~\ref{sec:trichotomy} we divide all admissible indecomposable fully reflectable components into three types:
finite, affine and indefinite. This trichotomy extends the similar trichotomy for Kac-Moody Lie algebras. There is a full classification of those types that contain an isotropic
root; it has been done by C.~Hoyt and V. Serganova, see~\cite{Hoyt}, \cite{S3}. Curiously, there are only two new indefinite series having an isotropic root; they are called $Q^\pm(m,n,k)$.

\subsection{On the (lack of) uniqueness of a root Lie superalgebra} We have already mentioned that, for an admissible component $\cR_0$ there is an initial $\fg^\U$ and a final $\fg^\CG$ root algebra supported at $\cR_0$. The natural map $\fg^\U\to\fg^\CG$ is surjective and all root algebras are factors lying in between. In Sections~\ref{sect:sym}
and \ref{sect:aff}
we study the gap between $\fg^\U$ and $\fg^\CG$ in the fully reflectable case.
The result of these sections can be summarized as follows.

\begin{Thm}
Let $\cR_0$ be an admissible indecomposable fully reflectable component. Then
$\fg^\U=\fg^\CG$ except for the cases  $\fg^\CG=\fg\fl(1|1)$, $\fg^\U=\fsl(n|n)^{(i)}$, 
$(i=1,2,4)$, $\fs\fq(n)^{(2)}$ and the case 
when $\cR_0$ is indefinite and nonsymmetrizable.
\end{Thm}
The similar result for symmetrizable Kac-Moody Lie algebras was proven by Gabber-Kac,
\cite{GabberKac}. Their proof  was adapted to our symmetrizable   case in Section~\ref{sect:sym}.
In the case when $\fg^\CG=\fg\fl(1|1)$ the algebra $\fg^\U$ has dimension $(4|2)$
and the algebras $\fg^\U$ and $\fg^\CG$ are the only two root algebras in this component,
see~\ref{rank1}.
Note that the explicit realization of $\fg^\CG$ for $\fg^\U=\fsl(n|n)^{(i)}$, $(i=1,2,4)$
and  $\fs\fq(n)^{(2)}$ is given in~\cite{S3}.

The results for nonsymmetrizable affine algebras, $S(2,1,b)$ and $\fs\fq(n)^{(2)}$ are new.

We also prove that if  $\fg^\CG\neq\mathfrak{gl}(1|1)$ then  any algebra $\fg$
sandwiched between $\fg^{\CG}$ and $\fg^{\U}$ is a root algebra.

\subsection{Examples of calculation of $\Aut(v)$}

In the last Section~\ref{sec:app} we compute the skeleton and the group 
$\Aut_\cR(v)$ for two classes of connected
components.

The first one is the case of a ``star-shaped'' spine. It includes the algebras 
$\fs\fq(3)^{(2)}$,  $B(1|1)^{(1)}$, $D(2|1,a)$, $D(2|1,a)^{(1)}$, $Q^{\pm}(m,n,k)$.
Here one has $\Aut_\cR(v)=W\times K$ as in this case $\Sp^D(v)$ is trivial.
For the same reason $\Aut_\cR(v)=W$  for all finite dimensional Lie superalgebras except for 
the case of $\fgl(n|n)$; the latter  is considered in~\ref{sss:glmn}.
The second class is the class of components whose skeleton identifies with that of
$\fsl_n^{(1)}$. This includes the root data for $\fsl(k|l)^{(1)}$, $\fs\fq(n)^{(2)}$
and  $S(2|1,b)$. In these cases the  Weyl group  $W(\fsl_n^{(1)})$ acts simply transitively
on the skeleta $\Sk(v)$. This allows one to realize the Weyl group and $\Sk^D(v)=\Aut_\cR(v)/K(v)$ as subgroups of $W(\fsl_n^{(1)})$.

\subsection{Borcherds-Kac-Moody algebras}
R.~Borcherds in \cite{Bo} introduced a generalization of Kac-Moody algebras, 
 where the Cartan matrix is real symmetric and satisfies additional
 conditions. The proof of Gabber-Kac is valid for this class
(see~\cite{Kbook}, 11.13) and give $\fg^{\CG}=\fg^\U$
if the Cartan matrix is symmetrizable and satisfies the conditions
(C1')--(C3') in ~\cite{Kbook}, 11.13). 
Borcherds-Kac-Moody (BKM) superalgebras were studied by M.~Wakimoto in~\cite{W}. Note that any symmetrizable Kac-Moody algebra is a BKM algebra,
but many symmetrizable Kac-Moody superalgebras (including $\fgl(m|n)$ for $m,n>2$)
are not BKM superalgebras. 
BKM superalgebras are  root superalgebras.

\subsection{Comment on the groupoids studied in~\cite{HY,AA}}
\label{ss:comment}

In~\cite{HY} the authors assign a groupoid (called Coxeter groupoid)
to a collection of vectors in a vector space $\fh^*$ endowed with a 
nondegenerate symmetric bilinear form.
The objects of "Coxeter groupoid" appearing in the definition
in~\cite{HY} correspond to different choices of  Borel
subalgebras of a Kac-Moody superalgebra given by a symmetrizable Cartan datum;
thus, they correspond to the vertices of what we call a skeleton
component. The arrows are generated by reflections with respect to all
simple roots. Our result claiming coxeterity of the skeleton
(Theorem~\ref{thm:skeleton-coxeter}) means that the groupoid defined
in~\cite{HY} is contractible. For instance, it assigns to a semisimple
Lie algebra the Cayley graph of its Weyl group (which is contractible
when considered as a groupoid). In order to get for a semisimple Lie
algebra the classifying groupoid of the Weyl group instead of the
contractible Cayley graph, one has to identify the vertices having the
same Cartan matrix as it is done in~\cite{AA}.

In the present paper we do something similar to~\cite{AA}, however,
instead of identifying equivalent vertices, we add
isomorphisms between them.

In our root  
groupoid we have generators for the arrows of three different types: 
apart from reflexions, we have tautological arrows and homotheties that 
connect vertices with D-equivalent Cartan matrices. In absence of 
isotropic reflexions (for instance for Kac-Moody algebras) the Weyl group
coincides with the automorphism group of an object of the
corresponding component of a root groupoid.  
In general, these two groups are different, see~\ref{prp:structure-Aut} and Section~\ref{sec:app}.

\subsection{}

Let us highlight the most important results of the paper.
\begin{itemize}
\item A criterion of admissibility in terms of weak symmetricity,
see~\ref{thm:admissible=wsym}.
\item The Weyl group is defined uniformly, see~\ref{ss:weyl}.
\item A general proof of the fact that the Weyl group is a Coxeter
group generated by the principal reflections, see~\ref{ss:Wcoxeter}.
\item A similar Coxeter property for the skeleton, 
see~\ref{thm:skeleton-coxeter}.
\item Root Lie superalgebras are classified in the finite and 
affine case, see~\ref{rank1}, \ref{crlfin}, \ref{thm:UKM-ns}
and \ref{rootalg}.
\end{itemize}

Apart from theorems proven, one of the important aims of the paper is to develop a useful language to deal with Lie superalgebras. Many parts of this language (isotropic (odd) reflexions, 
skeleton as the set of Borels, spine as the set of Borels with a fixed even part, Weyl group as the Weyl group of the even part of the Lie  superalgebra (slightly modified as the latter is not Kac-Moody)) were in use for a long time, but we believe our general setup gives a way to look at it uniformly.

\subsection{Acknowledgements}
V.S. enjoyed numerous visits to Weizmann Institute whose pleasant atmosphere is
gratefully acknowledged.
We are grateful to J.~Bernstein whose comment initiated the project
and to R.~Moody whose question triggered our Sections \ref{sec:coxeter} 
and \ref{sec:coxeter2} devoted to Coxeter properies of the root groupoid.
We are also grateful to V.~Kac and A.~Sherman for valuable comments and 
to L.~D.~Silberberg and S.~K.~Kerbis for locating numerous mistakes in 
the earlier versions of the manuscript. We thank the anonymous referees 
for helpful suggestions.

The work of V.H. was supported by ISF 786/19 grant.
M.G was supported by ISF 1957/21 grant.
The work of V.S. was supported by NSF grant 2001191.

\section{Setup}

\subsection{Groupoid of root data}
\label{ss:RDG}
Recall that a groupoid is a category in which all arrows are invertible.

Once  and forever we fix a finite set $X$. The cardinality of $X$
will be called {\sl the rank} of root data and of Lie superalgebras
connected to them.

For a complex vector space $\fh$ and a  set  $X$, a map $a:X\to\fh$
will be called injective if the induced map $\Span_\bC(X)\to \fh$ is an injective map of vector spaces.

\subsubsection{}
We now define {\sl the root groupoid} $\cR$.

The objects of $\cR$ (the root data) are the quadruples 
$(\fh, a:X\to \fh, b:X\to \fh^*, p:X\to\mathbb{Z}_2)$ where 
$\fh$ is a finite dimensional vector space over $\bC$
  such that $a,b$ are injective.

We will define the arrows of $\cR$ by generators and 
relations.

We have generating arrows of three types:
\begin{itemize}
\item[1.] a reflexion~\footnote{In this paper we follow the idea of
K.~Chukovsky~\cite{Krokodil} to use synonyms for different
(although related) objects. In {\sl loc. cit} these are Hyppopotamus and Behemoth that are synonymous in Russian. In this paper we will later introduce {\sl reflections} generating the Weyl group, that will be related to, but different from the reflexions defined now.}
  $r_x:(\fh,a,b,p) \to (\fh,a',b',p')$
defined by a source $(\fh,a,b,p)$ and 
{\sl  a reflectable element}
$x\in X$, see~\ref{ss:reflexions} for the explicit formulas;
\item[2.] a tautological arrow $t_\theta:(\fh,a,b,p)\to
(\fh',a',b',p)$ determined by $\theta:\fh\stackrel{\sim}{\to}
\fh'$. Here $a':=\theta\circ a$, $b'=((\theta^*)^{-1})
\circ b$.   
\item[3.] a homothety  $h_\lambda:(\fh,a,b,p)\to (\fh,a',b,p)$ 
determined by $\lambda:X\to\bC^*$, with 
$a'(x)=\lambda(x)a(x)$. 
\end{itemize}

This collection of objects and arrows (=quiver) generates a free category denoted (temporarily) $\cF$. In other words, the objects of $\cF$ are the root data;
the arrow in $\cF$ are the paths composed of the generating arrows.
The groupoid $\cR$  will be defined as the one with the same objects as 
$\cF$, and whose arrows are equivalence classes of the arrows above. The 
equivalence relation is defined below.

First of all, we define a functor $\fh:\cF\to\Vect$
to the category of vector spaces carrying $(\fh,a,b,p)$
to $\fh$, carrying the reflexions and the homotheties
to the identities, and tautological arrows to the respective
isomorphisms of the underlying vector spaces.

\subsubsection{}
\label{sss:eq}
The equivalence relation on each Hom-set of $\cF$
is defined as follows: two compositions of arrows 
$(\fh,a,b,p)\to(\fh',a',b',p')$ are equivalent
if they induce the same isomorphism $\fh\to\fh'$.

\subsection{Formulas for the reflexions}
\label{ss:reflexions}
Any root datum $(\fh,a,b,p)$ determines a Cartan matrix
$A(a,b)=(a_{xy})_{x,y\in X}$ given by the formula
$$a_{xy}:=\langle a(x),b(y)\rangle.$$

\begin{dfn}
An element $x\in X$ is called {\sl reflectable} at $v=(\fh,a,b,p)$ if
the following conditions hold.
\begin{itemize}
\item[1.] If $a_{xx}=0$ then $p(x)=1$;
\item[2.] If $a_{xx}\ne 0$ and $p(x)=0$ then $\frac{2a_{xy}}{a_{xx}}
\in\bZ_{\leq 0}$.
\item[3.] If $a_{xx}\ne 0$ and $p(x)=1$ then $\frac{a_{xy}}{a_{xx}}
\in\bZ_{\leq 0}$.
\end{itemize}
\end{dfn}
\subsubsection{}
\label{sss:reflexionformulas}
Let $x\in X$ be reflectable at $v=(\fh,a,b,p)$. The reflexion $r_x:v\to 
v'=(\fh,a',b',p')$ is defined as follows.
\begin{itemize}
\item[(anisotropic)] If  $a_{xx}\not=0$, then $p':=p$ and
$$a'(y):=a(y)- 2\frac{a_{yx}}{a_{xx}}a(x),\ \ \
b'(y):=b(y)- 2\frac{a_{xy}}{a_{xx}}b(x).$$
\item[(isotropic)] If $a_{xx}=0$ then $p(x)=1$ and 
$$(a'(y),b'(y),p'(y)):=\left\{\begin{array}{l}
(-a(x),-b(x),p(x)) \ \ \ \ \  \text{ if } x=y,\\
(a(y),b(y),p(y)) \ \ \ \ \  \ \ \ \ \text{ if } x\not=y,\ \  a_{xy}=0,\\
(a(y)+\frac{a_{yx}}{a_{xy}}a(x),
  b(y)+b(x), 1+p(y))
  \ \  \text{ if }  a_{xy}\not=0.\end{array}
\right.$$
\end{itemize}

\begin{dfn}
The pair $(A=\{a_{xy}\}, p)$ will be called {\sl Cartan datum} for $v$.
\end{dfn} 

Note that the reflectability of $x\in X$, as well as the formulas for the reflexion $r_x$
depend only on the Cartan datum.

\subsubsection{}
\label{sss:cartanmatrix-change}
Let us indicate what happens to a Cartan matrix under a reflexion.
Anisotropic reflexions preserve the Cartan matrix. If $r_x:v\to v'$ is 
an isotropic reflexion ($a_{xx}=0$), the Cartan matrix $(a'_{yz})$ is given by the following formulas
$$\begin{array}{ll}
a'_{xy}=-a_{yx}, \\
a'_{yx}=-a_{xy},\\
a'_{yy}=\left\{\begin{array}{ll}
a_{yy}+2a_{yx}&\text{ if } a_{xy}\ne 0\\
a_{yy}&\text{ if } a_{xy}=0.\end{array}
\right. \\
a'_{yz}=\left\{\begin{array}{ll}
a_{yz} \ \ \ \ \ & \text{ if }a_{xz}=0, x,y\not=z,\\
a_{yz}+a_{yx} &\text{ if } a_{xz}\not=0, a_{xy}=0, x,y\not=z\\
a_{yz}+a_{yx}(1+\frac{a_{xz}}{a_{xy}}) & \text{ if } a_{xz}\not=0, a_{xy}\not=0.\end{array}
\right.
\end{array}$$

\begin{PRP}
The category $\cR$ is a groupoid.
\end{PRP}
\begin{proof}
 It is enough to verify that each generating arrow in 
$\cF$ has invertible image in $\cR$. First of all,
in our category the composition of two tautological arrows is 
tautological, defined by the composition of the 
corresponding isomorpisms $\fh\stackrel{\sim}{\to}\fh'\stackrel{\sim}{\to}\fh''$. Similarly, 
composition of two homotheties is a homothety. This implies 
that tautological arrows and homotheties are invertible. 
Invertibility of reflexions follows from the formulas:
one has $r_x^2=\id$ for all $x$ (this is an explicit calculation).
\end{proof}

Note the following observation.
\begin{lem}
\label{lem:sym-stable}
All reflexions preserve the symmetricity of a Cartan matrix.
\end{lem}
\begin{proof}
Anisotropic reflexion does not change the Cartan matrix. Isotropic reflexions do change, but the resulting Cartan matrix remains symmetric if the original matrix was symmetric. This results from a direct calculation.
\end{proof}
 
\begin{dfn}
A connected component $\cR_0$ of $\cR$ is called symmetrizable if there
exists $v\in\cR_0$ having a symmetric Cartan matrix.
\end{dfn}

Note that $\cR_0$ is symmetrizable if all Cartan matrices of $v'\in\cR_0$
are symmetrizable in the sense of Kac~\cite{Kbook}.

\subsection{Properties}
\subsubsection{}
\label{sss:comcom}
One has obviously $t_\theta\circ t_{\theta'}=t_{\theta\circ\theta'}$ and
$h_\lambda\circ h_{\lambda'}=h_{\lambda\lambda'}$.
The morphisms $r_x$, $t_\theta$ and $h_\lambda$ commute with each other.

\

The root groupoid $\cR$ consists  of connected 
components some of which will lead to  
interesting Lie superalgebras. 

We present below properties of a component $\cR_0$ of $\cR$
that will be relevant to Lie theory.

This is  weak symmetricity.

\begin{dfn} 
\label{dfn:quasisym}
$ $
\begin{itemize} 
\item[1.] A root datum is {\sl locally weakly symmetric} if
 $a_{xy}=0$ implies $a_{yx}=0$ for any reflectable $x$.
 
\item[2.] A root datum is weakly symmetric  
if all root data in its connected component are locally weakly symmetric.
\end{itemize}
\end{dfn}

\begin{rem}
Let $v\in\cR$. The group of automorphisms $\Aut_\cR(v)$ acts on
$\fh(v)$. This action is faithful by definition of the equivalence 
relation on the Hom sets of $\cF$, see \ref{sss:eq}.
\end{rem}

\begin{rem}
The root groupoid  $\cR$ is an object of ``mixed'' nature.
It is a groupoid, but its objects and Hom sets carry an extra information
(markings $a,b,p$, generators $r_x,t_\theta,h_\lambda$ for arrows).
This is why we cannot easily replace $\cR$ with any equivalent groupoid
(for instance, leaving only one object for each connected component).

Nevertheless, we can safely assume that $\fh$ is the same vector space at 
all objects of a given connected component $\cR_0$, allowing however
the tautological arrows $t_\theta$ for automorphisms $\theta:\fh\to\fh$.
\end{rem}
 
\begin{rem}
Tautological arrows and  anisotropic 
reflexions (those with $a_{xx}\not=0$) preserve the Cartan datum.
Homotheties also preserve local weak symmetricity. 
Isotropic reflexions usually do not satisfy this property. For this reason  admissible root data with isotropic reflexions can be classified under the assumption
that all elements $x\in X$ are reflectable at every vertex, \cite{Hoyt}.
\end{rem}

\begin{dfn}
\label{dfn:Deq}
Two Cartan data, $(A,p)$ and $(A',p')$, will be called $D$-equivalent if $p=p'$ and there exists
an invertible diagonal matrix $D$ such that $A'=DA$.
\end{dfn}
Obviously, homotheties carry a Cartan datum to a $D$-equivalent one.

\begin{rem}
In studying a connected component $\cR_0$ of $\cR$ it is often important to describe  Cartan data $(A(v),p)$ at all vertices $v\in\cR_0$, up to $D$-equivalence. Since only isotropic reflexions
change the Cartan data, it is sufficient to perform only sequences of isotropic reflexions, see~\ref{sss:spine}.
\end{rem}

\subsection{Examples: reflectability}

\subsubsection{}
\label{ex:nonreflectable}
We present an example of a reflexion $r_x:v\to v'$ such that
all $y\in X$ are reflectable at $v$ but some are not reflectable
at $v'$.

Take the root datum $v$ with $X=\{x,y\}$, the Cartan matrix 
$\begin{pmatrix} 0 &-s\\-s &1\end{pmatrix}$, $s\in\bN$, $p(x)=p(y)=1$. Then
$x$ and $y$ are reflectable at $v$. For 
the reflexion $r_x:v\to v'$
the reflected Cartan matrix is  
$\begin{pmatrix} 0& s\\s &1-2s\end{pmatrix}$
and $p'(x)=1$, $p'(y)=0$.
Thus $y$ is reflectable at $v'$ only if $\frac{2s}{2s-1}\in\bN$ that is 
for $s=0,1$.

\subsection{Examples: calculation of $\Aut_\cR(v)$}

\subsubsection{Semisimple case}
\label{sss:ss}

Let $v=(\fh,a,b,p)$ represent a root system of a finite dimensional semisimple Lie algebra. This
means that $p(x)=0$, $a:X\to\fh$ is a set of simple coroots
and $b:X\to\fh^*$ is the set of simple roots. Both
$a$ and $b$ give bases in $\fh$ and $\fh^*$. Let us calculate
the group of automorphisms of $(\fh,a,b,0)$. 
Any reflexion $r_x:(\fh,a,b,0)\to(\fh,a',b',0)$ gives rise
to an automorphism $s_x:v\to v$, $s_x=t_{s_{b(x)}}\circ r_x$ where
the automorphism $s_{b(x)}:\fh\to\fh$ of $\fh$ is the standard reflection  on $\fh$ with respect 
to $b(x)\in\fh^*$. Note that $s_x:v\to v$  
induces precisely the automorphism $s_{b(x)}:\fh\to\fh$, so that the assignment $s_{b(x)}\mapsto s_x$ is compatible with the action of the Weyl group
$W$ and of $\Aut_\cR(v)$ on $\fh$. Since the actions are faithful,
this defines an injective group homomorphism
$$
i:W\to\Aut_\cR(v).
$$
We claim that it is bijective. In fact,
any automorphism  $\eta:v\to v$ in $\cR$ is a composition 
of reflexions $r_x$, tautological arrows and homotheties.
Since reflexions, tautological arrows and homotheties commute, one
can, using~\ref{sss:comcom}, present 
\begin{equation}
\label{eq:ss-deco}
\eta=h_\lambda\circ t_\theta\circ i(w),
\end{equation}
for a certain $w\in W$. It remains to verify that if $h_\lambda\circ t_\theta\in\Aut_\cR(v)$, then it is identity. Since $t_\theta$ does not change the Cartan matrix, $h_\lambda=\id$. Since any automorphism of
$v$ carries $a(x)$ and $b(x)$ to themselves, and $a(x)$ generate $\fh$,
$\theta=\id$.

\subsubsection{The case of Kac-Moody algebras} 
\label{sss:KMexample}
In the case when $(\fh,a,b,0)$  
has the Cartan matrix satisfying the conditions of~\cite{Kbook}, 1.1,
the calculation of ~\ref{sss:ss} works almost as well.

Let $W$ denote the Weyl group and let $\wt W=\Aut_\cR(v)$.
We have a group homomorphism $i:W\to\wt W$ defined exactly as in the
semisimple case. Precisely as in the semisimple case we have a 
decomposition~(\ref{eq:ss-deco}) of an automorphism $\eta\in\wt W$
and deduce that $h_\lambda=\id$ as the Cartan matrix has no zero rows.
Denote
$$
K=\{\theta:\fh\to\fh|\theta(a(x))=a(x),\theta^*(b(x))=b(x),\ x\in X\}.
$$
Any $\theta\in K$ commutes with $s_{b(x)}:\fh\to\fh$.
This implies that $\wt W=W\times K$.

Let us show $K$ is a commutative unipotent group.

Denote $A\subset\fh$ and $B\subset\fh^*$ the subspaces spanned by the images of $a$ and $b$. One has $\dim A=\dim B=|X|$ and $\dim\fh=2|X|-r$ where $r$ is the rank of the
Cartan matrix. This is equivalent to saying that the orthogonal complement $B^\perp\subset\fh$ of $B$ lies in $A$.
If $\theta$ is an automorphism of the triple $(\fh,a,b)$,
$\theta-1$ vanishes on $A$ and has image in $B^\perp$. 
This means that $(\theta-1)^2=0$. Moreover, any two such automorphisms
commute. The dimension of $K$ is $(|X|-r)^2$.

\subsubsection{Root datum for $\fg\fl(1|1)$}
\label{sss:aut-gl11}

We assume $\dim(\fh)=2$, $X=\{x\}$, $a=a(x)\in\fh$, $b=b(x)\in\fh^*$
so that $a\ne 0, b\ne 0$ but $\langle b,a\rangle=0$. The only isotropic reflexion
carries the quadruple $v=(\fh,a,b,p=1)$ to $v'=(\fh,-a,-b,1)$.
The tautological arrow $t_{-1}:v'\to v$ is defined by $-1':\fh\to\fh$.
The composition $t_{-1}\circ r_x$ is an automorphism of $v$ of order $2$.
It is easy to see that $\Aut(v)=\bZ_2\times K$ where $\bZ_2$
is generated by the automorphism described above  and 
$K=\{\theta:\fh\to\fh|\ \theta(a)\in\bC^*a,\theta^*(b)=b\}$.

	For more examples see~\ref{sss:gl12} and Section~\ref{sec:app}.

\section{Root Lie superalgebras}
\label{sec:root}

In this section we define root Lie superalgebras corresponding to
certain (admissible) connected components of the groupoid $\cR$ of root 
data. 

\subsection{Half-baked Lie superalgebra}
\subsubsection{}
\label{sss:half}
Let $v=(\fh,a,b,p)\in\cR$. We assign to $v$ a Lie superalgebra 
$\wt\fg(v)$ generated by $\fh=\fh(v)$, $\tilde e_x,\tilde f_x,\ x\in X$,
with the parity given by $p(\fh)=0,\ p(\tilde e_x)=p(\tilde f_x)=p(x)$,
subject to the relations
\begin{itemize}
\item[1.] $[\fh,\fh]=0$,
\item[2.]  $[h,\tilde{e}_x]=\langle b(x), h\rangle \tilde{e}_x,
[h,\tilde{f}_x]=-\langle b(x), h\rangle \tilde{f}_x$
\item[3.] $[\tilde e_x,\tilde f_y]=0$ for $y\ne x$
\item[4.] $[\tilde e_x,\tilde f_x]=a(x)$
\end{itemize}
for each $x\in X$.

We call $\wt\fg(v)$ {\sl the half-baked Lie superalgebra} defined by
the root datum $v\in\cR$.
 
\subsubsection{}
\label{sss:properties}
The following properties of $\wt\fg:=\wt\fg(v)$ are proven in Thm. 1.2 of~\cite{Kbook} for Lie algebras (the proof works verbatim for
Lie superalgebras).
 
\begin{itemize}
\item[1.] The algebra $\fh$ acts diagonally on 
$\wt\fg$. We denote by $\wt\fg_\mu$ the weight space of weight 
$\mu$, so that
$\wt\fg=\oplus_{\mu\in\Span_\bZ(b)}\wt\fg_\mu$, where $\Span_\bZ(b)$ denotes the abelian subgroup of $\fh^*$ generated by $b(x), x\in X$.
\item[2.]There is a standard triangular decomposition
$$
\wt\fg=\wt\fn^+\oplus\fh\oplus\wt\fn^-,
$$
where $\wt\fn^+$ is freely  generated by 
$\tilde e_x$, $x\in X$ 
and $\wt\fn^-$ is freely  generated by $\tilde f_x$.
\item[3.] For each $x\not=y$ one has 
$\wt\fg_{jb(x)+b(y)}=0$ for $j\not\in\mathbb{Z}_{\geq 0}$ and
$\wt\fg_{jb(x)+b(y)}$ is spanned by $(\ad \tilde{e}_x)^j\tilde{e}_y$.
\end{itemize}

The following theorem is very similar to \cite{Kbook}, Thm. 2.2.

\begin{prp}
\label{prp:likekac22}
Let $v\in\cR$ have a symmetric Cartan matrix $(a_{xy})$. 
Let $(\cdot|\cdot)$ be a nondegenerate symmetric form on $\fh$
satisfying the condition
\begin{itemize}
\item[] $(a(x)|h)=\langle b(x),h\rangle$ for any $x\in X$, $h\in\fh$.
\end{itemize}
Then there exists a unique extension of $(\cdot|\cdot)$ to an invariant
symmetric bilinear form on $\wt\fg=\wt\fg(v)$. This extension
enjoys the following properties.
\begin{itemize}
\item[1.] $(\tilde e_x|\tilde f_y)=\delta_{xy}$.
\item[2.] $(\wt\fg_\alpha|\wt\fg_\beta)=0$ unless $\alpha+\beta=0$.
\item[3.] $[z,t]=(z|t)\nu(\alpha)$ for $z\in\wt\fg_\alpha$,
$t\in\wt\fg_{-\alpha}$, where $\nu:\fh^*\to\fh$ is the isomorphism
defined by the original nondegenerate form.
\end{itemize}
\end{prp}
\qed

\subsubsection{}
\label{sss:automorphism}

The algebra $\wt\fg(v)$ admits a standard {\sl superinvolution} $\theta$,
that is an automorphism whose square is $\id$ on the even part 
and $-\id$ on the odd part of $\wt\fg(v)$.
We will define the superinvolution $\theta$  by the following formulas.
\begin{itemize}
\item $\theta|_\fh=-\id$.
\item $\theta(\tilde e_x)=\tilde f_x$.
\item $\theta(\tilde f_x)=(-1)^{p(x)}\tilde e_x$.
\end{itemize}

\subsubsection{Example: rank one}
\label{sss:rank1-wt}

Let $X=\{x\}$.  The Cartan matrix
is a $1\times 1$ matrix $(a_{xx})$.

In the discussion below we present $\fh$ having the smallest
possible dimension. The general case can be treated using~\ref{sss:decomposable}.

If  $a_{xx}\not=0$ and $p(x)=0$, we have $\wt\fg=\fsl_2$;
 if $p(x)=1$, we have $\wt\fg=\fosp(1|2)$.

If $a_{xx}=0$ and $p(x)=0$, $\wt\fg$ is the $(4|0)$-dimensional algebra
$a(x),d, e_x,f_x$,  
with $\fh=\Span(a(x),d)$,
 $a(x)=[e_x,f_x]$ central and $[d,e_x]=e_x$, $[d,f_x]=-f_x$.

In the remaining case
$p(x)=1$ and $a_{xx}=0$. The algebra $\wt\fg$ has dimension $(4|2)$
with a basis 
$$a(x),d, e_x,f_x,e_x^2,f_x^2,$$ 
($e_x$ and $f_x$ odd)
with $\fh=\Span(a(x),d)$,
$a(x)=[e_x,f_x]$  central and $[d,e_x]=e_x$, $[d,f_x]=-f_x$.

\subsubsection{}
\label{sss:properties-2}

The space
$[\wt\fg_{jb(x)+b(y)},\wt\fg_{-jb(x)-b(y)}]$ lies in 
$\fh$ for any $j\geq 0$ and is at most one-dimensional.
We wish to describe, under certain assumptions, the greatest
value of $j$ for which it is nonzero.

Assume that $x\ne y\in X$, $x$ is reflectable at $v$.

Let $r_x:v\to v'=(\fh,a',b',p')$ be the corresponding reflexion in 
$\cR$.
Choose $j_0$ such that $b(y)+j_0b(x)=b'(y)$,
that is $j_0=-2\frac{a_{xy}}{a_{xx}}$
for $a_{xx}\not=0$,
$j_0=1$ for $a_{xx}=0$, $a_{xy}\ne 0$,
and $j_0=0$ for $a_{xx}=0=a_{xy}$.

\begin{lem}
\label{lem:rk2-ideal}
Assume that $X=\{x,y\}$ and $x$ is reflectable at $v=(\fh,a,b,p)$.
Let $j_0$ be defined as above.
Define the ideal $I$ of $\wt\fg=\wt\fg(v)$ generated by
the elements
\begin{equation}
\label{eq:rk2-ideal}
E:=(\ad\tilde e_x)^{j_0+1}\tilde e_y,\ 
F:=(\ad\tilde f_x)^{j_0+1}\tilde f_y.
\end{equation}
Then
\begin{itemize}
\item[(a)] If $a_{xx}=0$ then the ideal $I'$ generated by $\tilde e_x^2$,
$\tilde f_x^2$ satisfies $I'\cap\fh=0$.
\item[(b)] If $a_{xx}=0$, $a_{xy}\ne 0$ then $I\subset I'$ and
$I=I'$ iff $a_{yx}\ne 0$.
\item[(c)] $I\cap\fh\ne 0$ if and only if $a_{xx}\ne 0, a_{yx}\ne 0$ and $a_{xy}=0$.
\end{itemize}
\end{lem} 
\begin{proof}
(a) 
Let $a_{xx}=0$. 
Then $p(x)=1$ and 
\begin{equation}\label{eff}
[\tilde e_x,\tilde f_x^{2}]=0.\end{equation}
Since $[\tilde e_y,\tilde f_x^{2}]=0$ we obtain $[\wt\fn^+, f_x^2]=0$;
similarly, $[\wt\fn^-, e_x^2]=0$. This gives
$I'\cap\fh=0$ and establishes (a).  

(b) Take 
$a_{xx}=0$,  $a_{xy}\not=0$. Then $j_0=1$ so
$$
F=(\ad\tilde f_x)^{2}\tilde f_y=(\ad\tilde f_x^{2})\tilde f_y,\ \ \ 
E=(\ad\tilde e_x^{2})\tilde e_y
$$
In particular, $I\subset I'$ and
$$[\tilde e_y, F]=\pm [\tilde f_x^{2},
a(y)]=\pm 2 a_{yx}\tilde f_x^2.$$
This gives $I=I'$ if $a_{yx}\not=0$. Consider the case $a_{yx}=0$. By above,
$[\tilde e_y, F]=0$. By~(\ref{eff}) we have
$[\tilde e_x, F]=0$. Hence $[\wt\fn^+, F]=0$ and so
$F\not\in I'$. This completes the proof of (b).

(c) By (a), (b) it follows that $I\cap \fh=0$ if $a_{xx}=0$, $a_{xy}\not=0$.
Therefore we may assume that $a_{xy}=a_{yx}=0$ or $a_{xx}\not=0$.
It is enough to verify that $[\tilde e_z,F]=[\tilde f_z,E]=0$
for $z=x,y$. These formulas are similar so we will check only the formula
$[\tilde e_z,F]=0$.

If $a_{xy}=a_{yx}=0$, then $j_0=0$ and
$$[\tilde e_x, F]=[\tilde e_x, [\tilde f_x,\tilde f_y]]=[[\tilde e_x, \tilde f_x]\tilde f_y]=[a(x),f_y]=-a_{xy}f_y=0$$
as well as $[\tilde e_y, F]=\pm a_{yx}f_x=0$
as required.

Consider the case when   $a_{xx}, a_{xy}, a_{yx}\not=0$. Then $j_0=-2\frac{a_{xy}}{a_{xx}}$.
Recall that $\tilde f_x,\tilde e_x$
generate $\fsl_2$ if $p(x)=0$ and $\fosp(1|2)$ if $p(x)=1$. Since
$[\tilde e_x,\tilde f_y]=0$, a direct computation implies
$$(\ad \tilde e_x)(\ad\tilde f_x)^{j_0+1}\tilde f_y=0.$$
On the other hand, $[\tilde e_y,\tilde f_x]=0$ implies
$$[\tilde e_y, F]=\pm (\ad\tilde f_x)^{j_0+1} a(y)=\pm a_{yx}(\ad\tilde f_x)^{j_0} \tilde f_x=0$$
since $[\tilde f_x,\tilde f_x]=0$ for $p(x)=0$ and 
$[\tilde f_x,[\tilde f_x,\tilde f_x]]=0$ if $p(x)=1$ (in the case $a_{xx}\not=0$, $p(x)=1$ the condition that $x$ is reflectable at $v$ implies that $j_0$ is even, in particular, $j_0\geq 2$).
Hence $[\tilde e_y, F]=[\tilde e_x, F]=0$ as required.

Finally, if $a_{xx}\ne 0$, $a_{xy}=0$, $a_{yx}\ne 0$, then $b'(y)=b(y)$ and $a'(y)=a(y)-2\frac{a_{yx}}{a_{xx}}a(x)$. Furthermore,
$E=[\tilde{e}_x,\tilde{e}_y]$, so that 
$$
[\tilde f_x,[\tilde{f}_y,E]]=\pm[\tilde f_x,[\tilde{e}_x,a(y)]]=\pm a_{yx}a(x)\ne 0.
$$

\end{proof}

\begin{prp}
\label{prp:bracket}
Assume that $x\not=y\in X$ and $x$ is reflectable.  
We also assume that if $a_{xx}\ne 0$ and $a_{xy}=0$ then $a_{yx}=0$.
\begin{itemize}
\item[1.]The bracket 
$[\wt\fg_{jb(x)+b(y)},\wt\fg_{-jb(x)-b(y)}]$ is zero for $j>j_0$.
\item[2.]$[\wt\fg_{b'(y)},\wt\fg_{-b'(y)}]$ is spanned by $a'(y)$.
\end{itemize}
\end{prp}
\begin{proof}
The claim immediately reduces to the case $X=\{x,y\}$.
Denote by 
$I$ the ideal of $\wt\fg$ generated by the elements
$$
E:=(\ad\tilde e_x)^{j_0+1}\tilde e_y,\ 
F:=(\ad\tilde f_x)^{j_0+1}\tilde f_y.
$$
By Lemma~\ref{lem:rk2-ideal} we have $I\cap\fh=0$.
The homomorphism $\wt\fg\to\fg=\wt\fg/I$ is identity on 
$\fh$, so both claims of the proposition would follow from the similar claims for $\fg$. Since the first claim of the proposition tautologically holds for $\fg$, we have proven it also for $\wt\fg$.

To prove the second claim for $\fg$, 
we will study the isotropic and the anisotropic cases separately.

{\sl The case $a_{xx}\ne 0$.} 
 The rank one subalgebra defined by 
$\{x\}\in X$ contains a copy of $\fsl_2$.
 $\fg$ is integrable as an $\fsl_2$-module as it is generated by
the elements on which $\tilde e_x,\tilde f_x$ act locally nilpotently,
 see~\cite{Kbook}, Lemma 3.4. 
Therefore, the automorphism $\sigma:\fg\to\fg$
given by the formula
\begin{equation}
\label{eq:sigma:gtog}
\sigma=\exp(\tilde f_x)\circ\exp(-\tilde e_x)\circ\exp(\tilde f_x),
\end{equation}
is defined. Its restriction on $\fh$ is given by the standard
formula
$
\sigma(h)=h-\frac{2}{a_{xx}}\langle h,b(x)\rangle a(x),
$
so
$\sigma(\fg^\U_\mu)=\fg^\U_{\sigma(\mu)}$,
where the action of $\sigma$ on $\fh^*$ is induced by its action on 
$\fh$. The latter implies the second claim of the proposition for the algebra $\fg$.  

{\sl The case $a_{xx}=0$. } If $a_{xy}=0$, the second claim is immediate.
In the case $a_{xy}\ne 0$ a direct calculation shows that
$$
[[\tilde e_x,\tilde e_y],[\tilde f_x,\tilde f_y]]=
(-1)^{p(y)}a_{xy}(a(y)+\frac{a_{yx}}{a_{xy}}a(x)).
$$

\end{proof}

\subsection{Coordinate systems and root algebras}\label{ss:rootalgebra}
\begin{dfn}
Let $v\in\cR$.
A $v$-coordinate system on a Lie superalgebra $\fg$
is a surjective homomorphism $\wt\fg(v)\to\fg$ whose kernel has zero intersection with $\fh(v)$.
\end{dfn}

In other words, a $v$-coordinate system on $\fg$ consists of an
injective map of Lie superalgebras $\fh\to\fg$ ($\fh$ is even commutative), and a collection of generators $e_x,f_x$ such that the
relations 1--4 of  \ref{sss:half} hold.

Here is our main definition.

\begin{dfn}
Let $\cR_0\subset\cR$ be  a  connected component. 
A root Lie superalgebra $\fg$ supported on $\cR_0$ is a collection
of Lie superalgebras $\fg(v),\ v\in\cR_0$, endowed with 
$v$-coordinate systems so that for any $\alpha:v\to v'$ in $\cR_0$
there exists an isomorphism $a:\fg(v)\to\fg(v')$ extending the isomorphism $\fh(\alpha):\fh(v)\to\fh(v')$.
\end{dfn}

Let $\fg$ be a root Lie superalgebra at $\cR_0$. There is a weight space decomposition
$$
\fg(v)=\fh(v)\oplus\bigoplus_{\mu\in\Delta(v)}\fg(v)_\mu
$$
with $\Delta(v)\subset\Span_\bZ(b)$. The elements of $\Delta(v)$ are called
{\sl the roots} of $\fg$ (at $v$). The elements $b(x),\ x\in X$, are {\sl the simple roots}
at $v$. Any $\alpha:v\to v'$ carries the root decomposition at $v$ to that at $v'$.

\begin{dfn}
A component $\cR_0$ of $\cR$ is called {\sl admissible} if it admits
a root Lie superalgebra.
\end{dfn}

\subsubsection{}
Let $v\in\cR$. The half-baked algebra $\wt\fg(v)$ has a triangular decomposition. This implies the existence of the maximal
ideal $\fr(v)$ having zero intersection with $\fh(v)$. If $\cR_0$ is admissible, then
the collection of $\fg^\CG(v)=\wt\fg(v)/\fr(v)$ is a root
Lie superalgebra supported at $\cR_0$. 
In fact, given a root algebra $\fg$ with $\fg(v)=\wt\fg(v)/I(v)$, the quotient ideal
 $\bar\fr(v)=\fr(v)/I(v)$ is the maximal ideal in $\fg(v)$ having zero
intersection with $\fh(v)$. Obviously, any isomorphism
$a:\fg(v)\to\fg(v')$ over $\alpha:v\to v'$ in $\cR$ carries
$\bar\fr(v)$ to $\bar\fr(v')$, and therefore induces an isomorphism
$\fg^\CG(v)\to\fg^\CG(v')$.

We call the collection
$\fg^\CG=\{\fg^\CG(v)\}_{v\in\cR_0}$ the {\sl contragredient} Lie superalgebra 
supported at an admissible component $\cR_0$. In other words, 
the contragredient
Lie superalgebra $\fg^\CG$ is the terminal object in the category of 
root Lie superalgebras
supported at an admissible component $\cR_0$.

The superinvolution $\theta$ of $\wt\fg$ defined in 
\ref{sss:automorphism} induces an automorphism of $\fg^\CG$.

\subsubsection{Rank one}
\label{rank1}
The Lie algebra $\fsl_2$ plays a prominent role in Lie theory.
A similar role in our setup will be played by root algebras of rank 1.
Let us describe them all.

Let $X=\{x\}$. In this case $\wt{\fg}(v)$ is described in
\ref{sss:rank1-wt}. It is a root algebra. 

If  $a_{xx}\not=0$ or $p(x)=0$, then $\fg^\CG=\wt\fg$.

If $a_{xx}=0$ and $p(x)=1$, the maximal ideal $\fr$ of $\wt\fg$ having zero intersection 
with $\fh$ is spanned by 
$e_x^2,f_x^2$ and $\fg^\CG=\wt\fg/\fr\cong \fgl(1|1)$. The algebras 
$\wt\fg$ and $\fg^\CG$ are
exactly two root algebras in this case as only these two allow an automorphism lifting  $\gamma=t_{-1}\circ r_x$, see~\ref{sss:aut-gl11}.

\subsubsection{Decomposable root datum}
\label{sss:decomposable}
Let $X=X_1\sqcup X_2$ and let
$v_i=(\fh_i, a_i:X_i\to\fh_i,b_i:X_i\to\fh_i^*,p_i:X_i\to\bZ_2$,
$i=1,2$, be two root data of ranks $|X_1|$ and $|X_2|$ respectively.

We define their sum $v=v_1+v_2$ in an obvious way, as the root datum
with $\fh=\fh_1\oplus\fh_2$ and $a:X\to\fh$, $b:X\to\fh^*$ and $p:X\to\bZ_2$ defined by the conditions
$$
a_{|X_i}=s_i(a_i), b_{|X_i}=s_i^*(b_i),\ p_{|X_i}=p_i,
$$ 
where $s_i:\fh_i\to\fh$ and $s_i^*:\fh_i^*\to\fh^*$
are the obvious embeddings.

We will denote by $\cR(X),\ \cR(X_1)$ and $\cR(X_2)$ the groupoids of root data for the sets $X,X_1$ and $X_2$. The component $\cR_0$
of $\cR(X)$ containing $v=v_1+v_2$ is obviously a direct product 
$\cR'_0\times\cR''_0$ of the corresponding components of $\cR(X_1)$ and 
$\cR(X_1)$. 
If $\fg_1$ and $\fg_2$ are root algebras supported on the components $\cR'_0$ and $\cR''_0$
respectively, the product $\fg=\fg_1\times\fg_2$ is a root algebra of $\cR_0$. In particular,
$\fg^\CG_1\times\fg^\CG_2$ is the contragredient root algebra for $\cR_0$.
Theorem~\ref{thm:admissible=wsym} implies that if $\cR_0$ is admissible, then both
$\cR'_0$ and $\cR''_0$ are admissible.
It is not true in general that any root algebra supported on $\cR_0$ is a product.

Here is the best  we can say. 

\begin{prp}
\label{prp:deco}
Let $X=X_1\sqcup X_2$, $v=v_1+v_2$ be defined as above, with
$v\in\cR_0$, $v_1\in\cR_0'$ and $v_2\in\cR_0''$. Assume that all $x\in X_1$ are reflectable
at all $v'\in\cR'_0$. Then any root algebra
supported on $\cR_0$ uniquely decomposes as a product of a root
algebra supported on $\cR_0'$ and a root algebra supported on $\cR_0''$.
\end{prp}
\begin{proof}
The algebra $\fg=\fg(v)$ is generated by $\fh$, $e_x,f_x,e_y,f_y$ where $x\in X_1$ and $y\in X_2$.
We have to verify that $[e_x,e_y]=0=[f_x,f_y]$ for $x\in X_1$ and $y\in X_2$. 
The reflexion $r_x:v\to v'$ with respect to $x\in X_1$ carries, up to scalars, $e_x$ to $f'_x$ and 
$f_x$ to $e'_x$, retaining $e_y$ and $f_y$. Since $[e'_x,f'_y]=0=[e'_y,f'_x]$, we deduce
$[e_x,e_y]=0=[f_x,f_y]$.
\end{proof}

We can apply the sum of root data operation to an empty root datum
$\emptyset_V$
corresponding to $X=\emptyset$ and uniquely defined by a vector space
$V$. For $v=(\fh,a,b,p)$ the sum $\emptyset_V+v$ has form 
$(\fh\oplus V,a,b,p)$ and any root algebra based on it is the direct product of a root algebra based on $v$ with the commutative algebra $V$.

The following result is a corollary of ~\ref{rank1}.

\begin{crl}\label{corgalpha}
Let $\cR_0$ be an admissible component of $\cR$ and $\fg:=\fg(v)$ be a 
root algebra.
Fix $x\in X$ and set $\alpha:=b(x)$.
We denote by $\fg\langle\alpha\rangle$ the subalgebra of $\fg$
generated by $\fg_{\alpha}$ and $\fg_{-\alpha}$.
\begin{enumerate}
\item
If $a_{xx}\not=0$ and $p(x)=0$, 
one has $\fg\langle\alpha\rangle=\fsl_2$ and
$\fg_{i\alpha}=0$ for $i\not\in\{0,\pm 1\}$.
\item
If $a_{xx}\not=0$ and $p(x)=1$, 
one has $\fg\langle\alpha\rangle=\fosp(1|2)$ and $\fg_{i\alpha}=0$
for $i\not\in\{0,\pm 1,\pm 2\}$.
\item If $a_{xx}=0$ and $p(x)=0$  then $\fg\langle\alpha\rangle$ is 
the Heisenberg algebra and $\fg_{i\alpha}=0$ for $i\not\in\{0,\pm 1\}$.
\item
If $p(x)=1$, $a_{xx}=0$ and $a_{xy}, a_{yx}\not=0$ for some $y$
then  $\fg\langle\alpha\rangle\cong \mathfrak{sl}(1|1)$ and 
$\fg_{i\alpha}=0$ for $i\not\in\{0,\pm 1\}$.  
\item
If $p(x)=1$, $a_{xx}=0$ and $a_{xy}, a_{yx}=0$ for all $y$
then  $\fg\langle\alpha\rangle$ is the $(4|2)$-dimensional algebra
described in~\ref{sss:rank1-wt}. 
\end{enumerate}
\end{crl}
\begin{proof}
Clearly, $\fg\langle\alpha\rangle$ is a quotient of the
algebra $[\wt\fg,\wt\fg]$ where $\wt\fg$ is the corresponding algebra listed in~\ref{rank1}; this gives
(1), (2), (3) and  shows that in (4) it is enough to verify 
$\fg_{2b(x)}=0$. This follows from Lemma~\ref{lem:rk2-ideal}(b).
%
\end{proof}

Note that Corollary~\ref{corgalpha} implies that $x\in X$ is
reflectable at $v$ iff for $\alpha=b(x)$ the algebra $\fg\langle
\alpha\rangle$ is not the Heisenberg algebra and $e_\alpha$ acts on $\fg$
locally nilpotently.

\subsection{Admissibility is just a weak symmetricity}
 
In this subsection we prove the following result.
\begin{thm}
\label{thm:admissible=wsym}
A connected component $\cR_0$ of $\cR$ is admissible iff it is
weakly symmetric.
\end{thm}
\begin{proof}

1. Let $\cR_0$ be a weakly symmetric component of $\cR$. We claim
that the collection of $\fg^\CG(v)=\wt\fg(v)/\fr(v)$ forms a root 
Lie superalgebra. Let $r_x:v'\to v$ be a reflexion. Denote
$\wt\fg'=\wt\fg(v')$, $\fg=\fg^\CG(v)$. Let us show  that
there exists a homomorphism $\rho:\wt\fg'\to\fg$ identical on $\fh$.
The half-baked Lie superalgebra $\wt\fg(v')$ is generated by $\fh$,
$\tilde e'_y$ and $\tilde f'_y$, $y\in X$. In order to construct
$\rho$, we have to find $\rho(\tilde e'_y)$, $\rho(\tilde f'_y)$,
and verify the (very few) relations.

The weight of $\tilde e'_y$ is $b'(y)$, so we have to look for
$\rho(\tilde e'_y)$ in $\fg^\CG_{b'(y)}$. We know that
$\wt\fg_{b'(y)}$ is one-dimensional. By Proposition~\ref{prp:bracket} (2), the ideal generated by $\wt\fg_{b'(y)}$ contains $a'(y)\in\fh$,
so $\fr(v)$ does not contain it. Therefore, $\fg_{b'(y)}$ is also one-dimensional.
We will define arbitrarily $0\ne\rho(\tilde e'_y)\in\fg_{b'(y)}$ 
and choose $\rho(\tilde f'_y)\in\fg_{-b'(y)}$ so that 
$[\rho(\tilde e'_y),\rho(\tilde f'_y)]=a'(y)$. The latter is also 
possible by~Proposition~\ref{prp:bracket}(2).
It remains to verify that 
$[\rho(\tilde e'_y),\rho(\tilde f'_z)]=0$ for $y\ne z$.

(a) $y\ne x,\ z\ne x$. In this case the bracket should have weight 
$b'(y)-b'(z)=b(y)-b(z)+cb(x)$ for some $c\in\bZ$. This is not a weight
of $\wt\fg$, so the bracket should vanish.

(b) $z=x\ne y$. In this case the bracket should have weight
$b'(y)-b'(x)=b(y)+j_0b(x)+b(x)$ where $j_0$ is defined as
in~\ref{sss:properties-2}. According to Lemma~\ref{lem:rk2-ideal}(c)
the ideal generated by this weight space has no intersection with $\fh$,
so this is not a weight of $\fg$ and the bracket vanishes.

Therefore, we have constructed a homomorphism $\rho:\wt\fg'\to\fg$
for each reflexion $r_x:v'\to v$. It is identity on $\fh$, so
it induces a homomorphism $\fg'\to\fg$. Any reflexion has order two, 
so there is also a homomorphism $\fg\to\fg'$ in the opposite direction. 
Their composition preserves weight spaces, so it is invertible.

2. Assume now that $\cR_0$ is an admissible component. We will deduce
that it is necessarily weakly symmetric. Assume
that there exists $v\in\cR_0$, a $v$-reflectable element $x\in X$ and 
another $y\in X$ such that $a_{xy}=0$. Let $\fg$ be a root algebra.

Look at the $x$-reflexion $r_x:v\to v'$.
Since
$$b'(x)=-b(x),\ \ b'(y)=b(y)$$
one has $\tilde{\fg}'_{b(x)+b(y)}=0$ so $\fg_{b(x)+b(y)}=0$.
Therefore $[e_x,e_y]=0$. One has
$$a_{yx} e_x=[a(y),e_x]=[[e_y,f_y],e_x]=0$$
so $a_{yx}=0$ as required.
\end{proof}

\subsection{Admissible components in rank two}
\label{ss:ranktwo}

In this subsection we show that any locally weakly symmetric
root datum of rank two belongs to an admissible component (that is, a 
local weak symmetricity implies a weak symmetricity).

\subsubsection{Fully reflectable} A component $\cR_0$ of $\cR$ is called
{\sl fully reflectable} if all $x\in X$ are reflectable at all $v\in\cR_0$.
Classification of fully reflectable root data is available for all ranks.  Fully reflectable admissible root data without isotropic real roots can be easily classified as all Cartan matrices
in the component are $D$-equivalent.
The classification of fully reflectable admissible root data with isotropic real roots was obtained
in~\cite{Hoyt}.  

\subsubsection{Symmetrizable}

The cases  $a_{xy}=a_{yx}=0$ as well as $a_{xy}\ne 0$ and $a_{yx}\ne 0$ are symmetrizable, therefore, symmetrizable at all vertices by 
Lemma~\ref{lem:sym-stable}.  

\subsubsection{Weakly symmetric but not symmetrizable}

This is possible only if $\cR_0$ contains an object $v$ having nonreflectable $y\in X$. Thus, the Cartan matrix should have form

$$
A=\begin{pmatrix}
a_{xx} & a_{xy}\\
0 & a_{yy}\end{pmatrix},
$$
with $a_{xy}\ne 0$. Since $y$ is nonreflectable, $a_{yy}=0$ and $p(y)=0$.

(a) Let $a_{xx}=0$ so $p(x)=1$ since $x$ is reflectable. Then
$$
A=\begin{pmatrix}
0 & a_{xy}\\
0 & 0\end{pmatrix},
$$
that, after the reflexion, will become
$$
A'=\begin{pmatrix}
0 & -a_{xy}\\
0 & 0\end{pmatrix}
$$
which is $D$-equivalent to $A$.

(b) $a_{xx}\ne 0$. In this case the Cartan matrix is not changed
and therefore the component is weakly symmetric.

\subsection{The canonical extension of $\cR_0$}

\subsubsection{}
Let $\cG,\cH$ be groupoids. 
A functor $f:\cG\to\cH$ is called
a {\sl fibration} if for any $g\in\cG$ and $\beta:f(g)\to h$ in $\cH$
there exists $\alpha:g\to g'$ in $\cG$ such that $f(\alpha)=\beta$.

Given a fibration $f:\cG\to\cH$ and $h\in\cH$, the fiber of $f$
at $h$, $\cG_h$,  is defined as follows.
\begin{itemize}
\item $\Ob(\cG_h)=\{g\in\cG| f(g)=h\}$.
\item $\Hom_{\cG_h}(g,g')=\{\alpha:g\to g'|f(\alpha)=\id_h\}$.
\end{itemize}
\begin{Rem}
If $f$ is not a fibration, the fiber $\cG_h$ defined as above may change if one 
replaces $\cG$ with an equivalent groupoid. A more invariant notion of fiber has as objects the pairs $(g,\alpha:f(g)\to h)$.
\end{Rem}

\subsubsection{}
\label{sss:wtR}

Let $\cR_0$ be an admissible component of the root groupoid and let
$\fg$ be a root algebra on $\cR_0$. Define the groupoid of symmetries
of $\fg$, $\cG_0$, together with a fibration $\pi:\cG_0\to\cR_0$,
as follows. The groupoids $\cG_0$ and $\cR_0$ have the same objects.
For $\alpha:v\to v'\in\cR_0$, we define $\Hom^\alpha_{\cG_0}(v,v')$,
the set of arrows  $v\to v'$ in $\cG_0$, as the set of isomorphisms
$\fg(v)\to\fg(v')$ extending the isomorphism $\fh(\alpha)$.

The fiber of $\pi$ at $v\in\cR_0$ consists of automorphisms of $\fg(v)$
that are identity on $\fh(v)$. Any such automorphism $a$ preserves the weight spaces, and so it is uniquely given by a collection of 
$\lambda_x\in\bC^*$, $\mu_x$ so that $a(e_x)=\lambda_x e_x$, $a(f_x)=\mu_xf_x$. Since $[e_x,f_x]=a(x)\ne 0$, one necessarily has
$\mu_x=\lambda_x^{-1}$.

Recall that the classifying groupoid of a group $G$
is the groupoid having a single object with the group of
automorphisms $G$.

The discussion above implies that the fiber of $\pi$ at $v$ identifies with the classifying groupoid of the torus $(\bC^*)^X$.

\subsubsection{Canonicity of $\cG_0$}

Let $\fg$ be a root algebra on $\cR_0$. For any $v$ the algebra $\fg(v)$ has a maximal ideal $\fr(v)$ having no intersection with $\fh(v)$.

Thus, $\fg(v)/\fr(v)=\fg^\CG(v)$ for all $v$. Let $\alpha:v\to v'$ be an arrow in $\cR$. Any isomorphism $\fg(v)\to\fg(v')$ extending $\fh(\alpha)$
induces an isomorphism $\fg^\CG(v)\to\fg^\CG(v')$. This leads to a
functor $\cG_0\to\cG_0^\CG$ over $\cR_0$, where $\cG_0^\CG$
denotes (temporarily) the groupoid extension of $\cR_0$ constructed
as in \ref{sss:wtR} with the root algebra $\fg^\CG$. It is an equivalence 
as it induces an equivalence of fibers at any $v\in\cR_0$.

\subsection{Universal root algebra}

\subsubsection{}
In this subsection we will prove the existence of an initial object
in the category of root algebras associated to an admissible component
$\cR_0$ of $\cR$.

Let $\fg$ be a root Lie superalgebra for the component $\cR_0$.
Fix $v\in\cR_0$. The $v$-coordinate system for $\fg$ is a Lie superalgebra epimorphism $\wt\fg(v)\to\fg(v)$. Let $\fk(v)$ be its kernel.

Choose an arrow $\alpha:v'\to v$ in $\cR$ presentable as a composition of reflexions. We denote $\fg'=\fg(v')$ and $\fg=\fg(v)$.
The existence of isomorphism $\fg'\to\fg$
lifting $\alpha$ proves that $\fg_{b'(x)-b'(y)}=0$ for $y\ne x$, so that
$\fk(v)\supset\fs(v)$ where $\fs(v)$ is the ideal of $\wt\fg(v)$ 
generated by $\sum\wt\fg_{b'(x)-b'(y)}(v)$,  the sum being taken
over all $\alpha:v'\to v$ presentable as compositions of reflexions.

Let us verify that the collection 
$\fg^\U=\{\fg^\U(v)=\wt\fg(v)/\fs(v),\ v\in\cR_0\}$ is a root Lie superalgebra.
Note that $\fs(v)\subset \fk(v)$, so one has an obvious surjective 
homomorphisms 
$q:\fg^\U(v)\to\fg(v)$.

We have to define, for each arrow $\alpha:v\to v'$ in $\cR$, an 
isomorphism $\tilde\alpha:\fg^\U(v)\to\fg^\U(v')$ extending 
$\fh(\alpha):\fh\to\fh'$. This is enough to verify separately for 
reflexions, homotheties and tautological arrows. In the case when 
$\alpha$ is a tautological arrow or a homothety, it extends to an 
isomorphism $\tilde\alpha:\wt\fg(v)\to\wt\fg(v')$. Since the homotheties 
and the tautological arrows commute with the reflexions,
$\tilde\alpha$ carries $\fs(v)$ to $\fs(v')$, and this induces
an isomorphism $\fg^\U(v)\to\fg^\U(v')$.
It remains to define, for each reflexion $r_x:v\to v'$ in $\cR$,
an isomorphism $\rho=\tilde r_x:\fg^\U(v)\to\fg^\U(v')$ extending 
$\id_\fh$.

The algebra $\fg^\U(v)$ is generated over $\fh$ by the elements
$e_y$ of weight $b(y)$, $f_z$ of weight $-b(z)$, subject to relations
listed in \ref{sss:half} and factored out by $\fs(v)$. Thus, in order
to construct $\rho$, we have to choose $\rho(e_y)\in\fg^\U_{b(y)}(v')$,
$\rho(f_z)\in\fg^\U_{-b(z)}(v')$, so that $\rho$ vanishes at all
the relations. 

The weight spaces $\fg^\U_{b(y)}(v')$ and $\fg^\U_{-b(y)}(v')$
are one-dimensional by property 3 of
\ref{sss:properties} as the map
$q:\fg^\U(v)\to\fg(v)$ is surjective and the weight spaces 
$\fg_{b(y)}(v')$ and $\fg_{-b(y)}(v')$ are one-dimensional.
We will define arbitrarily $0\ne\rho(e_y)\in\fg^\U_{b(y)}(v')$ 
and choose $\rho(f_y)\in\fg^\U(v')$ so that 
$[\rho(e_y),\rho(f_y)]=a(y)$. The latter is possible 
by~Proposition~\ref{prp:bracket}(2). The rest of the relations say that,
for any composition of reflexions $\gamma:v''\to v$ with 
$v''=(\fh,a'',b'',p'')$, the weight space $\fg^\U_{b''(y)-b''(z)}(v)$
vanishes for all $y\ne z$. 
Now $\rho$ defined as above
yields a homomorphism as
 $\fg^\U_{b''(y)-b''(z)}(v')=0$ by definition of $\fs(v')$.
Thus, we have constructed an algebra homomorphism $\rho:\fg^\U(v)\to\fg^\U(v')$. 
 
Any reflexion has order two, 
so there is also a homomorphism in the opposite direction. Their 
composition preserves weight spaces, so it is invertible. 

This proves that the collection of algebras 
$\fg^\U=\{\wt\fg(v)/\fs(v)\}$ is the initial object
in the category of root algebras based on $\cR_0$.

\begin{dfn}
\label{dfn:universalroot}
The root algebra $\fg^\U=\{\wt\fg(v)/\fs(v)\}$ defined as above
is called {\sl the universal  root Lie superalgebra} defined 
by the component $\cR_0$~\footnote{It was J.~Bernstein who once 
pointed out that factoring out
by the maximal ideal having no intersection with the Cartan may be 
unjustified. The present work is to a large extent outcome of his 
remark.} .
\end{dfn}
The superinvolution $\theta$ of $\wt\fg$ defined in \ref{sss:automorphism} induces an automorphism of the universal root algebra.

\subsubsection{Serre relations}

The classical Serre relations 
$$
(\ad e_x)^{-a_{xy}+1}(e_y)=0,\ 
(\ad f_x)^{-a_{xy}+1}(f_y)=0,
$$
for $x,y\in X$ such that $a_{xx}\ne 0$ are among the most obvious 
relations defining the universal Lie superalgebra. They correspond to
the summand $\wt\fg_{\pm(b'(x)-b'(y))}$ of $\fs(v)$ defined by the
reflexion $r_x:v'\to v$. The ideal $\fs(v)$, however, is usually not
generated by the classical Serre relations.

\subsubsection{}
\label{sss:invariantideals}
Let $\fg^\U=\{\fg^\U(v)\}$ denote the universal root algebra
and let $\fg=\{\fg(v)=\fg^\U(v)/I(v)\}$ be a root algebra.

Any automorphism $\eta\in\Aut_\cR(v)$ lifts to an automorphism of
$\fg^\U(v)$ preserving $I(v)$. 

The converse of this fact also holds; one has the following easy result.
\begin{Lem}
Let $\fg^\U$ be the universal root algebra at a component $\cR_0$,
$v\in\cR_0$. Any $\Aut_\cR(v)$-invariant ideal
$J(v)$ of $\fg^\U(v)$ such that $J(v)\cap\fh=0$ defines a canonical root algebra $\fg$ whose
$v$-component is $\fg(v)=\fg^\U(v)/J(v)$. 
\end{Lem}
\begin{proof}
For any $v'\in\cR_0$
choose an isomorphism $\tilde\gamma:\fg^\U(v)\to\fg^\U(v')$ and set
$J(v')=\tilde\gamma(J(v))$. By invariance of $J(v)$ the ideal $J(v')$ is
independent of the choice of $\tilde\gamma$.
\end{proof}

\begin{rem}
The lemma above implies that a root Lie superalgebra is canonically
determined by any its component $\fg(v)=\wt\fg(v)/I(v)$.
An ideal $I(v)\subset\wt\fg(v)$ defines a root superalgebra iff it 
contains $\fs(v)$ and its image in $\fg^\U(v)$ is 
$\Aut_\cR(v)$-invariant.
\end{rem}

\subsection{A side remark: groupoid extensions}
\label{ss:side}

The groupoid extension $\pi:\cG_0\to\cR_0$ has fibers isomorphic to
classifying spaces of a torus. This very special type of extension
admits a description in terms of gerbes.

For $v\in\cR_0$ and $\gamma:v\to v$ in $\cR_0$ choose a lifting
$\tilde\gamma:v\to v$ in $\cG_0$. This defines an automorphism
of the fiber $(\cG_0)_v$ given by the formula $\alpha\mapsto
\tilde\gamma\circ\alpha\circ\tilde\gamma^{-1}$. The result is independent
of the choice of $\tilde\gamma$ as tori are abelian groups.

The above described action can be encoded into a groupoid extension
$p:\cT\to\cR_0$ that is a group over $\cR_0$: one has a multiplication
$$m:\cT\times_{\cR_0}\cT\to\cT$$
corresponding to the fiberwise multiplication. Finally, 
$\pi:\cG_0\to\cR_0$ is a $\cT$-torsor: there is an action
$$\cT\times_{\cR_0}\cG_0\to\cG_0.$$

In more classical terms, we are talking about presenting an abelian
group extension as a torsor over a split abelian group extension
that is a semidirect product of the base and the fiber.

The group $p:\cT\to\cR_0$ is easy to describe. The groupoid $\cR_0$
comes with the functor $\fh:\cR_0\to\Vect$.

We define a functor $T:\cR_0\to\Gp$ into the category of groups
assigning to $v$ the factor group $T(v)=\fh(v)/K(v)$ where 

$$K(v)=\{h\in\fh|b(x)(h)\in 2\pi i\Z\textrm{ for all }x\in X\}.$$

The functor $T$ gives rise to a groupoid extension $p:\cT\to\cR_0$
with $\Ob(\cT)=\Ob(\cR_0)$ and $\Hom_\cT(v',v)=\Hom_{\cR_0}(v',v)\times
T(v)$.

The action $\cT\times_{\cR_0}\cG_0\to\cG_0$ is defined 
as follows. Let $\fg=\{\fg(v)\}$ be a root algebra based on $\cR_0$. 
To $(\alpha,\tau)\in\Hom_{\cR}(v',v)\times T(v)$ and
$\tilde\alpha:\fg(v')\to\fg(v)$, we assign
$\tau\circ\tilde\alpha$ where $\tau:\fg(v)\to\fg(v)$ is given by rescaling.

Note that the torsor $\cG_0$ is nontrivial as, for instance, 
for $\fg=\fsl_2$ the groupoid extension $\pi:\cG_0\to\cR_0$
is the projection $N(T)\to W$ of the normalizer of the torus to the Weyl group that is not split.

\section{Weyl group}

Throughout this section we assume that $\cR_0$ is an admissible
component of $\cR$.

\subsection{Real roots}
For $v\in\cR_0$ we denote
$$Q(v)=\Span_\bZ\{b(x)\}_{x\in X}\subset \fh^*(v),$$

The parity function $p:X\to\bZ_2$ extends to a group homomorphism
$p:Q(v)\to\bZ_2$ that we denote by the same letter $p$.

\begin{lem}
\begin{itemize}
\item[1.]
For any $\gamma:v\to v'$ the isomorphisms $\fh(v)\to\fh(v')$ and 
$\fh^*(v)\to\fh^*(v')$ induce isomorphisms 
$\Span_\bC\{a(x)\}_{x\in X}\to\Span_\bC\{a'(x)\}_{x\in X}$
and $Q(v)\to Q(v')$.
\item[2.] The isomorphisms $Q(v)\to Q(v')$ are compatible with the parity
$p$.
\end{itemize}
\end{lem}
\begin{proof}
The claim directly follows from the formulas for reflexions.
\end{proof}

\begin{dfn}\label{dfn:real}
An element $\alpha\in Q(v)$ is called a real root if there exists
$\gamma:v'\to v$ and $x\in X$ so that $\gamma(b'(x))=\alpha$.
\end{dfn}

\subsubsection{}
\label{sss:realinall}
The collection of real roots in $\fh(v)$ is denoted by
$\Delta^\re(v)$. By~\ref{corgalpha}, for any root algebra $\fg$,
$\Delta^\re(v)\subset\Delta(v)$ and all real root spaces of $\fg$ are one-dimensional. Real roots coming as described above from 
$\gamma:v\to v'$ form a subset $\Sigma_{\gamma}(v)$. 
We write $\Sigma(v)=\Sigma_\id(v)$ for the set
of simple roots at $v$.

Clearly 
\begin{equation}
\label{eq:rroots-bigunion}
\Delta^\re(v)=\bigcup_{\gamma:v\to v'}\Sigma_\gamma(v),
\end{equation}
but the union is not disjoint. Any $\alpha:v\to v'$
sends bijectively $\Delta^\re(v)$ to $\Delta^\re(v')$ and 
$\Sigma_{\gamma\circ\alpha}(v)$ to $\Sigma_\gamma(v')$.

\subsection{Isotropic, anisotropic and nonreflectable real roots}
\label{ssisoaniso}
\begin{dfn}
\begin{itemize}
\item[1.]
A simple root $b(x)\in\fh^*(v)$ is called isotropic if $x$ is reflectable at 
$v$ and $\langle a(x),b(x)\rangle=0$. {\sl One has always $p(x)=1$ for an 
isotropic root $b(x)$.}
\item[2.]A simple root $b(x)\in\fh^*(v)$ is called anisotropic 
if $x$ is reflectable at $v$ and $\langle a(x),b(x)\rangle\ne 0$. 
\item[3.]For an anisotropic simple root $\alpha=b(x)$ we define 
$\alpha^\vee=\frac{2a(x)}{a_{xx}}\in\fh(v)$. 
\end{itemize}
\end{dfn}

We are going to extend these definitions to real roots. 
Since a real root at $v$ is defined by a path $\gamma:v\to v'$
and a simple root at $v'$, the extension is possible if two simple roots
at $v'$ and $v''$ defining the same real root, are of the same type.

\begin{prp}
\label{prp:rroots-class}
Let $\alpha\in\Sigma_{\gamma_1}(v)\cap\Sigma_{\gamma_2}(v)$ so that
$\alpha=\gamma_1^*(b_1(x_1))=\gamma_2^*(b_2(x_2))$
for $\gamma_i:v\to v_i$. Then one of the
following options holds.
\begin{itemize}
\item[1.] Both $b_i(x_i)\in\fh^*(v_i)$ are isotropic roots.
\item[2.] Both $b_i(x_i)\in\fh^*(v_i)$ are anisotropic roots
and $(\gamma_2\circ\gamma_1^{-1})^*(b_1(x_1)^\vee)=b_2(x_2)^\vee$.
\item[3.] $x_1$ is nonreflectable at $v_1$ and $x_2$ is nonreflectable
ay $v_2$.
\end{itemize}
\end{prp}
\begin{proof}
We can assume, without loss of generality, that $\gamma_1=\id_v$
and $\gamma_2=\gamma:v\to v'$. Then $\alpha=b(x)=\gamma^*(b'(y))$.

Let $\fg$ be a root algebra and let $\alpha=b(x)$ for $v\in\cR_0$ so that $x$ is $v$-reflectable. Then $\fg\langle\alpha\rangle$ is not the Heisenberg algebra and $e_x$ acts locally nilpotently on $\fg$. If, for
$\gamma:v\to v'$, $\alpha=\gamma^*(b'(y))$, $e'_y$  acts locally 
nilpotently on $\fg(v')$, and, since $\fg\langle\alpha\rangle$ is not the Heisenberg algebra, this implies that $y$ is reflectable at $v'$. Let now $x$ be reflectable at $v$ and $y$ reflectable at $v'$.
Then Corollary~\ref{corgalpha} describes possible options for $\fg\langle\alpha\rangle$.
This implies the claim.
\end{proof}

\subsubsection{}
\label{sss:reflection}

Proposition~\ref{prp:rroots-class} allows one to extend the classification of simple roots to all real roots. 

One has a decomposition
\begin{equation}
\label{eq:reunion}
\Delta^\re(v)=\Delta_\iso(v)\sqcup\Delta_\an(v)
\sqcup\Delta_\nr(v),
\end{equation}
where
\begin{itemize}
\item[] $\Delta_\iso(v)$ is the set of isotropic real roots that 
are reflectable simple roots at some $v'\in\cR_0$.
\item[] $\Delta_\an(v)$ is the set of anisotropic real roots that
are reflectable simple roots at some $v'\in\cR_0$. Any anisotropic
real root $\alpha\in\Delta_\an(v)$ defines a coroot 
$\alpha^\vee\in\fh(v)$.
\item[] $\Delta_\nr(v)$ is the set of non-reflectable real roots,
those that for any $v'\in\cR_0$ and $x\in X$ such that $\alpha=b(x)$,
$x$ is non-reflectable at $v$.
\end{itemize}

{\begin{Rem}In our definition isotropic roots are necessarily real. 
In another tradition, a root of a symmetrizable Lie superalgebra is
called isotropic if it has length zero. For the real roots both notions 
of isotropicity coincide. 
\end{Rem}}

For $\alpha\in\Delta_\an(v)$ the pair $(\alpha,\alpha^\vee)$ defines 
a reflection $s_\alpha$
acting both on $\fh(v)$ and on $\fh^*(v)$ by the usual formulas
\begin{equation}
\label{eq:salpha}
s_\alpha(\beta)=\beta-\langle \beta,\alpha^\vee\rangle\alpha,\
s_\alpha(h)=h-\langle \alpha,h\rangle\alpha^\vee.
\end{equation}

\begin{crl}
\label{crl:deltare-w}
\begin{itemize}
\item[1.] The set of real roots $\Delta^\re(v)\subset\fh^*(v)$ is 
$\Aut_\cR(v)$-invariant.
\item[2.] For $\gamma\in\Aut_\cR(v)$ and 
$\alpha\in\Delta_\an(v)$ one has
\begin{equation}
s_{\gamma(\alpha)}=\gamma s_\alpha \gamma^{-1}.
\end{equation}
\end{itemize}
\end{crl}
\begin{proof}
The first claim is a direct consequence of formula (\ref{eq:rroots-bigunion})
and \ref{prp:rroots-class}. The second claim directly follows from the formulas for $s_\alpha$.
\end{proof}

\subsubsection{Skeleton}
\label{sss:skeleton}

We define $\Sk\subset\cR$ as the subgroupoid having the same objects
as $\cR$; an arrow $\gamma:v\to v'$ is in $\Sk$ if it can be presented
as a composition of reflexions. This is {\sl the skeleton groupoid}.

We denote by $\Sk(v)$ the connected component of the skeleton containing
$v$. Note that, by definition, any arrow in $\Sk(v)$ induces the identity
map of $\fh(v)$, so any two arrows with the same ends coincide.
Therefore, $\Sk(v)$ is a contractible groupoid. Note that any arrow
$\gamma:v\to v'$ in $\cR$ can be decomposed $\gamma=\gamma''\circ\gamma'$
where $\gamma'$ is in $\Sk$ and $\gamma''$ is a composition of a 
homothety and a tautological arrow.

\begin{rem}
\label{rem:uniqueness}
As we prove later in~\ref{crl:unique-in-sk}, this decomposition is
unique.
\end{rem}
\subsubsection{}
If $\beta:v\to v'$ is a homothety or a tautological arrow, 
$\beta(\Sigma(v'))=\Sigma(v)$. Therefore, for 
$\gamma=\gamma''\circ\gamma'$ as above, 
$\Sigma_\gamma(v)=\Sigma_{\gamma''}(v)$. 
Since $\Sk(v)$ is contractible, it makes sense to denote
$\Sigma_{v'}(v)=\Sigma_\gamma(v)$ for $\gamma:v\to v'$ in $\Sk(v)$.

Thus, we have

\begin{equation}
\Delta^\re(v)=\bigcup_{v'\in\Sk(v)}\Sigma_{v'}(v)
\end{equation}
(the union still does not have to be disjoint).

\subsubsection{Spine}
\label{sss:spine}

We denote by $\Sp$ the subgroupoid of $\Sk$ spanned by the isotropic reflexions only. The component of $\Sp$ containing $v$ is denoted by 
$\Sp(v)$. It is obviously contractible. Cartan data of $\Sp(v)$ describe all
possible Cartan data for the component $\cR_0$ of $\cR$ containing $v$,
up to $D$-equivalence.

\subsection{Weyl group and its actions}
\label{ss:weyl}

In this subsection we define the Weyl group assigned to a component 
$\cR_0$. By definition, the Weyl group identifies with a subgroup of 
$\GL(\fh(v))$, for every $v$. Any arrow $\gamma:v\to v'$
defines an isomorphism of the Weyl groups at $v$ and at $v'$.

We also define an action of $W(v)$ on $\Sk(v)$.

\begin{rem}
\label{rem:attainable}
Every vertex $v'\in\Sk(v)$ determines a Borel subalgebra $\fb_{v'}$ of
the root algebra $\fg(v)$. 
The objects of $\Sk(v)$ classify the {\sl attainable} Borel
subalgebras $\fb'$, that is those containing a given Cartan subalgebra $\fh(v)$ 
and such that $\codim_{\fb(v)}(\fb'\cap\fb(v))<\infty$. This follows
from~\ref{crl:unique-in-sk}.
\end{rem}

\begin{dfn}
\label{dfn:weylgroup}
The Weyl group $W=W(v)$ (at $v\in\cR$) is the group of automorphisms
of $\fh(v)$ generated by the reflections with respect to anisotropic
real roots.
\end{dfn}

\subsubsection{Embedding $i:W(v)\to\Aut_\cR(v)$}
\label{sss:weyltoaut}

The representation of $\Aut_\cR(v)$ in $\fh=\fh(v)$ is faithful
by definition of $\cR$. Let us show that $W(v)$ is a subgroup of
the image of $\Aut_\cR(v)$ in $\GL(\fh(v))$. Let $\alpha=b'(x)$ be an anisotropic root. Without loss of generality we can assume that there is an arrow $\gamma:v\to v'$ in $\Sk(v)$. Then the composition

$$\gamma^{-1}\circ t_{s_\alpha}\circ r_x\circ\gamma:v\to v$$
induces the reflection $s_\alpha$ on $\fh$. This proves that 
generators of $W(v)$ are in the image of the embedding $\Aut_\cR(v)\to\GL(\fh(v))$, so that the Weyl group identifies with a subgroup of 
$\Aut_\cR(v)$.

It is clear that any arrow $\gamma:v\to v'$ intertwines the 
canonical embeddings $W(v)\to\Aut(v)$ and $W(v')\to\Aut(v')$.

Note that $\Aut_\cR(v)$ acts on $W(v)$ so that the embedding
$i$ commutes with this action. This means that $W(v)$ is a normal subgroup of $\Aut_\cR(v)$.

\begin{lem}
\label{lem:step1}
Let $r_x:v\to v'=(\fh,a',b',p)$ be an anisotropic reflexion, 
$\alpha=b(x)\in\fh^*$.
Then $s_\alpha(a(y))=a'(y)$ and $s_\alpha(b(y))=b'(y)$
 for all $y\in X$.
\end{lem}
\begin{proof}
Immediate from the formulas~\ref{sss:reflexionformulas}
and~(\ref{eq:salpha}).
\end{proof}

\begin{lem}
\label{lem:step2}
Let $r_x: v\to v'=(\fh,a',b',p')$ and 
$r_x:w=(\fh,a_w,b_w,p_w)\to w'=(\fh,a'_w,b'_w,p'_w)$ be reflexions. Let 
$\alpha\in\Delta^\re$ satisfy the conditions 
\begin{equation}
\label{eq:v-to-w}
 s_\alpha(a(y))=a_w(y),\ s_\alpha(b(y))=b_w(y),\ p(y)=p_w(y),\ y\in X. 
\end{equation}
Then 
\begin{equation}
\label{eq:v'-to-w'}
 s_\alpha(a'(y))=a'_w(y),\ s_\alpha(b'(y))=b'_w(y),\ p'(y)=p'_w(y),\ 
y\in X. 
\end{equation}
\end{lem}
\begin{proof}
The automorphism $s_\alpha$ carries the basis $\{b(y)\}$ of $Q(v)$
to the basis $\{b_w(y)\}$ of $Q(w)$. The Cartan matrices at $v$ and $w$ 
coincide and the formulas defining $r_x$ are the same. 

\end{proof}

\begin{rem}
Note that if (\ref{eq:v-to-w}) holds then
$x$ is reflectable 
at $v$ if and only if it is reflectable at $w$. This is so as the 
Cartan matrices of $v$ and of $w$ coincide.
\end{rem}

\begin{prp}
\label{prp:WSigma}
Let $w\in W(v)$, $v'=(\fh,a',b',p')\in\Sk(v)$. Then there exists a unique 
$v''=(\fh,a'',b'',p')\in\Sk(v)$ such that
\begin{equation}
\label{eq:v'-to-w'}
 w(a'(y))=a''(y),\ w(b'(y))=b''(y),\ y\in X. 
\end{equation}
\end{prp}
The proposition defines an action of the Weyl group $W$ on 
$\Sk(v)$.
\begin{proof}
The uniqueness claim is obvious.
For the existence, it is sufficient to verify the claim for  
$w=s_\alpha$. We can assume that $\alpha=b(x)$ is a simple root at $v$
and let $r_x:v\to u$ be the reflexion. If $v'=v$ then $v''=u$ satisfies
the requirements by 
Lemma~\ref{lem:step1}. Otherwise, choose
an isomorphism $\phi:v\to v'$, present it as a composition
$\phi=\phi_n\circ\ldots\circ\phi_1$, where each $\phi_i$ is a reflexion. 
We define an arrow
$\psi:u\to v''$ as the composition $\psi=\psi_n\circ\ldots\circ\psi_1$ 
where $\psi_i=r_y$ if $\phi_i=r_y$~\footnote{Note that 
$\psi_i$ and $\phi_i$ are {\sl namesakes}: they have the same name
but are applied to different objects of the groupoid.}. 
Note that  the composition
$\psi$ necessarily makes sense. Now a consecutive application 
of Lemma~\ref{lem:step2} yields the result.
\end{proof}

\begin{rem}
\label{rem:explicit}
The proof provides us with an explicit formula: Let $\alpha=b_v(x)$. Then 
$v''=s_\alpha(v')$ 
is the target of the composition $\psi\circ r_x\circ\phi^{-1}:v'\to v''$, see the picture below.
\end{rem}

\begin{equation}
\label{eq:pic-salpha}
\xymatrix{
&\overset{v}{\bullet}\ar_{r_x}^{\alpha=b_v(x)}[dd]\ar^{\phi_1=r_{y_1}}[rr]&&\bullet
&\dots&\bullet\ar^{\phi_n=r_{y_n}}[rr]&&\overset{v'}{\bullet}\ar@/^1pc/
@[red]^{\color{red}{u'=s_\alpha(v')}}[dd] \\
&&&&&&\\
&\overset{u}{\bullet}\ar^{\psi_1=r_{y_1}}@{-->}[rr]&&\bullet
&\dots&\bullet\ar^{\phi_n=r_{y_n}}@{-->}[rr]&&\overset{v''}{\bullet} 
}
\end{equation}

The embedding $i:W(v)\to\Aut_\cR(v)$ can be easily expressed in terms 
of the action of $W$ on $\Sk(v)$.

\begin{crl}
\label{crl:w}
For any $w\in W(v)$ let $\gamma_w:v\to w(v)$ be the arrow in $\Sk(v)$.
Then 
$$
i(w)=t_w\circ\gamma_w.
$$
\end{crl}
\begin{proof}
The composition $t_w\circ\gamma_w$ is an endomorphism of $v$.
The automorphism $i(w)$ is uniquely defined by its action on $\fh$.
The composition $t_w\circ\gamma_w$ provides the same action.
\end{proof}

We will show later (see~\ref{crl:Wfree}) that the action of the 
Weyl group $W(v)$ on $\Sk(v)$ is free. It is not transitive in general.
Here is what we can say about the orbits of the action.

\begin{prp}
\label{prp:decomposition0}
For every $v,\ v'\in\Sk(v)$ there exists $w\in W(v)$ and a sequence of
isotropic reflexions
$$
v\stackrel{r_{x_1}}{\to}\ldots\stackrel{r_{x_k}}{\to}v''
$$
such that $v'=w(v'')$. In other words, there exists $w\in W(v)$ and 
$v''\in\Sp(v)$ so that $v'=w(v'')$.
\end{prp}
\begin{proof}
Choose a presentation of $\phi:v\to v'$ as a composition
$\phi=\phi_n\circ\ldots\circ\phi_1$ of reflexions. 
If $i$ is the first index for which $\phi_i$ is an anisotropic 
reflexion, we can, as in the proof of Proposition~\ref{prp:WSigma}, 
erase it, replacing reflexions $\phi_j$, $j>i$ with their namesakes
$\psi_j$, so that the target of the composition
$$
\psi_n\circ\ldots\circ\psi_{i+1}\circ\phi_{i-1}\circ\ldots\circ\psi_1:
v\to v''
$$
satisfies the property $s_\alpha(v'')=v'$, for an anisotropic root
$\alpha$ defined by $\phi_i$. Continuing parsing the
decomposition of $\phi$  in this way, 
we end up with the required decomposition.
\end{proof}

\subsubsection{Principal reflections}\label{sss:principal}

In the case $p(x)=0$ for all $x$ and for all $v\in\cR_0$, the Weyl group 
$W$ is known to be generated by simple reflections $s_{b(x)}, x\in X$
for a fixed vertex $v\in\cR_0$. This is not true in general, as, for instance, there may exist $v\in\cR_0$ for which all $a_{xx}=0$.

Here is what can be said in general.

\begin{dfn}
Fix $v\in\cR_0$. A root $\alpha\in\Delta_\an(v)$, 
is called $v$-principal if there exists $v'\in\Sp(v)$
and an element $x\in X$ such that $\alpha=b'(x)$.
A reflection $s_\alpha$ with respect to a $v$-principal root is
called a $v$-principal reflection.
\end{dfn}

One has
\begin{prp}\label{prp:generators}
The Weyl group $W(v)$ is generated by $v$-principal reflections.
\end{prp}
\begin{proof}
Let $\alpha\in\Sigma_\gamma(v)$ be anisotropic where $\gamma:v\to v'=(\fh,a',b',p')$
is a composition of reflexions and $\alpha=b'(x)$.  We will prove the 
claim by induction on length of the presentation of $\gamma$ as a composition of reflexions.

If the sequence consists of isotropic reflexions only, $\alpha$ is 
principal and there is nothing to prove. Otherwise there is an 
anisotropic reflexion in the sequence. We denote below by $\phi'$ a 
composition of isotropic reflexions and by $r_y$ the first anisotropic 
reflexion.
$$
v\stackrel{\phi'}{\to}v_1\stackrel{r_y}{\to} v_2\stackrel{\phi}{\to}v'.
$$
Let $v_1=(\fh,a_1,b_1,p_1)$ and $\beta=b_1(y)$. By Proposition~\ref{prp:WSigma}, $s_\beta$ carries $v'$ to 
a vertex $v''$ obtained as the target of a composition of reflexions
$\psi:v_1\to v''$ having the same indices as the components of 
$\phi:v_2\to v'$. We denote $v''=(\fh,a'',b'',p'')$ and we get
$b'(x)=s_\beta(b''(x))$. Therefore, $s_\alpha=s_{b'(x)}=s_{s_\beta(b''(x))}=s_\beta s_{b''(x)}s_\beta$, the last equality
by~\ref{crl:deltare-w}. Now $s_\beta$ is principal and $v''$
has a shorter sequence of reflexions connecting it to $v$.
\end{proof}

\begin{rem}
\label{rem:aniso-w}
The proof of \ref{prp:generators} implies that any root
$\alpha\in\Delta_\an(v)$ is $W$-conjugate to a principal root.
\end{rem}

\subsection{Modules over a root algebra}
\begin{dfn}
Let $\fg:=\fg(v)$ be a root Lie superalgebra supported at $\cR_0$.
A weight $\fg$-module $M$ is, by definition, an $\fg(v)$-module $M$ whose
restriction 
to $\fh$ is semisimple.
\end{dfn}

For a weight $\fg$-module $M$ we denote by $\Omega(M)$ the set of weights of $M$.

We will now define integrable $\fg$-modules. 
\begin{dfn}
Let $\fg=\fg(v)$ be a root Lie superalgebra.
 We say that a weight $\fg$-module $M$ is {\sl integrable}
 if  $\fg_\alpha$ acts locally nilpotently on $M$
for each anisotropic $\alpha\in\Delta^\re$. 
\end{dfn}

Note that the adjoint representation of any root Lie superalgebra
is integrable.

Let $\fg$ be a root Lie superalgebra and let $M$ be an integrable
$\fg$-module. Corollary~\ref{corgalpha} implies that $\Omega(M)$
is $W$-invariant. Moreover, the multiplicities of the weights $\mu$ and $w(\mu)$
coincide.

The adjoint representation of any root Lie superalgebra $\fg$ is integrable. In particular, the set of roots 
$\Delta(\fg)$ of any root algebra is  $W$-invariant.

\section{Coxeter structures}
\label{sec:coxeter}

\subsection{Introduction}
A Coxeter structure on a group $G$ is a set of elements $s_i\in G$
such that $(G,\{s_i\})$ is a Coxeter group.
A Coxeter structure on a group provides its combinatorial description.

In this section we prove that the Weyl group of any admissible component
$\cR_0$ has a Coxeter structure.  A somewhat similar combinatorial description can be given to the components of the root groupoid.

\subsubsection{}
\label{sss:notation-coxeter}

Fix an indecomposable admissible component $\cR_0$ and $v\in\cR_0$. 
In what follows we use the notation of \ref{sss:skeleton}, suppressing 
the parameter $v$ from the notation.  Thus, we will
 write $\fh$ for $\fh(v)$,
$\Sigma$ for $\Sigma(v)$, and, for  $v'\in\Sk(v)$, $\Sigma_{v'}$ for $\Sigma_{v'}(v)$. 
Recall that $\Sigma=\{b(x)\}_{x\in X}$ and $Q=\Span_\bZ(\Sigma_{v'})$
is independent of $v'$.
 We set 
$$ Q^+_{v'}:=\mathbb{Z}_{\geq 0}\Sigma_{v'}\subset Q, \ \
\ Q^+:=Q^+_v.$$

\subsection{Coxeter structure of the Weyl group} 
\label{ss:Wcoxeter}
Fix a vertex $v\in\cR_0$. 
Let $\alpha_1,\dots,\alpha_m$ be the set of $v$-principal roots
and $s_i$ be the reflection $s_{\alpha_i}$. 
The Weyl group $W$ is generated by $s_i$. We say that $w=s_{i_1}\ldots s_{i_l}$ is a reduced decomposition if it has a minimal length. In this case we say
that $\ell(w)=l$ is the length of $w$.

  Let $$C:=\bigcap_{v'\in\Sp(v)}Q^+_{v'}.$$
\begin{lem} Let $\alpha$ be an anisotropic real root.
    \begin{enumerate}
      \item  There is $w\in W$ such that $w(\alpha)$ is $v$-principal.
      \item If $\alpha\in Q^+_{v'}$ for some $v'\in\Sp(v)$ then $\alpha\in C$.
    \item Either $\alpha\in C$ or $\alpha\in -C$.
            \end{enumerate}
 \end{lem}
 \begin{proof} For (1) see~\ref{rem:aniso-w}.   
   To prove (2) we notice that 
   $(Q^+_{v'}\setminus Q^+_v)\cap\Delta^{re}$ consists of isotropic roots.

   Now let us show (3). By (1) and (2) it suffices to check that if $\alpha\in C$ and $s_i$ is a principal reflection then $s_i(\alpha)\in C$ or
   $s_i(\alpha)\in -C$. Indeed, let $v'$ be a vertex such that $\alpha_i\in \Sigma_{v'}$. Then $s_i(\alpha)\in Q^+_{v'}$ unless 
   $\alpha=-\alpha_i$.
   In the latter case $\alpha\in -C$.
    \end{proof}

    Claim (2) of the lemma above  means that 
    $$Q^+_{v'}\cap\Delta_\an=C\cap\Delta_\an.$$
This is the set of positive anisotropic roots (with respect to 
any $v'\in\Sp(v)$).

  \begin{lem}\label{lem:reduced} Let $w=s_{i_1}\dots s_{i_t}$ and let
  $\alpha_i$ be a principal root such that $w(\alpha_i)\in -C$. Then there exists  $j$
      such that $ws_i=s_{i_1}\dots \hat{s}_{i_j}\dots s_{i_t}$.
    \end{lem}
    \begin{proof} Define $\beta_k:=s_{i_{k+1}}\dots s_{i_t}(\alpha_i)$ 
    for $k=0,\ldots,t-1$ and $\beta_t:=\alpha_i$. Since $\beta_t\in C$ and $\beta_0\in-C$ there is $j$ such that
$\beta_j\in C$ and $\beta_{j-1}\in -C$. 
Hence $\beta_{j}=\alpha_{i_j}$. We get $\alpha_{i_j}=u(\alpha_i)$ for $u:=s_{i_{j+1}}\dots s_{i_t}$. Using the formula  $us_{\alpha}u^{-1}=s_{u\alpha}$, see~\ref{crl:deltare-w},
we obtain 
      $$ws_i=s_{i_1}\dots s_{i_{j-1}}(us_iu^{-1})us_i=s_{i_1}\dots \hat{s}_{i_j}\dots s_{i_t}.$$
    \end{proof}
    \begin{crl}\label{crl: exchange} If $w=s_{i_1}\dots s_{i_l}$ is a reduced decomposition and $\alpha_i$ is a principal root then
      \begin{enumerate}
      \item $\ell(ws_i)<\ell(w)=l$ if and only if $w(\alpha_i)\in -C$.
      \item $w(\alpha_{i_l})\in -C$.
      \item If $\ell(ws_i)<\ell(w)$ then for some $j$
        $$s_{i_j}\dots s_{i_l}=s_{i_{j+1}}\dots s_{i_l}s_i.$$
      \end{enumerate}
    \end{crl}
    \begin{proof} See \cite{Kbook}, Lemma 3.11.
      \end{proof}

    \begin{crl} 
    $W$ is a Coxeter group generated by $s_1,\ldots,s_m$.
   In the Coxeter relations $(s_is_j)^m=1$ the possible values of
    $m$ are $2,3,4,6$ or $\infty$.
    \end{crl}
    \begin{proof} See~\cite{B}, Th\'eor\`eme 6.1, \S 1, Ch.~4.
   If $\alpha$ and $\beta$ are principal roots so that 
    $s_1=s_\alpha$ and $s_2=s_\beta$, it is easy to see that the
    union $W'(\alpha)\cup W'(\beta)$, where $W'$ is the subgroup
    of $W$ generated by $s_1$ and $s_2$, is a classical root system of rank $2$. This implies that $m=2,3,4,6$ or $\infty$.
    \end{proof}
     \begin{crl} If  $w(\alpha_i)\in C$ for all $i$ then $w=1$.
    \end{crl}
    \begin{proof} Follows from~\ref{crl: exchange} (2).  
    \end{proof}
    \begin{crl}
    \label{crl:Wfree}
    Let $v'=w(v)\in\Sk(v)$. If $Q^+_v=Q^+_{v'}$ then $w=1$. In particular, the action of $W$ on $\Sk(v)$ is faithful.
    \end{crl}
    \begin{proof}
      If $Q^+_v=Q^+_{v'}$ then $w(Q^+_v)=Q^+_{v}$ and hence  $w(\alpha)\in C$ for all anisotropic $\alpha\in Q^+_v$.
      \end{proof}

\subsubsection{}
We denote by $\Delta^+_\re(v)$ the set of real roots positive at $v$. 
We set $\Delta^+_\an(v)=\Delta_\an\cap\Delta^+_\re(v)$.
\begin{crl}
\label{crl:W-len}
    Let $v'=w(v)\in\Sk(v)$. Then $\ell(w)$ is the cardinality
    of the set $\Delta^+_\an(v)-\Delta^+_\an(v')$.
\end{crl}
\begin{proof}
Let $w=s_{i_1}\ldots s_{i_l}$ be a reduced decomposition. Set 
$\beta_j=s_{i_1}\ldots s_{i_{j-1}}(\alpha_{i_j})$.
Then $v'=s_{\beta_l}\ldots s_{\beta_1}(v)$ and $\Delta^+_\an(v)-\Delta^+_\an(v')=\{\beta_1,\ldots,\beta_l\}$.
\end{proof}

  { 
   \begin{crl}
    \label{crl:Spiff}
   For $v'\in \Sk(v)$ there exists a unique $v''\in\Sp(v)$ and $w\in W$
   such that $v'=w(v'')$.  The cardinality
    of the set $\Delta^+_\an(v)-\Delta^+_\an(v')$  is equal to $\ell(w)$.
   \end{crl}
   \begin{proof}
   The existence of $v'',w$ follows  from
    Proposition~\ref{prp:decomposition0}.
An isotropic reflection does not change the set $\Delta^+_{\an}$, so
 $\Delta^+_{\an}(v'')=\Delta^+_{\an}(v)$ and the required formula for
 $\ell(w)$ follows from~\ref{crl:W-len}.
 For the uniqueness assume that $v'=w_1(v_1)=w_2(v_2)$
 for $v_1,v_2\in \Sp(v)$. Then $w_1^{-1}w_2(v_2)=v_1$
 and $\Delta^+_{\an}(v_1)=\Delta^+_{\an}(v_2)$, so $\ell(w_1^{-1}w_2)=0$.
 Thus $w_1=w_2$ and $v_1=v_2$ as required. \end{proof}
   
   }
    
    \subsection{Skeleton as a graph}
    \label{skeleton_property} 
A structure similar to the Coxeter structure on the Weyl group
exists also on admissible components of the root groupoid. We fix 
$v_0\in\cR$ and study a combinatorial structure of $\Sk(v_0)$. Note that, from the algebraic point of view, $\Sk(v_0)$ is a contractible groupoid, so it may be seen as something lacking any interest. However,
its arrows are compositions of reflexions, and remembering these
reflexions makes a lot of sense. In this subsection we present 
a description of the shortest path length in this graph, similar to
the one given in~\ref{crl:W-len}. {It has a nice application 
to the description of the group $\Aut_\cR(v)$ in \ref{ss:autv}.}
In Section~\ref{sec:coxeter2} we
study a Coxeter property of $\Sk(v)$.

\subsubsection{}
We look at the skeleton $\Sk(v_0)$ as the graph where the reflexions
connect the vertices. Thus, the reflexions are the edges of our graph.
We  color the edges by elements of $\fh^*=\fh(v_0)^*$: 
a reflexion $v\xrightarrow{r_{x}}v'$ is colored by the real root 
$\alpha=-b(x)=b'(x)$. Note that $\Delta^+_\re(v')$ is obtained
from $\Delta^+_\re(v)$ by replacing the (existing) root $-\alpha$
with $\alpha$.

For a path 
$$v_0\xrightarrow{r_{x_1}}v_1\xrightarrow{r_{x_2}}\dots \xrightarrow{r_{x_{t}}}v_{t}=v'$$
colored by the sequence $(\alpha_1,\ldots,\alpha_t)$ we have
\begin{equation}\label{Delta+rev'}
\Delta^+_{\re}(v')=\bigl(\Delta^+_{\re}(v')\cup\{\alpha_i\}_{i=1}^t\bigr)
\setminus \{-\alpha_i\}_{i=1}^t.\end{equation}
In particular, if a path is colored by the sequence $(\alpha_1,\ldots,\alpha_t)$
with $\alpha_t=\alpha_1$, then there exists $i$ such that 
$\alpha_i=-\alpha_1$.

We will start with an obvious remark.
\subsubsection{Remark}
\label{axyyx=0}
Let $v\xrightarrow{r_{x}}v'$  be a reflection.
If $a_{xy}=a_{yx}=0$ and $x\ne y$, then
$a'(y)=a(y)$, $b'(y)=b(y)$ and the $y$th rows (and the $y$th columns)
of the Cartan matrices $A_v,A_{v'}$ are equal.

\begin{lem}
\label{lem:independent}
Given a path $v_0\stackrel{r_x}{\to} v_1\stackrel{r_y}{\to} v_2$ colored by $(\alpha,\beta)$, $\alpha\ne-\beta$, the
following conditions are equivalent.
\begin{itemize}
\item[(1)] $\alpha-\beta\not\in\Delta^\CG$ (the set of roots of
$\fg^\CG$). 
\item[(2)] There exists a path $v_0\stackrel{r_y}{\to} v_3\stackrel{r_x}{\to} v_2$ colored by
$(\beta,\alpha)$.
\end{itemize}
\end{lem}
\begin{proof}
If (2) is fulfilled, both $\alpha$ and $\beta$ are simple roots at $v_2$, so their difference is not a root. Let us prove that
(1) implies (2). We have $\alpha, -\beta\in\Sigma_{v_1}$,
$-\alpha\in\Sigma_{v_0}$ and $\beta\in\Sigma_{v_2}$.
We will denote by $A^i=(a^i_{xy})$ the Cartan matrix at $v_i$
and we will write $a_i(z)$ and $b_i(z)$ instead of $a_{v_i}(z)$ and $b_{v_i}(z)$.

If $\beta$ is anisotropic,  $\langle\alpha,\beta^\vee\rangle=0$ as 
otherwise both $s_{-\beta}(\alpha)=\alpha-\langle\alpha,\beta^\vee\rangle\beta$ and $\alpha$ are roots, which would imply that  
$\alpha-\beta$ is also a root. This implies that $a^1_{xy}=0$. If 
$\beta$ is isotropic, we still have $a^1_{xy}=0$ as otherwise $r_y$
would carry $\alpha=b_1(x)$ to $\alpha-\beta$ that is not a root.
Thus, by admissibility, $a^1_{yx}=0$. Using Remark~\ref{axyyx=0},
we deduce $-\beta\in\Sigma_{v_0}$ and $\alpha\in\Sigma_{v_2}$ so that

$$b_0(x)=-\alpha, b_0(y)=-\beta$$ 
$$b_1(x)=\alpha, b_1(y)=-\beta\textrm{ and } a_1(y)=a_0(y)$$ 
$$b_2(x)=\alpha,\ b_2(y)=\beta\textrm{ and } a_2(x)=a_1(x).$$ 

We will show that $x$ is reflectable at $v_2$, $y$ is reflectable at $v_3$ and $r_y\circ r_x$ carries $v_2$ to $v_0$. This will give the square in $\Sp(v)$ shown in the picture.
$$
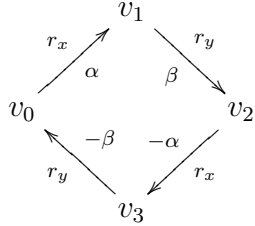
\xymatrix{
& & v_1\ar^{r_y}_\beta[rd] & \\
& v_0\ar^{r_x}_\alpha[ru]&& v_2\ar^{r_x}_{-\alpha}[ld]\\
&& v_3\ar^{r_y}_{-\beta}[lu]  &
}
$$
Reversing the lower reflexions, we get the required result.

Note that reflectability of $x\in X$ at $v$ is determined by the 
$x$-th row of the Cartan matrix at $v$. By~\ref{axyyx=0} the $x$-row of $A^2$ is equal to the $x$-row of $A^1$, so $x$ is reflectable at $v_2$.
Since $b_1(x)=b_2(x)$, $a_1(x)=a_2(x)$  and  the $x$th row 
(resp., $x$th column)
of $A^2$ is equal to the $x$th row (resp., $x$th column) of $A^1$ we have 
$$b_0(z)-b_1(z)=b_3(z)-b_2(z),\ \ \ a_0(z)-a_1(z)=a_3(z)-a_2(z).$$

Once more, by ~\ref{axyyx=0} applied to $r_x:v_2\to v_3$, the $y$ row of $A^3$ is equal to the $y$-row of $A^2$, so $y$ is reflectable at $v_3$. It remains to show that $r_y$ carries $v_3$ to $v_0$.
Since $b_2(y)=b_3(y)$, $a_2(y)=a_3(y)$  and  the $y$th row 
(resp., $y$th column)
of $A^3$ is equal to the $y$th row (resp., $y$th column) of $A^2$, we have 
$$b_1(z)-b_2(z)=b'_0(z)-b_3(z),\ \ \ a_1(z)-a_2(z)=a'_0(z)-a_3(z).$$
Therefore, $b'_0(z)=b_0(z)$ and $a'_0(z)=a_0(z)$. Hence $v'_0=v_0$.

\end{proof}

\begin{lem}
\label{lem:pre-exchange}
Let
$$v_0\xrightarrow{r_{x_1}}v_1\xrightarrow{r_{x_2}}\dots \xrightarrow{r_{x_s}}v_s$$
be a  path in $\Sp(v)$ colored by a  sequence $(\alpha_1,\ldots,\alpha_s)$
with the property   $\alpha_i\not=-\alpha_j$ for $i\not=j$. 
Assume that $\alpha=b_{v_0}(x)=b_{v_s}(y)$ is isotropic. 
Then $\alpha-\alpha_i\not\in\Delta^{\CG}$, $x=y$ and
$b_{v_i}(x)=\alpha$, $a_{v_i}(x)=a_{v_0}(x)$ for all $i$.
\end{lem}
\begin{proof}
Set $\beta:=\alpha-\alpha_1$. Let us show that $\beta$ is not a root.
Assume the contrary. Then  
$\beta$ is even and
$\frac{\beta}{2}$ is not a root. 
Since {the set of even positive roots $\beta$ such that
$\frac{\beta}{2}$ is not a root }
is preserved by isotropic reflexions,
$\beta=\alpha-\alpha_1\in \Delta^+_{v_s}$. Therefore  
$\alpha_1\in -\Delta^+_{v_s}$.
Since $\alpha_1\in \Delta^+_{v_1}$, there should exist $1<i\leq s$ such 
that $\alpha_i=-\alpha_1$, a contradiction.
Since $\beta\not\in\Delta^\CG$, we have
$b_{v_1}(x)=b_{v_0}(x)=\alpha$ and $a_{v_1}(x)=a_{v_0}(x)$. 

Now the assertion follows by induction in $s$.
\end{proof}

The following result describes an exchange property for a sequence of 
isotropic reflexions.

\begin{prp}
\label{prp:spine-com}
Let
$$v_0\xrightarrow{r_{x_1}}v_1\xrightarrow{r_{x_2}}\dots \xrightarrow{r_{x_d}}v_d\xrightarrow{r_{x_{d+1}}}v_{d+1}$$
be a  path in $\Sp(v)$ colored by a  sequence $(\alpha_1,\ldots,\alpha_{d+1})$
with the property $\alpha_{d+1}=-\alpha_1$ and $\alpha_i\not=-\alpha_j$
for $1\leq i<j\leq d$. Then 
$x_{d+1}=x_1$ and  there is a sequence of isotropic reflexions
$$v_0\xrightarrow{r_{x_2}}v'_2\xrightarrow{r_{x_3}}\dots \xrightarrow{r_{x_{d-1}}}v'_{d-1}\xrightarrow{r_{x_d}}v_{d+1}$$
colored by the sequence $(\alpha_2,\ldots,\alpha_d)$.
\end{prp}
\begin{proof}
We apply Lemma~\ref{lem:pre-exchange}
to the sequence of reflexions 
$v_1\xrightarrow{r_{x_2}}\dots \xrightarrow{r_{x_d}}v_d$
and the root $\alpha:=\alpha_1$. We deduce that  
$\alpha_1-\alpha_2\not\in\Delta^\CG$. This implies
that, by Lemma~\ref{lem:independent}, one can replace the sequence $v_0\to v_1\to v_2$ with $v_0\to v_2'\to v_2$ and then a simple induction gives the required result. 
\end{proof}

\begin{rem}
Lemma \ref{lem:pre-exchange} 
implies that for $v,v'$ in $\Sp(v_0)$ we have
$$b_v(x)=b_{v'}(y)\in \Delta_{\iso}\ \ \Longrightarrow\ \ x=y, a_v(x)=a_{v'}(y).$$
In \ref{ss:s21b} below
 we will see that $b_v(x)=b_{v'}(y)\in \Delta_{\an}$ does not imply
neither $x=y$ nor $a_v(x)=a_{v'}(y)$.

\end{rem}

\begin{crl}
\label{crl:unique-in-sk}
Let $v'\in\Sk(v)$ satisfy $\Delta^+_\re(v')=\Delta^+_\re(v)$.
Then $v'=v$. In particular, if a homothety can be presented as a composition of reflexions, it is the identity.
\end{crl}
\begin{proof}
By Proposition~\ref{prp:decomposition0} there exist $v''\in\Sp(v)$ and 
$w\in W$ such that $v'=w(v'')$.
The sets of positive anisotropic roots at $v$ and at $v''$ coincide as none of them can become negative under an isotropic reflexion. Therefore, $w=1$ by \ref{crl:W-len}. This implies that
$v'\in\Sp(v)$. Let
$$v=v_0\xrightarrow{r_{x_1}}v_1\xrightarrow{r_{x_2}}\dots \xrightarrow{r_{x_d}}v_d=v'$$
be a sequence of isotropic reflexions colored by a sequence $(\alpha_1,\ldots,\alpha_d)$. Since $\Delta^+_\re(v')=\Delta^+_\re(v)$,
the formula(\ref{Delta+rev'}) implies $\alpha_i=-\alpha_j$ for some $i,j$. Then by ~\ref{prp:spine-com} the sequence of isotropic reflexions can be shortened.
\end{proof}

\begin{dfn} 
For two vertices $v,v'\in\Sk(v_0)$  the distance $d(v,v')$ is defined
to be the minimal number of reflexions in the decomposition of the arrow
 $v\to v'$.
\end{dfn}

\begin{prp}
\label{prp:Sk-len}
For $v,v'\in\Sk(v_0)$ the distance $d(v,v')$ is the cardinality 
of $\Delta^+_\re(v)-\Delta^+_\re(v')$.
\end{prp}
\begin{proof}
If the difference $\Delta^+_\re(v)-\Delta^+_\re(v')$ is nonempty,
it has an element that is a simple root $\alpha$ at $v$ that can be replaced with $-\alpha$ by a reflection. Continuing this, we can get,
after the required number of steps, a vertex $v''$ having the same 
$\Delta^+_\re(v'')$ as $\Delta^+_\re(v')$. Then by~\ref{crl:unique-in-sk} $v''=v'$.
\end{proof}

Note the following description of non-reflectable roots.

\begin{crl}\label{crl:rereflectable}
$\Delta_\nr=\Delta^\re\setminus(-\Delta^\re)$. 
\end{crl}
\begin{proof}
Obviously, if $\alpha$ is isotropic or anisotropic, 
$-\alpha\in\Delta_\re$. Let us assume that $-\alpha\in\Delta_\re$, 
$\alpha\in\Sigma_v$ and $-\alpha\in\Sigma_{v'}$. By formula
(\ref{Delta+rev'}) any path connecting $v$ with $v'$ contains an edge
where $\alpha$ becomes negative. This proves reflectability of
$\alpha$.
\end{proof}

\subsubsection{Weyl vector}
\label{sss:weylvector}
Choose $\rho_v\in\fh^*$ such that 
\begin{equation}
\label{eq:rho-v}
2\langle\rho_v, a_v(x)\rangle =\langle b_v(x), a_v(x)\rangle
\end{equation}
for all $x\in X$. For each $v'\in \Sk(v)$ we define
$$\rho_{v'}:=\rho_v+\sum_{\alpha\in \Delta^+_{\an}(v')-\Delta^+_{\an}(v)}\alpha
-\sum_{\alpha\in \Delta^+_{\iso}(v')-\Delta^+_{\iso}(v)}\alpha.
$$

Note that the formula~(\ref{eq:rho-v}) holds for all $v'\in\Sk(v)$.

\begin{Rem}
If  $\rho_v=\rho_{v_0}$  and  $v\in \Sp(v_0)$, then  $v=v_0$. 
\end{Rem}
The collection of $\rho_{v'}$, $v'\in\Sk(v)$, is called the Weyl
vector. The choice of $\rho_v$ is not unique. Weyl vectors play an important role in Lie theory.

\subsection{Structure of $\Aut_\cR(v)$ }
\label{ss:autv}

The action of $W(v)$ on $\Sk(v)$ extends to an action of $\Aut_\cR(v)$.

\begin{prp}
\label{prp:Aut-action-skeleton}
There is a unique action of $\Aut_\cR(v)$ on $\Sk(v)$ such that
for any $u\in\Sk(v)$ and $\gamma\in\Aut_\cR(v)$,
$
b_{\gamma(u)}(x)=\gamma(b_u(x)).
$
\end{prp}
\begin{proof}
Uniqueness follows from \ref{crl:unique-in-sk}. It is therefore sufficient to verify that for each $u\in\Sk(v)$ and $\gamma\in\Aut_\cR(v)$ there is $u'\in\Sk(v)$ satisfying the
property $b_{u'}(x)=\gamma(b_u(x))$. We proceed as follows. We present
$\gamma=\gamma''\circ\gamma'$ where $\gamma':v\to v'$ is a composition 
of reflexions and $\gamma''$ is a composition of a homothety with a tautological arrow. Choose a path 
$$
v=v_0\stackrel{r_{x_1}}{\to}\ldots\stackrel{r_{x_k}}{\to}v_k=u
$$
of reflexions connecting $v$ with $u$. Since  the Cartan data at 
$v$ and at $v'$ are $D$-equivalent, there is a namesake path
$$
v'=v'_0\stackrel{r_{x_1}}{\to}\ldots\stackrel{r_{x_k}}{\to}v'_k=u'
$$
defining $u'\in\Sk(v)$. One obviously has $b_{u'}(x)=\gamma(b_u(x))$
which proves the claim.
\end{proof}

\begin{crl}
\label{crl:Aut-action-roots}
The action of $\Aut_\cR(v)$ on $\fh^*$ preserves $\Delta^\re$,
as well as $\Delta_\iso$, $\Delta_\an$, $\Delta_\nr$.
\end{crl}
\begin{proof}
The first claim follows from the formula 
$b_{\gamma(u)}(x)=\gamma(b_u(x))$. The rest follows from the fact
that $u$ and $u'=\gamma(u)$ have $D$-equivalent Cartan data.
\end{proof}

The group $\Aut_\cR(v)$ has a trivial part which we now describe.
\begin{dfn}
An automorphism $\gamma\in\Aut_\cR(v)$ is called irrelevant if it can be presented as a composition of a homothety and a tautological arrow.
\end{dfn}
The group of irrelevant automorphisms identifies with 
\begin{equation}
K(v)=\{\theta:\fh\to\fh|\forall x\in X\ \theta(a(x))\in\bC^*a(x),\theta^*(b(x))=b(x)\}.
\end{equation}

\begin{lem}
\label{lem:Kisnormal}
$K(v)$ is a normal subgroup of $\Aut_\cR(v)$.
\end{lem}
\begin{proof}
$K(v)$ is the kernel of the action of  $\Aut_\cR(v)$ on $\Delta^\re$.
\end{proof}

We can easily describe the image of $\Aut_\cR(v)$ in the automorphisms
of $\Sk(v)$. The description of the action presented above implies
that the automorphism of $\Sk(v)$ defined by $\gamma\in\Aut_\cR(v)$
is uniquely determined by the target $v'$ of $\gamma':v\to v'$ where
$\gamma'$ is the composition of reflexions appearing in the decomposition of $\gamma$. The vertex $v'\in\Sk(v)$ has the Cartan datum
$D$-equivalent to that of $v$. This identifies the image of $\Aut_\cR(v)$ with the set of the vertices on $\Sk(v)$ satisfying this property.

\subsubsection{}
We denote by $\Sk^D(v)$ the subset of (the vertices of) $\Sk(v)$ consisting of the vertices whose Cartan data are $D$-equivalent to that of $v$. The set $\Sk^D(v)$ is endowed with the group structure induced from
the group structure on $\Aut_\cR(v)$. It is combinatorially described 
using ``namesake path'' construction described in the proof of
Proposition~\ref{prp:Aut-action-skeleton}. By construction we have
an isomorphism
\begin{equation}
\label{eq:skd}
\Aut_\cR(v)/K(v)=\Sk^D(v).
\end{equation}

The composition $W(v)\stackrel{i}{\to}\Aut_\cR(v)\to\Sk^D(v)$
is injective as $W(v)\cap K(v)$ is trivial by~\ref{rem:uniqueness}. 

\subsubsection{}
The group $\Sk^D(v)$ has a subgroup $\Sp^D(v)$ defined as the subset
of $\Sk^D(v)$ belonging to $\Sp(v)$. 
The following proposition summarizes what we know about the structure
of the automorphism group.

\begin{prp}
\label{prp:structure-Aut}
\begin{itemize}
\item[1.] $W(v)\subset\Aut_\cR(v)$ is a normal subgroup.
\item[2.] $K(v)\subset \Aut_\cR(v)$ is a normal subgroup.
\item[3.] There is a canonical isomorphism $\Aut_\cR(v)/K(v)=\Sk^D(v)$.
\item[4.] $\Sk^D(v)=W(v)\rtimes\Sp^D(v)$.
\end{itemize}
\end{prp}
 
\begin{proof}Only Claim 4 needs proof.
The intersection $W(v)\cap\Sp^D(v)$ is trivial. Indeed, let $v'=w(v)\in\Sp^D(v)$. Any isotropic reflexion preserves the set of positive 
anisotropic roots, so $\Delta^+_\an(v)=\Delta^+_\an(v')$. Thus, $w=1$
by~\ref{crl:W-len}.

Every automorphism $\phi:v\to v$ decomposes as
$$
v\stackrel{\psi}{\to}v'\stackrel{\eta}{\to}v
$$
where $\psi$ is a composition of reflexions and $\eta$ is a composition
of a homothety with a tautological arrow. By~\ref{prp:decomposition0}
$\psi$ decomposes as $v\stackrel{\rho}{\to}v''\stackrel{\gamma_w}{\to}v'$
where $\rho$ denotes a composition of isotropic reflexions and 
$\gamma_w$ 
is the unique arrow in $\Sk(v)$ connecting 
$v''$ with $v'=w(v'')$. 
The Cartan data of $v'=w(v'')$ and $v''$ are $D$-equivalent 
(actually, the same), so $\Sk^D(v)$ is generated by $W$ and $\Sp^D$.
\end{proof}

\begin{crl}
\label{crl:invariantideals}
Let $\fg^\U$ be the universal root algebra at a component $\cR_0$,
$v\in\cR_0$. An ideal
$J(v)$ of $\fg^\U(v)$ such that $J(v)\cap\fh=0$ defines a root algebra 
$\fg$ having the
$v$-component $\fg(v)=\fg^\U(v)/J(v)$ if and only if it is invariant
with respect to $\Sp^D(v)$. In particular, if $\cR_0$ has no isotropic
reflexions, any ideal of $\fg^\U$ having zero intersection with $\fh$
defines a root algebra.
\end{crl}
\begin{proof}
By~\ref{sss:invariantideals} one has to verify that $J(v)$ is invariant 
with respect to any $\gamma\in\Aut_\cR(v)$. We will verify that any
ideal is invariant with respect to the action of $W(v)$ and of $K(v)$.
The Weyl group is generated by reflections that are inner automorphisms 
by formula (\ref{eq:sigma:gtog}). So, the Weyl group preserves all 
ideals. Any $\gamma\in K(v)$ preserves the weights, so it preserves
the weight spaces. Thus, it multiplies by a constant each $\fg^\U_\alpha$
where $\alpha$ is a simple root or its opposite. Since any root $\beta$
 of $\fg^\U$ is either sum of simple roots or a sum of its opposites, 
 $\gamma$ acts on each $\fg^\U_\beta$ by multiplication by a constant.
Since any ideal of $\fg^\U(v)$ is a sum of its weight subspaces, any 
$\gamma\in K(v)$ preserves it.  Proposition~\ref{prp:structure-Aut} 
now implies the  claim.
\end{proof}

We will see in~\ref{rootalg}  that for all admissible fully reflectable indecomposable components
$\cR_0$, except for $\fgl(1|1)$, any ideal $J(v)$ of $\fg^\U(v)$ having zero intersection 
with $\fh$ is automatically $\Sp^D(v)$-invariant and therefore gives rise to a root algebra.

\begin{crl}
\label{crl:all-different}
Assume that no Cartan data at different vertices of $\Sp(v)$ are 
$D$-equivalent. Then $\Aut_\cR(v)$ is the direct product of the Weyl group $W$ and the subgroup $K$ of irrelevant automorphisms.
If, moreover, the Cartan matrix $A(a,b)$ at $v$ has no zero rows and
$\dim\fh=2|X|-\rk A(a,b)$ is minimal possible,   $K(v)$ is a commutative unipotent group.
\end{crl}
\begin{proof}

Under the assumption, $\Sp^D(v)$ is trivial and so $\Sk^D(v)=W(v)$.
Since $W(v)$ is a normal subgroup of $\Aut_\cR(v)$, one has a direct
decomposition $\Aut_\cR(v)=W(v)\times K(v)$.
\end{proof}

\begin{prp}
\label{prp:finsuper}
Let $\cR_0$ be an admissible component with finite dimensional
$\fg^\CG\ne\fgl(n|n)$. Then $\Aut_\cR(v)=W(v)$.
\end{prp}
\begin{proof}
By~\cite{S1} the conditions of Corollary~\ref{crl:all-different} are fulfilled. The rest follows from triviality of the group $K$.
\end{proof}

Note that for $\fg^\CG=\fgl(n|n)$ one has $\Aut_\cR(v)=W(v)\rtimes \bZ_2$, see~\ref{sss:glmn}.
\subsubsection{Example}
\label{sss:gl12}
Look at the root datum containing the root algebra $\fgl(1|2)$.
Here $X=\{1,2\}$, 
$\fh=\Span\{e,h_1,h_2\}$ and $\fh^*=\Span\{\epsilon,\delta_1,\delta_2\}$
(the dual basis), the spine $\Sp(v_0)$  has three vertices
\begin{itemize}
\item[$v_0$:] $a(1)=-e-h_1$, $a(2)=h_1-h_2$, $b(1)=\epsilon-\delta_1$, $b(2)=\delta_1-\delta_2$, $p(1)=1$, $p(2)=0$;
\item[$v_1$:] $a(1)=e+h_1$, $a(2)=-e-h_2$, $b(1)=\delta_1-\epsilon$, $b(2)=\epsilon-\delta_2$, $p(1)=p(2)=1$;
\item[$v_2$:] $a(1)=h_1-h_2$, $a(2)=e+h_2$, $b(1)=\delta_1-\delta_2$, $b(2)=\delta_2-\varepsilon$, $p(1)=0$, $p(2)=1$.
\end{itemize}

The Weyl group $W(v_0)$ has two elements, with the nonunit interchanging 
$\delta_1$ with $\delta_2$. The group $\Aut_\cR(v_0)$ coincides with
$W(v_0)$ by~\ref{prp:finsuper}.

\section{The Coxeter property of the skeleton}
\label{sec:coxeter2}

In this section we define Coxeter graphs and prove that the skeleton $\Sk(v)$ satisfies this
property. The notion of Coxeter graph generalizes that of a Coxeter group. The Cayley graph
of a group $G$ with respect to a set of generators $S=\{s_i\}$ is Coxeter iff $(G,S)$
is a Coxeter group. There are, however, Coxeter graphs that do not come from Coxeter groups.
It is an interesting question to describe all finite Coxeter graphs.

\subsection{Coxeter graphs}

Let $X$ be a finite set, $G$ a graph with the set of vertices $V$ and the set of edges $E$,
endowed with a marking $r:E\to X$.  We assume that $G$ is connected 
and that the edges having a common end, have different markings. We denote by $r_x:v\to v'$ the edge connecting $v$ and $v'$ marked with $x$. By the assumption, for a chosen $v$ such edge is unique, if exists.
Note that $r_x$ comes with a choice of direction for the edge connecting $v$ and $v'$. 

A path $\phi:v\to v'$ consists of a sequence of arrows
$$
v=v_0\stackrel{r_{x_1}}{\to}\dots\stackrel{r_{x_n}}{\to}v_n=v'.
$$
We denote $\ell(\phi)=n$ the length of $\phi$.

The path $\phi^{-1}:v'\to v$ is obtained from $\phi$ by changing the direction of all 
arrows.

\begin{dfn}
A Coxeter loop $\phi:v\to v$ is one of the following.
\begin{itemize}
\item[1.] $\phi=r_x^2$ (These are called the trivial loops.)
\item[2.] $\phi=(r_y\circ r_x)^m$.
(These are called the loops of length $2m$).
\end{itemize}
\end{dfn}

\begin{dfn}
Let $\phi,\psi:v\to v'$ be a pair of paths. If the concatenation $\psi^{-1}\circ\phi$
is a Coxeter loop,  we will say that one has an elementary Coxeter modification $\phi\Rightarrow\psi$.
\end{dfn}

\begin{dfn}
A Coxeter modification from $\phi$ to $\psi$ is a presentation 
$\phi=\phi_1\circ\phi_2\circ\phi_3$, $\psi=\psi_1\circ\psi_2\circ\psi_3$ such that
$\phi_1=\psi_1$, $\phi_3=\psi_3$ and one has an elementary Coxeter modification $\phi_2\Rightarrow\psi_2$.

\end{dfn}

\begin{dfn}
A marked graph $(X,G,r)$ is called Coxeter if any pair of paths from $v$ to $v'$
can be connected by a sequence of Coxeter modifications.
\end{dfn}

\subsubsection{}
As an example, take a group $\Gamma$ generated by a set $S$ of elements with $s^2=1$.
Let $G$ be the corresponding Cayley graph, where the vertices are $g\in\Gamma$, $X=S$, and $g$ and $h$ are connected by the edge  marked by $s$ if $g=hs$.
Then $\Gamma$ is a Coxeter group iff $G$ is a Coxeter graph.

Let $v\in\cR$. We look at the skeleton $\Sk(v)$ as marked graph, with the reflection
$r_x$ marked with $x\in X$. 
Conversely, one has the following easy result.

\begin{prp} Let $(X,G,r)$ be a Coxeter graph such that for any $v\in V$ 
and $x\in X$ there exists an edge $r_x:v\to v'$. Then $(X,G,r)$ is 
the Cayley graph of a Coxeter group if and only if for any pair $x,y\in X$ the length $2m_{xy}$ of $(x,y)$ loop 
$\phi=(r_y\circ r_x)^{m_{xy}}:v\to v$
is independent of $v$.
\end{prp}
\begin{proof}
The necessity of the condition is clear. Define $\Gamma$ as the Coxeter group generated by $s_x,\ x\in X$ subject to the relations 
$(s_xs_y)^{m_{xy}}=1$. The isomorphism of $(X,G,r)$ with the Cayley graph 
of $\Gamma$ is defined by an arbitrary choice of a vertex $v\in V$ and
the assignment of $s_x$ to $r_x$. Coxeterity of the graph implies that
any two paths $v\to v'$ in $G$ define the same image in $\Gamma$.
\end{proof}

Here is our main result.

\begin{thm}
\label{thm:skeleton-coxeter}
\begin{itemize}
\item[1.]$\Sk(v)$ is a Coxeter graph.
\item[2.] Nontrivial Coxeter loops may have length $2m$ where $m=2,3,4$ or $6$.
\end{itemize}
\end{thm}

The proof of the theorem is based on a presentation of the skeleton 
$\Sk(v)$  as the $1$-skeleton of a convex polyhedron. In the following subsection we present basic facts about convex polyhedra. 
In \ref{ss:proof-skeleton-coxeter} we construct a polyhedron having 
$\Sk(v)$ as its $1$-skeleton. This easily implies Theorem~\ref{thm:skeleton-coxeter}.

\begin{rem} Note that in the case when $\cR_0$ is fully reflectable and all reflexions are anisotropic the skeleton $\Sk(v)$ is isomorphic to the
    Cayley graph of the Weyl group.
    \end{rem}
\subsection{Convex polyhedra: generalities}

\subsubsection{Polytopes}
Recall that a polytope $P$ in a real finite dimensional affine space 
$E$ is defined as the convex hull of a finite set of points. 
The dimension of $P$ is, by definition, the dimension of the affine span of $P$.

A polytope $P$ of dimension $n$ has stratification 
$P=P_0\sqcup\ldots\sqcup P_n$,
where $P_n$ is the inerior of $P$ in its affine span and $P_k$ for $k<n$ consists of points $v$ for which the intersection of all supporting hyperplanes at $v$ has dimension $k$.
Thus, $P_0$ is the set of vertices of $P$ and $P$ is the convex hull of $P_0$.

\subsubsection{Polyhedra}

In this paper we use a slightly generalized notion of 
a convex polyhedron. We collect all necessary material here.

\begin{Dfn}
A polyhedron $\cP$ in $E$ is a closed convex set such that any $v\in\cP$
has a neighborhood isomorphic to a neighborhood of a point
of a polytope, where by isomorphism we mean the one given by an 
affine transformation.
\end{Dfn}

The dimension of a polyhedron is the dimension of its affine span.
The stratification of points of a convex polytope extends to
a stratification of a polyhedron: one has
$\cP=\cP_0\sqcup\dots\sqcup\cP_n$
where $\cP_n$ is the interior of $\cP$ in its affine span and
$\cP_k$ consists of the points for which the intersection of all supporting hyperplanes has dimension $k$. In particular, $\cP_0$ is
the set of vertices of $\cP$. This is a discrete subset of $E$, not necessarily finite. Moreover, $\cP$ is in general not a convex hull
of $\cP_0$.

For any $v\in\cP_{n-1}$ there is a unique supporting hyperplane at $v$.
Its intersection with $\cP$ is a face of dimension $n-1$. Each of them is a polyhedron of dimension $n-1$ and their union is $\partial\cP$.

The following notation is used below. A linear hyperplane $H\subset V$
and $v\in E$ define an affine hyperplane $v+H$. The complement $V\setminus H$ consists of two components; their closures are the halfspaces defined by $H$ and denoted by $H^+$ and $H^-$. In the same manner $v+H^+$ denotes
the affine halfspace.

Note that $\cP$ coincides with the intersection of the affine halfspaces 
$v+H^+$ defined by the faces of $\cP$ of maximal dimension. 
 
\begin{dfn}
Let $A$ be the set of supporting hyperplanes $v_\alpha+H_\alpha$ of $\cP$
and let $v+H_\alpha^+$ be the affine halfspaces containing $\cP$. The cone
of $\cP$, $C(\cP)$ is defined as the intersection 
$\cap_{\alpha\in A}H^+_\alpha.$
\end{dfn}

Obviously, if $A_0\subset A$ satisfies the condition 
$\cP=\cap_{\alpha\in A_0}(v_\alpha+H^+_\alpha)$ then
$C(\cP)=\cap_{\alpha\in A_0} H^+_\alpha$. In particular, $C(\cP)$ is
the intersection of the linear halfspaces $H_\alpha$ defined by 
the $(n-1)$-faces of $\cP$.

Note that by definition $C(\cP)$ is a convex cone in $V$ and $\cP$ is invariant under the action of $C(\cP)$: for $\xi\in\cP$ and $\eta\in C(\cP)$ one has $\xi+\eta\in\cP$.

\begin{lem}
\begin{itemize}
\item[1.] If $C(\cP)\ne\{0\}$ then $\partial\cP$ is contractible.
\item[2.] $C(\cP)=\{0\}$ iff $\cP$ is compact.
\item[3.] $\cP$ is compact iff it is a polytope.
\end{itemize}
\end{lem}
\begin{proof}
Choose an interior point $\zeta\in\cP$ and define the projection from  
$\partial\cP$ to the unit sphere $S$ with the center at $\zeta$ by the formula $$\phi(\xi):=(\zeta+\mathbb R^+(\xi-\zeta))\cap S.$$ Since $\cP$ is convex, $\phi$ is injective. From the definition of $\cP$ we see that $\xi\in S$ is not in the image of $\phi$ iff  $\xi\in\zeta- C(\cP)$. Set $U=(\zeta- C(\cP))\cap S$. The restriction of $\phi$ to any $(n-1)$-face 
is a stereographic projection. Since any point of $\cP$ has a neighborhood isomorphic to a neighborhood of a point of a polytope,
 the map $\phi$ is an open embedding and so it defines a homeomorphism of 
$\partial\cP$ with $S\setminus U$.  
If $C(\cP)\ne\{0\}$, $U$ is a nonempty convex subset of $S$, so 
$S\setminus U$ is contractible. This proves Claim 1.

To prove Claim 2, note that the $C(\cP)$ acts on $\cP$: if $c\in C(\cP)$ 
and $p\in\cP$ then $p-c\in\cP$. Therefore, if $C(\cP)\ne\{0\}$, $\cP$ cannot be compact. On the contrary, if $C(\cP)=\{0\}$, $\partial\cP$
is homeomorphic to sphere, so it is compact. $\cP$ is the convex hull 
of its boundary, so it is also compact.

Finally, if $\cP$ is compact
then it is a convex hull of its boundary that is a finite union of
compact polyhedra of smaller dimension. This implies that $\cP$ is the convex hull of the set of its vertices.
\end{proof}

The only result we need in our study of Coxeter property of the skeleton
is the following.
\begin{crl}
\label{crl:h1}
For any polyhedron $\cP$ {of dimension $>2$} one has $H^1(\partial\cP)=0$.
\end{crl}
\qed

\subsection{A polyhedron defined by $\Sk(v)$}
\label{ss:proof-skeleton-coxeter}

Let $\cR_0$ be an admissible component of a root groupoid, $n=|X|$ and $\Sk(v)$ the skeleton.
Let $Q_{\mathbb R}:=Q\otimes_{\mathbb Z}\mathbb R$ and for any vertex $u$ of $\Sk(v)$ set $Q^+_{u,\mathbb R}:=\sum_{\alpha\in\Sigma_u}\mathbb R^+\alpha$.
\begin{lem}
\label{lem:lambda} 
There exists an injective map $\lambda:\Sk(v)\to Q$, $u\mapsto \lambda_u$ such that
  $$\lambda_u-\lambda_{u'}=\sum_{\alpha\in \Delta_{re}^+(u)-\Delta_{re}^+(u')}\alpha.$$
  \end{lem}
  \begin{proof} Choose $\lambda_v=0$, and set $$\lambda_u:=\sum_{\alpha\in \Delta_{re}^+(u)-\Delta_{re}^+(v)}\alpha.$$
    Here we use Corollary~\ref{crl:unique-in-sk} and 
    Proposition~\ref{prp:Sk-len}  to check injectivity of $\lambda$.
  \end{proof}
We define
\begin{equation}
\label{eq:thepolyhedron}
\cP=\bigcap_{u\in\Sk(v)}(\lambda_u- Q^+_{u,\mathbb R})
\end{equation}
and
\begin{equation}
\label{eq:q++}
Q^{++}_\bR=\bigcap_{u\in\Sk(v)}Q^+_{u,\mathbb R}.
\end{equation}

\begin{prp}
\label{prp:ppolyhedron}
$\cP$ is a polyhedron in $Q_\bR$ and $C(\cP)=-Q^{++}_\bR$.
\end{prp}
\begin{proof}
Set $\lambda_v=0$. Let $f$ be the linear function on $Q_\bR$ such that $f(b_x(v))=1$ for all $x\in X$. Denote
$$H_N:=\{\xi\in Q_{\mathbb R}\mid f(\xi)= N\},\ H_N^+:=\{\xi\in Q_{\mathbb R}\mid f(\xi)\geq  N\},
$$
$$\cP_N:=\cP\cap H_N^+,\ \ \Sk_N(v):=\{u\in\Sk(v)\mid f(\lambda_u)\geq N\},\ \ \cQ_N:=H^+_N\cap\bigcap_{u\in \Sk_N(v)}(\lambda_u- Q^+_{u,\mathbb R}).
$$
    The following claims are obvious:
    \begin{enumerate}
      \item $\Sk_N(v)$ is finite (the vertices are in $-Q^+(v)$).
    \item $\cP=\bigcup_{N<0}\cP_{N}$,
    \item $\cP_N\subset\cQ_N$,
      \item $\cQ_N$ is a convex polytope (compact, bounded by finitely many hyperplanes).
      \end{enumerate}
We intend to show that $\cP_N=\cQ_N$ and that the vertices of the
polytope $\cP_N$ belonging to  $H^+_N\setminus H_N$ are precisely 
$\{\lambda_u|f(\lambda_u)>N\}$. This implies that $\cP$ is a polyhedron.
In fact, for $\mu\in\cP$ choose $N$ so that
$f(\mu)>N$. Then $\mu\in\cP_N=\cQ_N$,  so $\mu$ has a neighborhood
that is a neighborhood in a polytope. 

Note that all $\lambda_u$ are vertices of $\cP$ since there is
a hyperplane in $Q_\bR$ intersecting $\cP$ at one point $\lambda_u$.
For the same reason all $\lambda_u$
satisfying $f(\lambda_u)>N$ are vertices of $\cQ_N$. 
In order to show that $\cQ_N=\cP_N$, it is sufficient to verify that
any vertex $\mu$ of $\cQ_N$ belongs to $\cP$.   
The 1-skeleton of $\cQ_N$ is connected, so it is enough to verify that
any edge of $\cQ_N$ connecting $\lambda_u$ with another vertex $\mu$,
belongs to $\cP$. We know all edges of $\cQ_N$ in a neighborhood of 
$\lambda_u$: they are just $b_u(x)$, $x\in X$. If $x$ is reflectable at $u$, there is an arrow $r_x:u\to u'$, and $\mu$ lies
on the segment connecting $\lambda_u$ with $\lambda_{u'}$. If $b_u(x)$ is
non-reflectable, $b_u(x)\in\cQ^{++}$, so $\lambda_u-\bR^+ b_u(x)$ is the
infinite edge of $\cP$ containing $\mu$.

The minus sign in the formula for $C(\cP)$ is due to the minus sign in 
the formula (\ref{eq:thepolyhedron}).
\end{proof}

\begin{lem}
Let $\cP$ be bounded. Then $\cR_0$ is fully reflectable, $\Sk(v)$ is 
finite.
\end{lem}
\begin{proof}
$\Sk(v)$ embeds into the intersection of $\cP$ with a lattice, therefore, it is finite.
If $x\in X$ is not reflectable at $u\in\Sk(v)$, the root $b_u(x)$
belongs to $Q^+_u$, and, therefore, to all $Q^+_{u'},\ u'\in\Sk(v)$.
This contradicts the condition $Q^{++}_\bR=\{0\}$. 
\end{proof}

\

We will now be able to describe the faces of $\cP$.
  Let $Y\subset X$, $|Y|=k$ and $u\in \Sk(v)$. Let $H_Y(u)$ be the affine $k$-plane passing through $\lambda_u$ and spanned by $b_u(y), y\in Y$.
  Set $F_Y(u):=\cP\cap H_Y(u)$. By definition $F_{\emptyset}(u)=\lambda_u$.
 
\begin{lem}\label{lem:faces}
\begin{itemize}
\item[1.] Any $k$-dimensional face of $\cP$ is of the form $F_Y(u)$ 
for a certain $u\in\Sk(v)$ and a $k$-element set $Y\subset X$.
\item[2.] One has
$$F_Y(u)=\bigcap_{u'\in \Sk_Y(u)}(\lambda_{u'}-\sum_{y\in Y}\mathbb R^+b_{u'}(y)),
$$
where  $\Sk_Y(u)$ denotes the connected component of $u\in\Sk(v)$ in the subgraph spanned by the arrows $r_y$ for $y\in Y$.  
\end{itemize}
\end{lem}
 
  \begin{proof} The boundary $\partial\cP$ of $\cP$ by the proof of
  \ref{prp:ppolyhedron} lies in the union of
    hyperplanes $H_Y(u)$ for all $(n-1)$-element subsets $Y$ of $X$. It is clear that $\lambda_{u'}\in F_Y$ if and only if $\lambda_{u'}-\lambda _u\in -\sum_{y\in Y}\mathbb R^+b_{u}(y)$. Note that 
    $\lambda_{u}-\lambda_{u'}=
    \sum_{\alpha\in \Delta_{re}^+(u)-\Delta_{re}^+(u')}\alpha$, so 
    each of
$\alpha\in  \Delta_{re}^+(u)-\Delta_{re}^+(u')$ lies in the 
non-negative span of $b_u(y)$ for $y\in Y$. Consider
 the arrow $u\xrightarrow{\gamma}u'$. Write it as  $\gamma=r_{x_s}\ldots r_{x_1}$ so that $s$ is minimal possible. Let us show that all $x_i\in Y$. 
    Let $\gamma_i=r_{x_i}\dots r_{x_1}$, $\gamma_i:u\to u_i$ and $\beta_i=b_{u_{i-1}}(x_i)$. Choose minimal $i$ such that $x_i\notin Y$.
Then $\beta_i\equiv b_u(x_i)\mod \sum_{y\in Y}\mathbb R b_u(y)$ --- a contradiction.
    That proves (2).
    Now for $k=n-1$ the statement (1) follows since (2) implies that $F_Y(u)$ has codimension $1$. For general $k$ it follows by induction in codimension.
    \end{proof}
    \begin{crl}
    \label{crl:identification} 
The map $\lambda$ as in Lemma~\ref{lem:lambda} establishes a one-to-one correspondence between $\Sk(v)$ and the set of vertices of $\cP$.
Moreover, $\Sk(v)$ identifies with the $1$-skeleton of $\cP$ so that
the reflexions $r_x:u\to u'$ in $\Sk(v)$ identify with the edges
connecting $\lambda_u$ with $\lambda_{u'}$. 
    \end{crl}
    \begin{crl} \label{lem:twodim} 
The two-dimensional face $F_Y(u)$ of $\cP$ defined by a two-element 
subset $Y$ of $X$ is compact iff $\Sk_Y(u)$ is the finite skeleton of a 
rank 2 fully reflectable component. In this case 
$\Sk_Y(u)$ isomorphic to the Cayley graph of the dihedral group $D_m$ where $m=2,3,4$ or $6$.
    \end{crl}
    \begin{proof} The claim immediately follows from Lemma ~\ref{lem:faces}. The allowable values for $m$ result from a well-known classification of rank $2$ fully reflectable components with finite skeleton, see, for example, \cite{S3}.
    \end{proof}

\begin{rem}
 The noncompact face $F_Y(u)$ has a non-compact
contractible boundary homeomorphic to a line. 
It can be isomorphic to the Cayley graph of $D_\infty$, or it might contain one or two infinite rays corresponding to non-reflectable roots.  
\end{rem}
\subsubsection{Proof of Theorem~\ref{thm:skeleton-coxeter}}
By~\ref{crl:identification} $\Sk(v)$ indentifies with the $1$-skeleton
of the polyhedron $\cP$. By~\ref{crl:h1} any pair of paths leading
from $u$ to $u'$ in $\Sk(v)$ is connected by relations defned by
compact $2$-faces. Finally, by~\ref{lem:twodim}, compact $2$-faces 
gives rise to Coxeter relations with $m=2,3,4,6$.

\begin{rem}
The polyhedron $\cP$ appeared firstly in \cite{VGRS} for the finite dimensional Lie superalgebras.
\end{rem}

\section{A trichotomy for admissible fully reflectable components}
\label{sec:trichotomy}
\subsection{Overview}
From now on  we will consider only indecomposable admissible fully reflectable components.

In this section we define three types of such
components: finite, affine and indefinite. We investigate
the structure of the sets of roots of corresponding root algebras.  
Expectedly, the trichotomy for admissible components is closely
connected to the trichotomy for the
types of Cartan matrices defined by Kac in~\cite{Kbook}, Theorem 4.3.
\subsubsection{}
We keep the notation of~\ref{sss:notation-coxeter}.
Fix an indecomposable admissible fully reflectable component $\cR_0$ and $v\in\cR_0$. Let $\fg$ be a root Lie superalgebra supported at $\cR_0$.
 We denote by $\Delta=\Delta(\fg)$ the set of roots of 
$\fg$ and by $\fr$ the kernel of the canonical map $\fg\to\fg^\CG$.
Recall that $\fr$ is the maximal ideal of $\fg$ having zero intersection
with $\fh$.

In this section we will deduce a certain information  abot the ideal  
$\fr$ for different types of components, see~\ref{crlfin}, \ref{corfindim}. In particular, we will be able to deduce, 
for certain types of components, that they admit a unique root Lie superalgebra  $\fg^\CG$.

\subsection{Roots}
Recall that $\Sigma_{v'}=\{b_{v'}(x)\}_{x\in X}$ and  
$Q^+_{v'}:=\mathbb{Z}_{\geq 0}\Sigma_{v'}\subset Q,
\ Q^+:=Q^+_v.$
We have   $\Delta\subset (-Q^+\cup Q^+)$.
Recall~\ref{sss:realinall} that
$$
\Delta^\re=\bigcup_{v'\in\Sk(v)} \Sigma_{v'}\subset\Delta
$$
and the root spaces
$\fg_{\alpha}$, $\alpha\in\Delta^\re$, are one-dimensional, in particular, are purely even or purely odd. This yields a decomposition of the family of real roots into
even and odd part
 
$$\Delta^\re=\Delta^{\re,0}\sqcup\Delta^{\re,1}.
$$
For anisotropic $\alpha\in\Delta^\re$ the elements
$\alpha^\vee\in\fg\langle\alpha\rangle\cap\fh$ are defined so that 
$\langle\alpha,\alpha^\vee\rangle=2$.

We define

$$
\Delta^\im=\{\alpha\in \Delta|\ \mathbb{Q}\alpha\cap\Delta^\re=\emptyset\}.
$$

 For each $v'\in\Sk(v)$ we have 
the triangular decompositions 
$$\Delta=\Delta^+_{v'}\sqcup (-\Delta^+_{v'}),\ \ \text{
where }\Delta^+_{v'}:=\Delta\cap Q^+_{v'}.$$

\begin{prp}\label{crlDeltare}
\begin{enumerate}
\item 
For  $v'\stackrel{r_x}{\to}v''$ with $x\in X$,
let $\alpha=b_{v'}(x)$. One has 
$$\Delta^+_{v''}=\left\{
\begin{array}{ll}
\{-\alpha\}\cup \Delta^+_{v'} \setminus\{\alpha\} \  & \text{ if } 2\alpha\not\in\Delta\\
\{-\alpha,-2\alpha\}\cup \Delta^+_{v'} \setminus\{\alpha,2\alpha\}\ &
\text{ if } 2\alpha\in\Delta.\\
\end{array}\right.$$
\item  
For any $v'$ one has
$\Delta^\im\cap \Delta^+_{v'}=\Delta^\im\cap \Delta^+_{v}$.
\item 
$\Omega(\fr)\subset \Delta^\im$, except for the rank one
algebra $\wt\fg=\fg^\U$ with $\fg^\CG=\fgl(1|1)$, see~\ref{rank1}.
\item
If $\cR_0$ has rank greather than one, then
$$\Delta=\Delta^\re\cup\Delta^\im\cup
\{2\alpha|\ \alpha\in\Delta^{\re,1}\ \text{ is anisotropic}\}.$$
\end{enumerate}
\end{prp}
\begin{proof}
Claim (1) is standard and (2) follows from (1). Claims (3) and (4) follow 
from~\ref{corgalpha}.
\end{proof}

%
%
%
%
%
%
%

\subsection{Types of $\cR_0$}
\subsubsection{The case of Kac--Moody Lie algebras}
In ~\cite{Kbook}, Thm. 4.3 Kac-Moody Lie algebras are divided in three types according to the corresponding type of Cartan matrices as follows. 
Let $V:=\bR\otimes_\bZ Q$; for $v\in V$ we set $v>0$ (resp., $v\geq 0$) if $v=\sum_{\alpha\in\Sigma} k_{\alpha}\alpha$ with $k_{\alpha}\geq 0$
(resp., $k_{\alpha}>0$) for each $\alpha\in\Sigma$.

View an indecomposable Cartan matrix $A$ as a linear operator on $V$.
It is given by the formula
$$
A(v)=\sum_iv(\alpha_i^\vee)\alpha_i,\ v\in V.
$$

 By~\cite{Kbook}, Thm.4.3,
$A$ satisfies exactly one of the following conditions
\begin{itemize}
\item $\exists v>0$ such that $Av>0$ (type (FIN)).
\item $\exists v>0$ such that $Av=0$ (type (AFF)).
\item $\exists v>0$ such that $Av<0$ (type (IND)).
\end{itemize}

Moreover, one has
\begin{itemize}
\item (FIN)  $Au\geq 0$ implies  $u>0$ or $u=0$.
\item (AFF)   $Au\geq 0$ implies  $u\in\mathbb{R}v$.
\item(IND)  $Au\geq 0$ with $u\geq 0$ implies $u=0$.
\end{itemize}

It is proven there that the Kac-Moody Lie algebras of type (FIN) are all simple finite-dimensional Lie algebras, the Kac-Moody Lie algebras of type
(AFF) have finite growth: they are always symmetrizable and can be obtained as (twisted) affinizations of simple finite-dimensional Lie algebras.
The Kac-Moody algebras of indefinite type have infinite growth.

We present below a version of this trichotomy in terms of connected components
of root groupoids. The component is required to be indecomposable and
fully reflectable. Note that both conditions hold in the context
of \cite{Kbook}, Thm. 4.3.

\subsubsection{} Let $\cR_0$ be a component of the root groupoid with a 
fixed vertex $v$ and indecomposable $A(v)$.
Set
$$Q^{++}:=\displaystyle\bigcap_{v'\in\Sk(v)} Q^+_{v'}.$$
{Obviously, $Q^{++}=Q^{++}_\bR\cap Q$.}
Note that the sets $\Delta^\re$ and $Q^{++}$ depend on the component 
$\Sk(v)$ only.
One has $Q^{++}\cap \mathbb{Q}\alpha=0$ for each 
$\alpha\in\Delta^\re$.

In the definition below we introduce three classes of components 
analogous  to the classes (FIN), (AFF), (IND) of Cartan matrices
defined in \cite{Kbook}, Thm. 4.3.

\begin{dfn}
\label{dfn:types}
We say that $\cR_0$ {\em is of type }
\begin{itemize}
\item[(Fin)] if $Q^{++}=\{0\}$.
\item[(Aff)] if $Q^{++}=\mathbb{Z}_{\geq 0}\delta$
for some $\delta\not=0$.
\item[(Ind)] if $\cR_0$ is not of type (Fin) or (Aff).
\end{itemize}
\end{dfn}

%

\subsubsection{Purely anisotropic case}\label{Deltareisoempty} 
Assume that all simple roots $b(x)$ at $v$ are anisotropic. Then
 the Cartan matrices $A(v')$ are the same at all $v'\in\cR_0$. 
Lemma~\ref{WorbitQ+} below shows that in this case the classes (Fin), (Aff) and (Ind) coincide with (FIN), (AFF) and (IND).
Indeed, in this case
$Q^{++}=\bigcap_{w\in W}w(Q^+)$
is the union of $W$-orbits belonging to $Q^+$.
\begin{lem}\label{WorbitQ+}$ $
\begin{itemize}
\item[1.] In the case {\rm(FIN)}  the unique $W$-orbit lying in
$Q^+$ is $\{0\}$.
\item[2.]
In the case {\rm(AFF)} all  $W$-orbits lying in
$Q^+$  are of the form $\{j\delta\}$ for $j\in\mathbb{Z}_{\geq 0}$
for some $\delta\ne 0$.
\item[3.]
In the case {\rm(IND)}  the unique finite $W$-orbit lying in $Q^+$ is 
$\{0\}$; $Q^+$ contains an infinite $W$-orbit.
\end{itemize}
 \end{lem}
\begin{proof}
Notice that $Au\geq 0$ ($Au=0$) for $u\in V\subset \fh^*$ means 
$u(\alpha^{\vee})\geq 0$ (resp., $u(\alpha^{\vee})=0$ for each $\alpha\in\Sigma$.

For $\nu=\sum_{\alpha\in\Sigma}k_{\alpha}\alpha\in Q^+$ set
$\htt \nu:=\sum_{\alpha\in\Sigma}k_{\alpha}$.
 Let $\nu\in Q^+$ be such that $W\nu\in Q^+$
 and $\htt \nu$ is minimal in its orbit.
Viewing $\nu$ as an element of $V$ we have $\nu\geq 0$ and
 $\htt r_{\alpha}\nu\geq \htt\nu$ 
for each $\alpha\in\Sigma$. Then $\nu(\alpha^\vee)\leq 0$ for all
$\alpha\in\Sigma$
and therefore  $A\nu\leq 0$. 
Hence $\nu=0$ in type {\rm(FIN)} and
$\nu$ is proportional to $\delta$ in type {\rm(AFF)}.

In the remaining type {\rm(IND)}, assume $W\nu\subset Q^+$ is finite
and $\htt \nu$ is maximal. Then $\nu(\alpha^\vee)\geq 0$ for all $\alpha$
and, therefore, 
$A\nu\geq 0$. Hence $\nu=0$.  By the assumption there exists  $v>0$ such that $Av<0$. Then  $Wv\subset Q^+$ by~\cite{Kbook}, 
Lemma 5.3 and, by above, this is an infinite orbit.
\end{proof}

\subsubsection{Purely anisotropic components  of finite and affine types}
\label{sss:aniso-fin-aff}
If $p(x)=0$ for each $x$, then $\fg^\CG$ is a Kac-Moody Lie algebra.
In this case $\fg^\CG$ is finite-dimensional if and only if the Cartan matrix $A$ if of type {\rm(FIN)} and a (twisted) affine Lie algebra if $A$ is of type {\rm(AFF)}.

If we do not require all generators to be even, we have an extra 
requirement saying that the $x$-row of  $A$ 
consists of even entries if $p(x)=1$.
Therefore, to every anisotropic component one can associate a Kac-Moody Lie algebra by changing the parity of
all generators to $0$. As we showed in the previous subsection, this operation does not change the type of the corresponding components.
We call all contragredient Lie superalgebras obtained in this way from a Kac-Moody Lie algebra $\fg$ {\sl the cousins of } $\fg$.

The Cartan matrices of types (FIN) and (AFF) are 
well-known. Let us describe the cases when such a matrix has  a row
with even entries.

In the type (FIN) the only such case is the type $B_n$ and it has exactly 
one row with even entries. The Kac-Moody Lie algebra with Cartan matrix 
$B_n$ is $\mathfrak{so}(2n+1)$ and its cousin is  a finite-dimensional 
simple Lie superalgebra $\fosp(1|2n)$.

The affine Kac-Moody Lie algebras whose Cartan matrices have at least one row with even entries are $\mathfrak{so}(2n+1)^{(1)}$,
$\mathfrak{sl}(2n+1)^{(2)}$ and   
$\mathfrak{so}(2n+2)^{2}$. The cousin of $\mathfrak{so}(2n+1)^{(1)}$ is $\mathfrak{sl}(1|2n)^{(2)}$, the cousin of $\mathfrak{sl}(2n+1)^{(2)}$ is
$\mathfrak{osp}(1|2n)^{(1)}$, and $\mathfrak{so}(2n+2)^{(2)}$ has two cousins $\mathfrak{osp}(2|2n)^{(2)}$ and $\mathfrak{sl}(1|2n+1)^{(4)}$, see ~\cite{vdL}
for construction of (twisted) affine superalgebras.

\subsection{Components of type {\rm (Fin)}}
Most of the root Lie superalgebras of finite type have isotropic
roots.
\begin{lem}\label{crlfin}
  Assume that $\cR_0$ is of type {\rm(Fin)}. Then
  \begin{enumerate}
\item  $\Delta^\im=\emptyset$. 
\item $\fg=\fg^\CG$ except for the case $\fg^\CG=\fgl(1|1)$ (see~\ref{rank1}).
\item $\fg$ is finite-dimenisonal.
\end{enumerate}
\end{lem}
\begin{proof} (1) follows from \ref{crlDeltare}(4), (2) and (3) from \ref{crlDeltare} (5).
  \end{proof}

\begin{crl}\label{corfindim}
If $\dim\fg<\infty$ then $\cR_0$ is of type {\rm(Fin)}.
\end{crl}
\begin{proof}
It suffices to check that $\Sk(v)$ contains $v'$ with $\Sigma_{v'}=-\Sigma$
which is equivalent to
 $\Delta^+_{v'}(\fg^\CG)=-\Delta^+_{v}(\fg^\CG)$.
Since $\dim\fg<\infty$, $\Delta(\fg^\CG)$ is finite.
For each $v'\in\Sk(v)$ let $k(v')$ be the cardinality of
$\Delta^+_{v'}(\fg^\CG)\cap \Delta^+_{v}(\fg^\CG)$.
If $k(v')\not=0$, then $\Delta^+_{v'}(\fg^\CG)$ does not lie
in $-\Delta^+_{v}(\fg^\CG)$, so there exists $\alpha\in\Sigma_{v'}$ with
$\alpha\in \Delta^+_{v}(\fg^\CG)$. By~\ref{crlDeltare} (2), 
there is a reflexion $v'\to v''$ that replaces $\alpha$
(and, possibly, $2\alpha$) in $\Delta^+_{v'}$ with $-\alpha$ (and, possibly, $-2\alpha$). This means that
$k(v'')$ is equal to $k(v')-1$ or to $k(v')-2$.
Hence $k(v')=0$ for some $v'\in\Sk(v)$.
\end{proof}

\subsubsection{} The results of C.~Hoyt~\cite{Hoyt}, see~\ref{sss:hoytclass} below, together
with \ref{sss:aniso-fin-aff}, imply that $\fg^{\CG}$ of finite type 
are: $\fgl(1|1)$ and
all basic classical Lie superalgebras (except 
that the simple algebra $\mathfrak{psl}(n|n)$ should be replaced with 
$\fg^\CG=\mathfrak{gl}(n|n)$). In all cases except
$\mathfrak{gl}(1|1)$ we have $\fg^{\CG}=\fg^\U$ by \ref{crlDeltare}(4).

\subsection{Components of type {\rm(Aff)}}  
\begin{lem}\label{crlaff}
  Let $\cR_0$ be of type {\rm(Aff)}. Then
  \begin{enumerate}
\item $\Omega(\fr)\subset\Delta^\im\subset\mathbb{Z}\delta\setminus\{0\}$.
\item $\fr$ lies in the center of $[\fg,\fg]$.
\item  If  $\langle\delta,a(x)\rangle\not=0$ for some $x\in X$ then  $\fg=\fg^{\CG}$.
\end{enumerate}
\end{lem}
\begin{proof}
Using~\ref{crlDeltare} 
we get (1) and $\Omega(\fr)\subset\Delta^{im}\subset\mathbb{Z}\delta\setminus\{0\}$.  

Since $\fg=[\fg,\fg]+\fh$, $\fr$ lies in $[\fg,\fg]$ and $[\fg,\fg]$ 
is generated by $\fg_{\pm\alpha}$ for $\alpha\in\Sigma$. Since $j\delta\pm\alpha\not\in\mathbb{Z}\delta$, 
$[\fg_{\pm\alpha},\fr]=0$. This gives $[[\fg,\fg],\fr]=0$ and establishes (2).
For (3) assume that $\fr\not=0$. Then $\fr\cap\fg_{j\delta}\not=0$ for some $j\not=0$. Hence
 $\fg_{j\delta}$ has a non-zero intesection with the center of
$[\fg,\fg]$. Since $a(x)\in [\fg,\fg]$ for each $x\in X$ this gives
$\langle\delta,a(x)\rangle=0$.
\end{proof}

\subsubsection{Hoyt's classification}
\label{sss:hoytclass}
Indecomposable contragredient Lie superalgebras with at least one simple isotropic root were classified in~\cite{Hoyt}. 
In this subsection we review the results of C.~Hoyt classification
that will be used in the following sections.
Exactly one of the following options holds in this case:
\begin{enumerate}
\item $\dim \fg^\CG<\infty$.
\item $\dim\fg^\CG=\infty$ and $\Delta^{im}=\mathbb Z\delta$, 
$\Delta\subset\mathbb Z\delta+\Delta'$ for some finite set 
$\Delta'\subset \fh^*$ and some $\delta\in\Delta^+$.
Note that, even though $\Delta^+_v$ depends on $v$,  it is positive or negative regardless of the choice of $v\in\cR_0$ as $\delta$ is imaginary. In this case all symmetrizable contragredient Lie 
superalgebras are twisted affinizations of simple
  finite-dimensional Lie superalgebras. They also appear in Van de Leur classification of symmetrizable Kac-Moody superalgebras of finite growth.
  In addition, there is one-parameter  contragredient superalgebra $S(2,1;a)$ and the twisted affinization $\fq(n)^{(2)}$ of the strange superalgebra
  $\mathfrak{psq}(n)$ for $n\geq 3$. By direct inspection one can check that there exists $m\in\mathbb Z$ such that if $\alpha\in\Delta$ then
  $\alpha\pm m\delta\in \Delta$.
\item The algebra $\fg^\CG=Q^{\pm}(m,n,t)$ with $\dim(\fh)=3$ where
  $m,n,t$ are negative integers, not all equal to $-1$, with
  non-symmetrizable and nondegenerate Cartan matrices. There are three
  linearly independent principal roots, therefore the Weyl group has no non-zero fixed vectors in $\fh^*$.
  Hence $Q^{\pm}(m,n,t)$ are of type {\rm(Ind)}. Little is known about Lie superalgebras of this type.  
\end{enumerate}
\subsubsection{}
 Let $\cR_0$ be a  component of $\cR$ of type (2) in Hoyt's classification~\ref{sss:hoytclass}. We will prove that it is of type {\rm(Aff)}. 

\begin{lem}\label{lem_aff-iso} Let $F:=Q_\bR^*$ and $\gamma\in F$ satisfy $\langle\gamma,\delta\rangle=1$ and $\langle\gamma,\beta\rangle\neq 0$ for any $\beta\in\Delta$. Then there exists $v\in\cR_{0}$
  such that $\langle\gamma,\alpha\rangle>0$ for any $\alpha\in\Sigma_v$. 
\end{lem}
\begin{proof} Choose a vertex $u\in\cR_0$. Let
  $$T_u(\gamma)=\{\beta\in\Delta^+_u\mid \langle\gamma,\beta\rangle<0\}.$$
  We claim that $T_u(\gamma)$ is finite. Indeed, since $\delta\in \Delta^+_u$ we have $\alpha+M\delta\in \Delta^+_u$ for sufficiently large $M$ and all $\alpha\in\Delta'$ while $\alpha-M\delta\notin \Delta^+_u$. On the other hand, if
  we choose $$M>\max\{\langle\gamma,\alpha\rangle\mid\alpha\in\Delta'\},$$
  then $\langle\gamma,\alpha+s\delta\rangle>0$ for all $s>M$. Thus,
  $$T_u(\gamma)\subset \{\alpha+s\delta\mid \alpha\in\Delta', -M\leq s\leq M\}$$
  and hence $T_u(\gamma)$ is finite.
  Suppose that $u$ does not satisfy the conditions of the lemma. Then there is $x\in X$ such that $ \langle\gamma,b(x)\rangle<0$. Consider
  $u\stackrel{r_x}{\longrightarrow}u'$.  By Corollary ~\ref{crlDeltare}(2) we get
  $T_{u'}(\gamma)=T_u(\gamma)\setminus\{b(x)\}$ or $T_u(\gamma)\setminus\{b(x),2 b(x)\}$ if $2b(x)$ is a root. Anyway $|T_{u'}(\gamma)|<|T_u(\gamma)|$. Repeating the argument several times, we end up with a vertex $v$ such that
  $T_v(\gamma)=\emptyset$.
  \end{proof}

  \begin{crl}\label{cor1_aff-iso} If $\cR_0$ is of type (2), then $Q^{++}=\mathbb Z_{\geq 0}\delta$ and hence $\cR_0$ is of type {\rm(Aff)}. 
  \end{crl}
  \begin{proof} Let $$F_1:=\{\gamma\in F\mid \langle\gamma,\delta\rangle=1\},\quad S_\gamma^+=\{\nu\in Q\mid \langle\gamma,\delta\rangle\geq 0\}.$$
    Then by Lemma ~\ref{lem_aff-iso}
    $$Q^{++}=\cap_{\gamma\in F_1}S_\gamma^+=\mathbb Z_{\geq 0}\delta.$$
  \end{proof}

\subsection{}
Combining the results of~\cite{Hoyt} with~\ref{Deltareisoempty} 
we obtain the following result.

\begin{prp}
Let $\cR_0$ be an indecomposable fully reflectable component.
\begin{enumerate}
\item
The following conditions are equivalent:
\begin{itemize}
\item
$\cR_0$ of type {\rm(Fin)};
\item $W$ is finite;
\item
  $\dim\fg<\infty$;
\item
$\dim\fg^\CG<\infty$.
\end{itemize}
\item
The following conditions are equivalent:
\begin{itemize}
\item
$\cR_0$ of type {\rm(Aff)};
\item $W$ is infinite and $\fh^*$ contains a non-zero trivial $W$-orbit.
\end{itemize}
\item
The following conditions are equivalent:
\begin{itemize}
\item
$\cR_0$ of type {\rm(Ind)};
\item $\fg$ has an infinite Gelfand-Kirillov dimension. 
\end{itemize}
\end{enumerate}
\end{prp}

\begin{rem}
Cartan matrices of components of type (Fin) are usually nondegenerate. The only exception
is $\fgl(n|n)$. Cartan matrices of type (Aff) are always degenerate, usually of corank one.
The only exception is $\fsl(n|n)^{(1)}$ where corank is two. 
\end{rem}

\section{Symmetrizable root data}
\label{sect:sym}
\label{sectKacThm}

We retain the notation of Section~\ref{sec:trichotomy}. We continue to assume that all $x\in X$ are reflectable at all $v\in\cR_0$. In this section we prove, following a method of Gabber-Kac~\cite{GabberKac},
that if $\cR_0$ has a symmetric Cartan matrix (and, therefore,
 all Cartan matrices associated to $\cR_0$ 
are symmetrizable) then $\fg^\CG$ is the only root algebra, except
for the cases $\fg^\CG=\fgl(1|1)$ and $(\rho|\delta)=0$ where $(-|-)$
is the nondegenerate symmetric bilinear form on $\fh^*$
introduced in~\ref{prp:likekac22} and $\rho$ is as in~\ref{sss:weylvector}.

Fix $v\in\cR_0$, an admissible component of $\cR$. We keep the notation
of Section~\ref{sec:root} for the half-baked algebra 
$\wt\fg=\wt\fn^-\oplus\fh\oplus\wt\fn^+$,
a root algebra $\fg$ and the contragredient algebra $\fg^\CG=\wt\fg/\fr$.
We set $\wt{\fb}:=\wt{\fn}^+ +\fh$, its image $\fb$ in $\fg$ and $\fr^{\pm}:=\fr\cap\tilde{\fn}^{\pm}$. Note  
that $\fr^{\pm}$ are ideals of $\wt{\fg}$.

\subsection{Verma modules}
Let $\wt{M}(\lambda)$ (resp., $M(\lambda)$, $M^\CG(\lambda)$) denote
a Verma module of highest weight $\lambda$ over $\wt{\fg}$ 
(resp., $\fg$,  $\fg^\CG$). Since $\Omega(\wt{M}(\lambda))
\subset\lambda-Q^+$, the module 
 $\wt{M}(\lambda)$ admits a unique maximal proper submodule
$\wt{M}'(\lambda)$.

The Verma modules $\wt{M}(\lambda)$, $M(\lambda)$, $M^\CG(\lambda)$ 
admit unique simple quotients.

\begin{lem}\label{lemMM}
One has
$$M(\lambda)=\cU(\fg)\otimes_{\cU(\wt{\fg})} \wt{M}(\lambda).$$
\end{lem}
\qed

\subsection{Embedding of $\fr^-/[\fr^-,\fr^-]$}

The composition 
$$\fr^-\hookrightarrow \wt{\fg}/\wt{\fb}\hookrightarrow 
\cU(\wt{\fg})/\cU(\wt{\fg})\wt{\fb}=\wt{M}(0)$$
has the image in $\wt{M}'(0)=\bigoplus_{\alpha\in\Sigma}\wt{M}(-\alpha)$.
We denote by 
\begin{equation}
\label{eq:phifromrf-}
\phi:\fr^-\to \bigoplus_{\alpha\in\Sigma}M^\CG(-\alpha)
\end{equation}
the composition of this with the projection
$$
\bigoplus_{\alpha\in\Sigma}\wt{M}(-\alpha)\to
\bigoplus_{\alpha\in\Sigma}M^\CG(-\alpha).
$$
\begin{prp} 
\label{propK911}
The map $\phi$ defined above is a map of $\ \wt\fg$-modules with kernel
$[\fr^-,\fr^-]$.
\end{prp}
\begin{proof}
This result is the main part of the proof of  Proposition 9.11 of~\cite{Kbook}.

\end{proof}

\subsubsection{Example}
If $\fg^\CG=\fsl_2\times\fsl_2$ with $\Sigma=\{\alpha_1,\alpha_2\}$,
the image of $\phi$ in $M^\CG(-\alpha_i)$ is equal to 
$M^\CG(-\alpha_1-\alpha_2)$.

\

Recall that $\fg^\U$ denotes the universal root algebra.
\begin{crl}\label{crlanotherideal}
Assume that $\bigoplus_{\alpha\in\Sigma} M^\CG(-\alpha)$ has no nonzero 
integrable subquotients. Then $\fg^\U=\fg^\CG$.
\end{crl}

\begin{proof} 
Let $\fs=\Ker(\wt\fg\to\fg^\U)$. Set
$\fs^-:=\wt{\fn}^-\cap \fs$. Obviously, $\fs\subset\fr$ so 
$\fs^-\subset\fr^-$.

Assume that $\fr^-/\fs^-\ne 0$. This Lie superalgebra is a semisimple 
$\fh$-module with the weights belonging to $-Q^+\setminus\{0\}$.
This implies that it does not coincide with its commutator, that is,
that $\fr^-/(\fs^-+[\fr^-,\fr^-])\ne 0$. Since the adjoint representation
of $\fg$ is integrable, $\fr^-/(\fs^-+[\fr^-,\fr^-])$ is a nonzero 
integrable $\fg$-module. Using \ref{propK911} we get a nonzero
 integrable subquotient in $\bigoplus_{\alpha\in\Sigma} M^\CG(-\alpha)$
which contradicts the conditions.
Thus,  $\fs^-=\fr^-$, so  automatically
$\fs^+=\fr^+$  as the automorphisms $\theta$, see~\ref{sss:automorphism}, defined on $\wt\fg$, $\fg^\U$ and $\fg^\CG$, identifies
$\fs^+$ with $\fs^-$ and $\fr^+$ with $\fr^-$.
\end{proof}

\subsection{Main result}
In this subsection we assume that the Cartan matrix for $r$ is symmetric, i.e.
$$\forall x,y\in X\ \ \ \ \langle b(x),a(y)\rangle=\langle b(y),a(x)\rangle.$$
Note that by \ref{lem:sym-stable} all Cartan matrices at $r'\in\cR_0$ are symmetrizable.

By Proposition~\ref{prp:likekac22} \ $\wt{\fg}$ admits an invariant 
bilinear form such that the restriction of this form on $\fh$ is non-degenerate and
and $(a(x)|h)=\langle b(x),h\rangle$ for each $h\in\fh$.

\subsubsection{}
Let us show that $\fr$ coincides with the kernel of this form
(for symmetrizable Kac-Moody algebras this fact was earlier noted in~\cite{SchV}).
 Indeed,
since the kernel is an ideal and the restriction of $(-|-)$ on $\fh$ is 
non-degenerate, the kernel lies in $\fr$.
Since $(\wt{\fg}_{\alpha}|\wt{\fg}_{\beta})=0$
for $\alpha+\beta\not=0$, one has $(\fh|\fr)=0$.  Thus
$$\fr^{\perp}:=\{g\in\wt{\fg}|\ (g|\fr)=0\}$$
is an ideal containing $\fh$, so $\fr^{\perp}=\wt{\fg}$, that is
$\fr$ lies in the kernel of $(-|-)$.
Thus, the algebra $\fg^\CG$ inherits a non-degenerate  invariant 
bilinear form having the properties listed in \ref{prp:likekac22}.

\begin{thm}{}
\label{thm:symmetric-g-gkm}
Let $\cR_0$ be symmetrizable and let $\fg$ be a root Lie superalgebra.
Then $\fg=\fg^\CG$, except for the cases $\fgl(1|1)$ and {\rm(Aff)} with 
$(\rho|\delta)=0$.
\end{thm}
\begin{proof}
Symmeric nondegenerate bilinear form of $\fg^\CG$ allows one to define  a {\em Casimir operator}, see~\cite{Kbook}, 2.5.
This operator acts on $M^\CG(\lambda)$ by $(\lambda|\lambda+2\rho)\cdot\id$.
This implies
\begin{equation}
\label{Casimir}
[M^\CG(\lambda):L^\CG(\mu)]\not=0\Longrightarrow
(\lambda|\lambda+2\rho)=(\mu|\mu+2\rho).
\end{equation}

Assume that $\fr\not=\fs$.
By~\ref{crlanotherideal}, for some
$\alpha\in\Sigma$ there is a non-zero homomorphism
$$\fr^-\to M^\CG(-\alpha).$$
Hence $M^\CG(-\alpha)$ admits an integrable subquotient $L^\CG(\mu)$
for some $\mu$.
Since $L^\CG(-\alpha)$ is a subquotient of $M^\CG(0)$, 
the formula (\ref{Casimir}) gives
\begin{equation}
\label{rhonu}
(\mu|\mu+2\rho)=0.
\end{equation}

If $\cR_0$ is of type (Fin) and not $\fgl(1|1)$ then $\fr=\fs$ by~\ref{crlfin} (1). 

Let us consider the case when $\cR_0$ is of type (Aff). By~\ref{crlfin} (2), 
$\mu=j\delta$
for some $j\in\mathbb{Z}_{>0}$ and  $\delta(h)=0$ 
for each $h\in\fh\cap [\fg,\fg]$. Therefore  
$(\delta|\alpha)=0$ for each $\alpha\in\Sigma$. This gives
$(\delta|\delta)=0$. Using~(\ref{rhonu}), we get
$h^{\vee}_v=2(\rho|\delta)=0$.

It remains to consider the component $\cR_0$ of type (Ind).
By~\cite{Hoyt}, the algebras $Q^{\pm}(m,n,t)$ are not symmetrizable. 
The rest of indefinite types satisfy $\Delta_\iso=\emptyset$. Then $a_{xx}\not=0$
for each $x\in X$ and $a_{xy}=a_{yx}$. 
It is easy to see that we can choose $v\in\cR_0$ in such a way that 
$a_{xx}\in\mathbb{Z}_{>0}$. Then the integrability gives
$(\mu|\alpha)\geq 0$ for each $\alpha\in \Sigma$. 
Since $-\mu\in Q^+$ and $\mu\not=0$, we obtain 
$(\mu|\rho)<0, (\mu|\mu)<0$, a contradiction to (\ref{rhonu}).
\end{proof}

\section{The affine case}
\label{sect:aff}
\subsection{}In this section we prove the following result.

\begin{thm}
\label{thm:UKM-ns}
 Let $\cR_0$ be an indecomposable component  of type
(Aff).
If $\cR_0$ is of type  $A(n-1|n-1)^{(1)}$ (resp., $A(2n-1|2n-1)^{(2)}$,
$A(2n|2n)^{(4)}$), then 
 $\fg^\U=\mathfrak{sl}(n|n)^{(1)}$ (resp.,  $\fg^\U=\mathfrak{sl}(2n|2n)^{(2)}$,
$\mathfrak{sl}(2n+1|2n+1)^{(4)}$).
 If $\cR_0$ is of type  $\fq(n)^{(2)}$ then $\fg^\U=\mathfrak{sq}(n)^{(2)}$.
In the rest of the cases $\fg^\U=\fg^{\CG}$.
\end{thm}

Let us first notice that for $S(2,1,b)$ Lemma~\ref{crlaff} (3) and
\ref{sss:s21b}
imply $\fg^\U=\fg^{\CG}$.
In all other cases we define for any root algebra  $\fg$ its subfactor
 $\bar{\fg}:=[\fg,\fg]/Z(\fg)$.

Then
$\bar{\fg}^{\CG}=[\fg^\CG,\fg^\CG]/Z(\fg^\CG)$ is isomorphic to the twisted loop algebra $\cL(\fs)^{\sigma}$ for some simple superalgebra $\fs$ and an automorphism $\sigma$ of finite
order $m$. In particular, $\bar\fg^\CG$ is perfect. 
The superalgebra $\fs$ is basic classical, exceptional or $\fp\fs\fq_n$.
Its even part $\fs_{\bar 0}$, therefore, is a reductive Lie algebra.

Let $\fh'$ be the even part of the Cartan subalgebra of $\fs$. One can choose $\sigma$ so that $\sigma(\fh')=\fh'$. 

Furthermore, if $k\delta$ is an even root and $\varepsilon=e^{\frac{2\pi i}{m}}$  then
$$\bar{\fg}^\CG_{k\delta}=\{h\otimes t^k\mid h\in\fh',\sigma(h)=\varepsilon^k h\}.$$

The cohomology group $H^i(\bar{\fg}^{\CG},\mathbb C)$ has a natural structure of $\fh$-module. We write
$H^i(\bar{\fg}^{\CG},\mathbb C)_\mu$ for the cohomology group of weight $\mu$ with
respect to $\fh$-action.

\begin{lem}\label{lem:extension} For every $k\neq 0$
  $$\dim\fg^{\U}_{k\delta}-\dim\fg^{\CG}_{k\delta}=\dim H^2(\bar{\fg}^{\CG},\mathbb C)_{k\delta}.$$
  \end{lem}
  \begin{proof}     Let $\hat\fg$ be the graded central extension of $\bar{\fg}^{\CG}$ given by the exact sequence
    $$0\to \bigoplus_{k\neq 0}H^2(\bar{\fg}^{\CG},\mathbb C)^*_{k\delta}\to \hat\fg\to \bar{\fg}^{\CG}\to 0.$$
    Take the pullback
    $$0\to \bigoplus_{k\neq 0}H^2(\bar{\fg}^{\CG},\mathbb C)^*_{k\delta}\to \hat\fg'\to [{\fg}^{\CG},{\fg}^{\CG}]\to 0,$$
    and then extend to the exact sequence
    $$0\to \bigoplus_{k\neq 0}H^2(\bar{\fg}^{\CG},\mathbb C)^*_{k\delta}\to \fg\to {\fg}^{\CG}\to 0$$
  using the semidirect product decomposition
  $\fg^{\CG}=\ft\ltimes [{\fg}^{\CG},{\fg}^{\CG}]$ where $\ft\subset\fh$ is a suitable abelian subalgebra.
    
  We claim that $\fg$ is a root algebra. Indeed, we just have to check the relations \ref{sss:half} at every vertex $v\in\cR_0$. The only
  non-trivial relation is $[\tilde e_x,\tilde f_y]=0$ for $x\neq y$. This  is equivalent to $b(x)-b(y)\neq k\delta$ and the latter follows
  from $k\delta\in Q^+(v)$ for positive $k$ and  $k\delta\in -Q^+(v)$ for negative $k$.

  Finally, let us prove that $\fg=\fg^{\U}$. Indeed, by ~\ref{crlaff} the kernel $\fk$ of the map $\fg^{\U}\to\fg$ lies in the center of $[\fg^{\U},\fg^{\U}]$
  and is a direct sum $\bigoplus_{k\neq 0}\fk_{k\delta}$. Therefore $\fg^{\U}=\fg$.
  \end{proof}

\subsubsection{}
Let $\delta$ have degree $d$ in the standard grading of $\cL(\fs)^{\sigma}$.
The base change $H^2(\fs,\bC)\to H^2(\cL(\fs),\bC[t,t^{-1}])$ composed
with the linear map $\bC[t,t^{-1}]\to\bC$ carrying $\sum c_it^i$ to $c_{kd}$, yields a homomorphism
\begin{equation}
\label{eq:h2map}
H^2(\fs,\bC)\to H^2(\cL(\fs)^\sigma,\bC).
\end{equation} 
It is given on $2$-cocycles by the
formula
\begin{equation}
\label{eq:2cocycles}
 \tilde c(x\otimes t^a,y\otimes t^b)=\delta_{kd,a+b}c(x,y).
\end{equation}
 Let $\bar\fg^{\CG}=\cL(\fs)^\sigma$ and $\fh$ be a Cartan subalgebras of $\fg^{\CG}$. Set
$\fh^{\circ}:=\ker\delta$. Then $\fh^\circ$ acts on $\fs$ and therefore on $H^2(\fs,\bC)$. We denote by $H^2(\fs,\bC)^{\circ}$ the
$\fh^{\circ}$-invariant subspace. The automorphism $\sigma$ acts on
$H^2(\fs,\mathbb C)^{\circ}$ and induces a $\mathbb Z/m\mathbb Z$-grading.

\begin{Rem}
In most cases $\fh^\circ=(\fh')^\sigma$ and $H^2(\fs,\bC)^{\circ}=H^2(\fs,\bC)$.
The only case $\fh^\circ\neq(\fh')^\sigma$ is when the Cartan matrix of $\fg^{\CG}$ has corank $2$ and that happens for $\fs=\mathfrak{psl}(n|n)$,
$n\geq 2$ and $\sigma=\id$.
\end{Rem}

  \begin{lem}\label{lem:redfindim} If $k\delta$ is an even root and $kd\equiv p\mod m$ then the homomorphism (\ref{eq:h2map}) induces an isomorphism
    $H^2(\fs,\mathbb C)^{\circ}_{p}\simeq H^2(\cL(\fs)^{\sigma},\mathbb C)_{k\delta}$.
  \end{lem}
  \begin{proof} The correspondence between the weight spaces follows from formula (\ref{eq:2cocycles}). Injectivity of the map is straightforward. To prove surjectivity it suffices to show that every class in
    $H^2(\cL(\fs)^{\sigma},\mathbb C)_{k\delta}$ is represented by a cocycle $\varphi$ such that
\begin{equation}
\label{eq:goodcocycle}
    \varphi(x\otimes t^{a-m},y\otimes t^{b+m})=\varphi(x\otimes t^{a},y\otimes t^{b})
\end{equation}
    for all $a,b\in\bZ$ and $x,y\in\fs$.
    The Lie algebra $\fs'=[\fs_{\bar 0},\fs_{\bar 0}]$ is semisimple.
    The corresponding twisted affine Lie algebra $\hat\fs'$ is symmetrizable and, therefore, $(\hat\fs')^\U=(\hat\fs')^\CG$. 
    By Lemma~\ref{lem:extension} $H^2(\cL(\fs')^\sigma,\mathbb C)_{k\delta}=0$. On the other hand  $\cL(\fs)_{\bar 0}^\sigma=\cL(\fs')^\sigma\oplus \fa$ for some abelian Lie algebra 
    $\fa$. Thus, we can choose $\varphi$ so that
    $\varphi(\cL(\fs')^\sigma,\cL(\fs)^\sigma_{\bar 0})=0$.
    Since $k\delta$ is an even root, $\varphi$ is an even cocycle, so
     $\varphi(\cL(\fs')^\sigma,\cL(\fs)^\sigma)=0$.
     In particular, for every $h\in(\fh'\cap\fs')^\sigma$ we have $\varphi(h\otimes t^m,\cL(\fs)^\sigma)=0$. Let $\alpha$ be a non-zero weight of
     $\fs$ with respect to
     $(\fh'\cap\fs')^\sigma$ and $x\in \fs_\alpha, y\in \fs_{-\alpha}$, we can choose
    $h$ so that $\alpha(h)\ne 0$. Then the cocycle condition
    $$d\varphi(x\otimes t^{a-m},y\otimes t^{b},h\otimes t^m)=0$$
    implies (\ref{eq:goodcocycle}) for $x\in \fs_\alpha, y\in \fs_{-\alpha}$. Since the $\fs_{\alpha}$ for all nonzero weights $\alpha$ generate $\fs$
    and $\varphi(x\otimes t^a,y\otimes t^b)=0$ for $x\in\fs_\alpha$
    and $y\in\fs_\beta$ with $\alpha+\beta\ne 0$,
    one proves the desired identity for all $x,y$ using linearity and the cocycle condition.
 
    \end{proof}

    Lemma~\ref{lem:redfindim} implies Theorem~\ref{thm:UKM-ns}   in all 
    cases when $\delta$ is an even  root. If $\fs\ne\mathfrak{psl}(n|n)$ or $\fp\fs\fq(n)$, $H^2(\fs,\mathbb C)=0$ and then $\fg^\U=\fg^\CG$. 
    If $\fs=\mathfrak{psl}(n|n)$ or $\fp\fs\fq(n)$,
     $H^2(\fs,\mathbb C)^{\circ}=\mathbb C$, see, for instance, \cite{S4}.  This gives the cases 
     $\fg^\U=\fsl(n|n)^{(1)}$ and  $\fg^\U=\fsl(2n|2n)^{(2)}$.
    The only cases left are  $\fg^\CG=\fp\fsl(2n+1|2n+1)^{(4)}$ and $\fp\fs\fq(n)^{(2)}$ where $\delta$ is an odd root.
    For these remaining cases the theorem will follow from the lemma below.
    \begin{lem}\label{lem:oddextensions}  If $\cR_0$ is of type  $A(2n|2n)^{(4)}$ or $\fq(n)^{(2)}$ then
      $H^2(\bar{\fg}^{\CG},\mathbb C)_{k\delta}=0$ for any odd $k$.
      \end{lem}
      \begin{proof} First let us deal with  $A(2n|2n)^{(4)}$. In this case $\fs=\mathfrak{psl}(2n+1|2n+1)$, $m=4$ and we can choose $\sigma$ so that
        $\fs^{\sigma}=\mathfrak{so}(2n+1)\oplus\mathfrak{so}(2n+1)$. We will establish
        an isomorphism $H^2(\fs,\mathbb C)^\circ_{p}\simeq H^2(\cL(\fs)^{\sigma},\mathbb C)_{k\delta}$ for odd $k$. As in the proof of
        Lemma ~\ref{lem:redfindim}, it suffices to check that we can choose a cocycle $\varphi$ satisfying (\ref{eq:goodcocycle}). This in
        turn would follow from the condition $\varphi(h\otimes t^4,\cL(\fs)^\sigma)=0$ for all $h\in(\fh')^{\sigma}$. Using the root
        description,~\cite{vdL}, we see that $\alpha$ and
        $-\alpha+k\delta$ are both real roots of $\bar{\fg}^{\CG}$ only for the short anisotropic $\alpha$. Thus, if $x\in\bar{\fg}^{\CG}_\beta$ for
        some long anisotropic root $\beta$ then $\varphi(x,\bar{\fg}^{\CG})=0$. On the other hand, every $h\otimes t^4$ can be obtained as a linear
        combination of $[x,y]$, $x\in\bar{\fg}^{\CG}_{\beta}$ and $y\in\bar{\fg}^{\CG}_{-\beta+4\delta}$ for some long anisotropic roots $\beta$. Therefore 
        $\varphi(h\otimes t^4,\cL(\fs)^\sigma)=0$ for all $h\in(\fh')^{\sigma}$. The statement of lemma now follows from 
        $H^2(\fs)_{\bar 1}=0$.

        In the case of $\fq(n)^{(2)}$ we have a grading $\bar{\fg}^{\CG}=\bigoplus\bar{\fg}^{\CG}_i$ induced by the standard grading on Laurent polynomials, with
        $\bar{\fg}^{\CG}_0=\mathfrak{sl}(n)$. For every $i$ the term $\bar{\fg}^{\CG}_i$ is the adjoint $\bar{\fg}^{\CG}_0$-module. The parity of
        $\bar{\fg}^{\CG}_i$ equals the parity of $i$. Let $s=2k+1$. To compute $H^2(\bar{\fg}^{\CG},\mathbb C)_{s\delta}$ we consider the first layer
        of Hochshild-Serre spectral sequence (see, for instance, \cite{F}, Sect. 5) with respect to subalgebra $\bar{\fg}^{\CG}_0$:

        $$H^2(\bar{\fg}^{\CG}_0,\mathbb C)\oplus H^1(\bar{\fg}^{\CG}_0,(\bar{\fg}^{\CG}_s)^*)\oplus
        H^0(\bar{\fg}_0^{\CG},\oplus_{a+b=s} (\bar{\fg}^{\CG}_a\otimes\bar{\fg}^{\CG}_b)^*).$$

        Since $H^2(\bar{\fg}^{\CG}_0,\mathbb C)=0$, $H^1(\bar{\fg}^{\CG}_0,(\bar{\fg}^{\CG}_s)^*)=0$ and
        $H^0(\bar{\fg}_0^{\CG},(\bar{\fg}^{\CG}_a\otimes\bar{\fg}^{\CG}_b)^*)=\mathbb C$ we obtain that every cocycle
        $c\in H^2(\bar{\fg}^{\CG},\mathbb C)_{s\delta}$ can be written in the form
        $$c(x\otimes t^a,y\otimes t^b)=\gamma(a,b)\operatorname{tr} (xy),\ \gamma:\bZ\times\bZ\to\bC.$$
        Furthermore $\gamma$ has the following properties
        \begin{itemize}
        \item weight condition: $\gamma(a,b)=0$ unless $a+b=s$;
        \item skew-symmetry: $\gamma (a,b)=-\gamma(b,a)$;
        \item $\gamma(0,s)=0$;
          \item cocycle condition: $\gamma(a,b+c)=\gamma(a+b,c)-\gamma(b,a+c)$.
          \end{itemize}
          The last condition follows by direct computation using the property of the trace
$\operatorname{tr}(uvw)=\operatorname{tr}(vwu)$.
    Without loss of generality assume that $s>0$. By the cocycle condition and 
    skew-symmetry
    $$\gamma(p,s-p)=\gamma(p,s-p+1-1)=\gamma(s+1,-1)+\gamma(p-1,s-p+1).$$
    By induction
    $$\gamma(p,s-p)=p\gamma(s+1,-1)+\gamma(0,s)=p\gamma(s+1,-1).$$
    Hence $0=\gamma(s,0)=s\gamma(s+1,-1)$
 that implies $\gamma(s+1,-1)=0$. Therefore  $\gamma\equiv 0$. Thus, $H^2(\bar{\fg}^{\CG},\mathbb C)_{s\delta}=0$.

  \end{proof}

\section{Description of root algebras. Examples}
\label{sec:app}


In Subsection \ref{rootalg} we describe root algebras in the indecomposable fully reflectable case. In the rest of this section 
we compute some of the groups $\Aut_\cR(v)$.

In this section we identify admissibile components of $\cR$ by
root Lie superalgebras supported on them.

\subsection{}\label{rootalg} By contrast with the case  $\mathfrak{gl}(1|1)$, see~\ref{rank1},
we have the following

\begin{thm} Let $\cR_0$ be a indecomposable admissible fully reflectable  component of the root groupoid, not isomorphic to $\mathfrak{gl}(1|1)$.
  Then any ideal of $\fg^\U$ having zero intersection with $\fh$
  defines a root algebra. If $\cR_0$ is of type (Aff) and $\fg^{\U}\neq \fg^{\CG}$ then all such ideals are in natural bijection with subsets of $\mathbb Z\setminus 0$.  
\end{thm}
\begin{proof} By~\ref{crl:invariantideals} we need to consider only components with isotropic reflexions. Furthermore, we are only interested in the case $\Sp^D(v)\neq \{1\}$
  and $\fg^{\CG}\neq \fg^{\U}$. By~\ref{crlfin} and~\ref{Qpm} (see below) 
  this leaves us with components of type (Aff) listed in~\ref{thm:UKM-ns}. Let $\fg$ be a root algebra and 
  $$J^{\CG}:=\Ker(\fg^{\U}\to\fg^{\CG}),\ \ \  J:=\Ker(\fg^{\U}\to\fg).$$
  By~\ref{thm:UKM-ns} we have
  $$J^{\CG}=\bigoplus_{s\in\mathbb Z\setminus 0} J^{\CG}_{s\delta},\ \ \ \dim  J^{\CG}_{s\delta}\leq 1.$$
  It follows from the definition of $Q^{++}$ that it is $\Aut_{\cR}(v)$-stable. Since $Q^{++}=\mathbb Z_{\geq 0}\delta$ and $ J^{\CG}$ is
  $\Aut_{\cR}(v)$-stable we obtain that
  $ J^{\CG}_{s\delta}$ is $\Aut_{\cR}(v)$-stable for any $s$. Therefore any graded subspace of $ J^{\CG}$ is $\Aut_{\cR}(v)$-stable. Moreover, by~\ref{crlaff} (2)
  any graded subspace of $ J^{\CG}$ is an ideal. Hence by~\ref{sss:invariantideals} the root algebras are in bijection with the graded subspaces of $J^{\CG}$.
  The last assertion follows from the description of $\fg^{\U}$ given in~\ref{thm:UKM-ns}.
 \end{proof}
 
  \begin{rem} Note that by above theorem a root algebra may not admit a superinvolution $\theta$ defined in~\ref{sss:automorphism}.
    \end{rem}

\subsection{Star-shaped spines} Here we calculate the automorphism groups in a few small
examples.

\subsubsection{Example}\label{q32}
The following root datum contains root algebra $q(3)^{(2)}$.
Take $X=\{x_1,x_2,x_3\}$ and let $\fh=\fh(v)$ have dimension 4
 with the Cartan matrix
$$\begin{pmatrix}
0 & -1 & 1\\
-1 & 0 & 1\\
1 & -1& 0
\end{pmatrix},\ \ \  p(x_i)=1\ \text{ for } i=1,2,3.$$
Then the graph $\Sp(v)$ is a star with $v$ at the center and  three other vertices $v_i$ with
$r_{x_i}:v\to v_i$ and the Cartan matrices
$$v_1:\ \begin{pmatrix}
0 & -1 & 1\\
1 & -2 & 1\\
-1 & -1& 2
\end{pmatrix}
\ \ \ 
v_2:\ \begin{pmatrix}
-2 & 1 & 1\\
1 & 0 & -1\\
1 & 1& -2
\end{pmatrix}
\ \ \ 
v_3: \begin{pmatrix}
2 & -1 & -1\\
-1 & 2 & -1\\
-1 & 1& 0
\end{pmatrix}$$
with $p_{v_j}(x_i)=\delta_{ij}$. We have three  principal
reflections $s_{\alpha_k}$,  where 
$$\alpha_k:=b(x_i)+b(x_j)=b_{v_i}(x_j)=b_{v_j}(x_i)$$ 
for $\{i,j,k\}=\{1,2,3\}$. The Weyl group is generated
by these reflections (this group is isomorphic to the affine Weyl group $A_2^{(1)})$. The group $K(v)$ is the additive group $\bC$. If we choose
$\fh$ of dimension greater than $4$, the Weyl group will remain the same, but $K(v)$ will be different. Regardless of $\fh$, $\Aut_\cR(v)=
W(v)\times K(v)$ by~\ref{crl:all-different}.

\subsubsection{Example: $B(1|1)^{(1)}$, $D(2|1,a)$, $D(2|1,a)^{(1)}$, $Q^{\pm}(m,n,t)$}\label{Qpm} All these cases are similar to~\ref{q32}. We can
(and will)  choose a vertex $v$ such that
$p(x)=1$ for all $x\in X$. We always have   $a_{xy}\not=0$ if $x\not=y$.
The graph $\Sp(v)$ is a star with the center at $v$. The other
vertices are $v_x$ with the edges $r_{x}: v\to v_x$.
If $a_{xx}=0$ then $p'(y)=0$ for each $y\not=x$. Hence $\Sp(v)$  consists of $v$ and all
$v_x$ such that $a_{xx}=0$. Cartan data at all vertices of $\Sp(v)$ are not
$D$-equivalent, so~\ref{crl:all-different} is applicable.
This gives $\Aut(v)=W\times K$.

\subsection{$\fsl_n^{(1)}$, its relatives and friends}


There is a number of components of the root groupoid whose Cartan matrices satisfy
common properties listed below in (\ref{aij01}) and whose automorphism groups allow
a more or less uniform description. We call them ``relatives and friends of $\fsl_n^{(1)}$''
and they consist of the types
$\fsl(k|\ell)^{(1)}$ for $k,\ell$ such that
$k+\ell=n$ and $\fq_n^{(2)}$. 

We take  $X=\{x_i\}_{i\in\mathbb{Z}_n}$.
Let $v\in\cR_0$ be a vertex with the  Cartan matrix of the following form:
\begin{equation}\label{aij01}
\begin{array}{ll}
a_{ij}=0\ \text{ for }j\not=i,i\pm 1;\\
a_{i,i\pm 1}\in \{\pm 1\},\ \ \ 
a_{i,i-1}+a_{ii}+a_{i,i+1}=0\\
p(x_i)=1\ \ \Longleftrightarrow\ \ a_{ii}=0.
\end{array}\end{equation}

\subsubsection{}\label{goodmatrices}
It is easy to check that 
\begin{itemize}
\item
If a Cartan matrix satisfies~(\ref{aij01}), then all $x_i$ are reflectable
at $v$ and $\sum_i b_{v}(x_i)=\sum_i b_{v'}(x_i)$ for each reflexion
$v\to v'$;
\item 
all  Cartan matrices in $\Sk(v)$ satisfy~(\ref{aij01});
\item two Cartan matrices $A,A'$ satisfying~(\ref{aij01}) are $D$-equivalent
if and only if $p(x_i)=p'(x_i)$ for all $i$. 
\end{itemize}

\subsubsection{}
\label{sss:iota}
Let $\ol{\cR}_0$ be the component of $\cR$ corresponding to $\fsl_n^{(1)}$;
we will use bar notation $\bar v$ etc. for the objects connected to $\ol\cR_0$.
Fix a linear isomorphism $\iota:  Q_{\bar v}\stackrel{\sim}{\to} Q_{v}$
given by $\iota({b}_{\bar v}(x_i)):=b_v(x_i)$.

Let $v\to v'$ be a path in $\cR_0$ and $\bar v\to \bar v'$ be its namesake
in $\ol{\cR}_0$. It is easy to see that 
$$b_v(x_i)=\iota({b}_{\ol{v}}(x_i)).$$
This provides a bijection between the sets of real roots $\Delta^\re=\ol{\Delta}^\re$.
Note that all roots of $\ol\Delta^\re$ are anisotropic.
Since the set $\{b_v(x_i)\}_{i\in\bZ_n}$  determines a vertex in $\Sk(v)$
by~\ref{crl:unique-in-sk},
this gives a bijection between $\Sk(v)$ and $\Sk(\bar v)$.

\subsubsection{}
We identify $Q_{v}$ and $Q_{\bar v}$ via $\iota$.

By~\ref{crl:Wfree} the Weyl group $W(\fsl_n^{(1)})$ acts freely  on $\Sk(\bar v)$.
By~\ref{prp:decomposition0} this action is transitive.
This gives a  simply transitive action of $W(\fsl_n^{(1)})$ on $\Sk(v)$.
Note that the Weyl group $W$ can be identified with a subgroup of $W(\fsl_n^{(1)})$ as it is generated by a part of the reflections
belonging to $W(\fsl_n^{(1)})$.

Let us compute 
$$\Aut(v)/K(v)=\Sk^D(v)=\{w\in W(\fsl_n^{(1)})|\ \ 
A_{w(v)}\ \text{ is $D$-equivalent to }A_v\}.$$

\subsubsection{Action of $W(\fsl_n^{(1)})$}
By~\ref{goodmatrices}, the vector
$$\delta:=\sum_{i=1}^n b_{v'}(x_i)$$
does not depend on the choice of $v'\in\Sk(v)$.

View $Q_{v}$ as a subset of $V=\Span_\bZ(\varepsilon_1,\dots,\varepsilon_n,\delta)$  by setting 
$$b(x_i)=\varepsilon_i-\varepsilon_{i+1}\ \text{ for }i=1,\ldots,n-1; \ \ 
b(x_n)=\delta+\varepsilon_n-\varepsilon_1.$$
We can extend the parity function $p: Q_v\to \mathbb{Z}_2$ to $p: V\to \mathbb{Z}_2$ by setting $p(\varepsilon_1)=0$. Set
$$\bar Q:=\{\sum_{i=1}^n  k_i\varepsilon_i|\ \sum_{i=1}^n  k_i=0,\ k_i\in\mathbb{Z}\}.$$
(Note: $\bar Q$ is the lattice for the finite root system $A_{n-1}$.) 
By \cite{Kbook}, Thm. 6.5,  $W(\fsl_n^{(1)})=S_n\ltimes \bar Q$ and this group acts on $V$ 
as follows:  
\begin{itemize}
\item
$S_n$ acts on $\{\varepsilon_i\}_{i=1}^n$ by permutations and stabilizes $\delta$;
\item $\bar Q$ 
acts on $V$ by the formula
$$\nu*\mu:=\mu-(\mu,\nu)\delta\ \ \text{ for }\nu\in \bar Q,\ \mu\in V$$
where the bilinear form on $V$ is given by 
$$(\varepsilon_i,\varepsilon_j)=\delta_{ij},\ \ \  (\varepsilon_i,\delta)=(\delta,\delta)=0.$$
\end{itemize}
Note that $W(\fsl_n^{(1)})$ stabilizes $\delta$. 
By~\ref{goodmatrices}, $A_{w(v)}$ is $D$-equivalent to $A_v$ if and only if
$p_v(x_i)=p_{w(v)}(x_i)$ for all $i$.
Therefore,
\begin{equation}
\label{eq:skd=}
w\in\Sk^D(v)\ \Longleftrightarrow\ \
p(w\varepsilon_i)-p(\varepsilon_i)\ \text{ is independent of $i$}.
\end{equation}
We will now compute the groups $\Sk^D(v)$ using  the formula (\ref{eq:skd=}).

\subsubsection{Case $\fsl(k|\ell)^{(1)}$, $k,\ell\not=0$}
We can choose $v$ in such a way that
$p(x_i)=0$ for $i\not=k,n$ and $p(x_n)=p(x_k)=1$. Note that $p(\delta)=0$.
Denote by $S_k\subset S_n$ (resp., $S_\ell\subset S_n$)
the group of permutations of  $\{\varepsilon_i\}_{i=1}^k$ (resp., of $\{\varepsilon_i\}_{i=k+1}^n$). 
In this case $p(w\varepsilon_i)=p(\varepsilon_i)$ for $w\in \bar Q$, so $\Sk^D(v)\supset \bar Q$.

One has
$$S_n\cap\Sk^D(v)=\{w\in S_n|\ p'(w(\varepsilon_i-\varepsilon_{i+1}))=p'(\varepsilon_i-\varepsilon_{i+1})\ \text{ for }i=1,\ldots,n-1\}.$$
If $k\ne\ell$ this gives $S_n\cap\Sk^D(v)=S_k\times S_{\ell}$. In the case $k=\ell$
we have \newline
$S_n\cap\Sk^D(v)=(S_k\times S_k)\rtimes\mathbb{Z}_2$, where
$\mathbb{Z}_2$ interchanges the two copies of $S_k$.
Hence
$$\Sk^D(v)=\left\{ \begin{array}{ll}
(S_k\times S_{\ell})\ltimes \bar Q\ & \text{ if }k\not=\ell\\
((S_k\times S_k)\rtimes\mathbb{Z}_2)\ltimes \bar Q\ & \text{ if }k=\ell.\end{array}
\right.$$
Note that the Weyl group has the form
$W=(S_k\times S_{\ell})\ltimes Q_0$
where  $Q_0\subset\bar Q$ is the subgroup spanned 
$\{\varepsilon_i-\varepsilon_{i+1}\}_{i=1}^{k-1}\coprod \{\varepsilon_i-\varepsilon_{i+1}\}_{i=k+1}^{n-1}$. Observe that $W$ has an infinite index in $\Sk^D(v)$.

\begin{rem}
\label{sss:glmn}
For $\cR_0$ of type  $A(k-1|\ell-1)$ a
similar reasoning (replacing the index set
$X=\{x_i\}_{i\in\mathbb{Z}_n}$ with the set
$X=\{x_1,\ldots, x_n\}$)
 shows that $S_{k+\ell}$ acts transitively on $\Sk(v)$ and that
$$\Sk^D(v)=\left\{ \begin{array}{ll}
S_k\times S_{\ell} & \text{ if }k\not=\ell\\
(S_k\times S_k)\rtimes\mathbb{Z}_2 & \text{ if }k=\ell.\end{array}
\right.$$
Note that the Weyl group is in both cases $S_k\times S_{\ell}$.
\end{rem}

If $k=l$ then $K(v)=\bC$ and $\Aut(v)$ is a nontrivial semidirect
product of $\bC$ and $\Sk^D(v)$.

\subsubsection{Case $\fq_n^{(2)}$}
Using~\cite{Kbook}, Thm. 6.5 and \cite{S3}, one gets
$$W=S_n\ltimes 2\bar Q.$$
We will choose $v$ so that
$p(x_i)=0$ for $i=1,\ldots,n-1$ and $p(x_n)=1$. Note that $p(\delta)=1$.

In this case $p(w\varepsilon_i)=p(\varepsilon_i)$ for $w\in S_n$, so $S_n\subset\Sk^D(v)$.
Hence
$$\Sk^D(v)=S_n\ltimes Q'$$
where $Q'=\bar Q\cap \Sk^D(v)$.
Take $\nu\in\bar Q$. One has 
$$p(\nu*\varepsilon_i)-p(\varepsilon_i)\equiv (\nu,\varepsilon_i)\mod 2,$$
so
$$Q'=\{\sum_{i=1}^n  k_i\varepsilon_i|\ \sum_{i=1}^n  k_i=0,\ k_i\in\mathbb{Z},
k_i-k_j\equiv 0\mod 2\}.$$

If $n$ is odd this gives
$Q'=2\bar Q$, so $\Sk^D(v)=W$ and $\Aut_\cR(v)=W\times K$.

If $n$ is even, $2\bar Q$ has index $2$ in $Q'$. Thus $W$ has index two in $\Sk^D(v)$,
so that $W\times K$ is an index 2 subgroup of $\Aut_\cR(v)$.

\subsection{A deformation of $\fsl(2|1)^{(1)}$}
\label{ss:s21b}

A very interesting relative of $\fsl(2|1)^{(1)}$
is the root Lie superalgebra  $S(2|1,b)$ defined in~\cite{S3}. We will recall some of the results of~\cite{S3} below. Set $X:=\{x_1,x_2,x_0\}$ and fix $\fh$ with $\dim\fh=4$.

Let $\cR(b)$, $b\ne 0$, be the component of $\cR$ containing a vertex $v$
such that $p_v(x_1)=p_v(x_2)=1$, $p_v(x_0)=0$ and
the Cartan matrix $A_v$ is equal to 
$$A(b):=\begin{pmatrix}
0 & b & 1-b\\
-b & 0 &1+b\\
-1 & -1 & 2
\end{pmatrix}$$
for $b\not=0$. 

In studying skeleta of $\cR(b)$ it is convenient to allow permutations of the elements of $X$. This leads to the action of $S_3$ on the components of $\cR$ with the index set $X$
and, as we will see soon, carries components $\cR(b)$ to components of the same type.
 
Permuting $x_1$ and $x_2$ in $A(b)$ we obtain
 $A(-b)$, so $\cR(b)$ is mapping to $\cR(-b)$. In particular,
each root algebra for $S(2|1;b)$ is isomorphic to a root algebra
for  $S(2|1; -b)$.

\begin{lem}\label{lemS21b}
For any vertex $v\in\cR(b)$ the Cartan matrix $A_v=(a^{(v)}_{xy})$ is of the form
$\sigma(D A(b+i))$ where $i\in \mathbb{Z}$,
 $D$ is an invertible diagonal matrix and $\sigma\in S_3$ is an even permutation.
One has $p_v(x)=1$ if $a^{(v)}_{xx}=0$ and $p_v(x)=0$ otherwise.
\end{lem}
\begin{proof}
It is enough to verify what happens to the Cartan datum under an isotropic reflexion
$r_{x}:\ v\to v'$. Since permuting $x_1$ and $x_2$ in $A(b)$ yields 
 $A(-b)$, it is enough to verify the assertion for $x=x_1$. In this case we have
$$A_{v'}=\begin{pmatrix}
0 & -b & -1+b\\
b & -2b & b\\
1 & \frac{2-b}{b-1} & 0
\end{pmatrix}.$$
Taking the homothety $h_{\lambda}:v'\to v''$ with $\lambda=(-1,-b^{-1}, b-1)$ we get 
$$A_{v''}=\begin{pmatrix}
0 & b & 1-b\\
-1 & 2 & -1\\
b-1 & 2-b & 0
\end{pmatrix}.$$ 
Applying now the cyclic permutation carrying $x_2$ to $x_1$, we get the Cartan matrix 
$A(b-1)$.
It is easy to see that going along the other isotropic reflexion would produce in the same way the matrix $A(b+1)$.
\end{proof}

\begin{crl}\label{crlS21b}
\begin{itemize}
\item[1.] $\cR(b)$ is admissible
if and only if $b\not\in\mathbb{Z}$;
\item[2.] if $\cR(b)$ is admissible, then
for $i\in\mathbb{Z}$
each root algebra for
 $S(2|1;\pm b\pm i)$ is isomorphic to a root algebra for $S(2|1;b)$.
\end{itemize}
\end{crl}
\begin{proof}
Note that $A(b)$ is  locally weakly symmetric  for $b\not=\pm 1$. 
Using Lemma~\ref{lemS21b} we obtain the assertions.
\end{proof}

\subsubsection{}\label{S21bproperties}
From now on we assume that  
$\cR(b)$ is admissible i.e. $b\not\in\mathbb{Z}$.
Using Lemma~\ref{lemS21b} we obtain

\begin{enumerate}
\item all $x$ are reflectable at each $v\in\cR(b)$;
\item for each reflexion $r_{x}: v\to v'$ we have $b_{v'}(y)=b_{v}(x)+b_{v}(y)$ 
if $y\not=x$;
\item a real root is isotropic if and only if
it is odd.
\end{enumerate}

\subsubsection{}
Let ${\cR}_{\bar v}$ be the component of the root groupoid with $\dim\fh'=4$ and
a vertex $\bar v$ such that $p_{\bar v}(x_1)=p_{\bar v}(x_2)=1$, $p_{\bar v}(x_0)=0$ and
the Cartan matrix 
$$A_{\bar v}:=\begin{pmatrix}
0 & -1 & -1\\
-1 & 0 &-1\\
-1 & -1 & 2
\end{pmatrix}.$$
Then the component ${\cR}_{\bar v}$ of $\bar v$ is  of type $\fsl(2|1)^{(1)}$.

As in \ref{sss:iota}, \ref{S21bproperties}(2) yields a linear isomorphism $\iota: Q_{\bar v}\to Q_{v}$ by setting $\iota(b_{\bar v}(x_i)):=
b_v(x_i)$; by the same arguments, this gives a bijection between $\Sk(v)$ and $\Sk(\bar v)$ with
$b_{v}(x_i)=\iota(b_{\bar v}(x_i)$. 

Note that, contrary to~\ref{sss:iota}, $\iota$ preserves $p(x_i)$. 

\subsubsection{}
\label{sss:s21b}
We have
$$Q^{++}_v=\iota(Q^{++}_{\bar v})=\mathbb{N}\delta\ \text{ for }
\delta:=\sum b_v(x_i).$$

 Therefore, $S(2|1,b)$ is of type (Aff).
Note that  $\langle\delta,a_v(x_1)\rangle=1\not=0$, so by 
Corollary~\ref{crlfin}(3) $\fg^\U=\fg^{\CG}$. 

\subsubsection{}
 By~\ref{S21bproperties}(3) we see that $\iota:Q_{\bar v}\to Q_v$ establishes bijection of
 real, isotropic and anisotropic roots for $\bar v$ and $v$.
Moreover, the bijection between $\Sk(v)$ and $\Sk(\bar v)$ gives a bijection between
the spines $\Sp(v)$ and $\Sp(\bar v)$. In particular,
$\Sp(v)$ has two principal roots
$\alpha:=b_{v}(x_0)$ and $b_v(x_1)+b_v(x_2)=\delta-\alpha$. Using~\ref{S21bproperties} we obtain
$$W=\ol{W}\cong A_1^{(1)}$$ 
and for each $\nu\in Q_{\bar v}$ we have
$w\iota(\nu)=\iota(w\nu)$.

\begin{prp}
$\Aut(v)=W\times K$.
\end{prp}
\begin{proof}
It is enough to check that all Cartan matrices in 
$\Sp(v)$ are not $D$-equivalent. Note that $\Sp(v)$
 can be seen as the infinite graph
$$\ldots\stackrel{r_{x_0}}{\to} v_{-1} \stackrel{r_{x_2}}{\to} v_0\stackrel{r_{x_1}}{\to} v_1 \stackrel{r_{x_0}}{\to} v_2\stackrel{r_{x_2}}{\to} v_3\stackrel{r_{x_1}}{\to} v_3 \stackrel{r_{x_0}}{\to}\ldots $$

Consider the equivalence relation on the set of $3\times 3$ matrices generated
by the action of $A_3$ (the group of even permutations
in $S_3$)  and $B\sim DB$ for a diagonal invertible matrix $D$.
Observe that $A(b)\not\sim A(b')$ if  $b\not=b'$.

In the proof of Lemma~\ref{lemS21b} we showed that
if $A_v\sim A(b)$, then
for an isotropic reflexion $v\stackrel{r_{x}}{\to} v'$ we have $A_{v'}\sim A(b\pm 1)$.
This implies that $A_{v_k}\sim A(b-k)$, so $A_{v_k}\not\sim A_{v_0}$ for any $k\not=0$.
Hence   the group
$\Sp^D(v_0)$ is trivial, so $\Aut(v_0)=W\times K$.
\end{proof}

\section{Short glossary  }

$$\begin{array}{lcl}
\text{admissible component} & &   \text{a component admitting a root superalgebra
} \\
\text{contragredient superalgebra} & &  \text{the smallest root Lie superalgebra for an }\\
& & \text{ admissible component; for example, Kac-Moody}\\
& & \text{ (super)algebras and Borcherds algebras of finite rank} \\
\text{half-baked algebra} & &  \ref{sss:half}\\
\text{fully reflectable} & &  \ref{ss:fully}, \ 
\text{these are the components corresponding to }\\
& & \ \ \text{the Kac-Moody (super)algebras} \\
\text{reflexion isotropic/anisotropic}& &  \ref{sss:reflexionformulas}\\
\text{reflection}& &  \ref{sss:reflection}\\ 
\text{root Lie superalgebra} & &  \ref{ss:rootalgebra}\\
\text{roots: isotropic, anisotropic} & & \ref{ssisoaniso}\\
\ \ \ \ \ \ \ \ \text{ non-reflectable } & & \\
\text{roots: real} & & \ref{dfn:real}\\
\text{roots: principal} & & \ref{sss:principal}\\
\text{skeleton}& &  \ref{sss:skeleton}\\
\text{spine}& &  \ref{sss:spine}\\
\text{weakly symmetric} & & \ref{dfn:quasisym}\\
\text{universal root superalgebra}& &   \text{the largest root Lie superalgebra for }\\
& &\ \  \text{ an admissible component}
\end{array}$$

\end{document}